\documentclass{article}

\usepackage{algorithm} 
\usepackage{algorithmic} 
\usepackage{amsmath} 
\usepackage{amssymb} 
\usepackage{amsthm} 
\usepackage{enumerate} 
\usepackage[utf8]{inputenc}
\usepackage{mathrsfs} 
\usepackage{mathtools} 
\usepackage{pgfplots}
\usepackage{tikz}
\usepackage{xcolor} 

\title{Structure-Preserving Linear Quadratic Gaussian Balanced Truncation for Port-Hamiltonian Descriptor Systems}

\theoremstyle{plain} 
\newtheorem{theorem}{Theorem}
\newtheorem{corollary}[theorem]{Corollary}
\newtheorem{definition}[theorem]{Definition} 
\newtheorem{example}[theorem]{Example} 
\newtheorem{lemma}[theorem]{Lemma}
\newtheorem{proposition}[theorem]{Proposition}
\newtheorem{remark}[theorem]{Remark}

\newcommand{\Ac}{A_{\mathrm{c}}}
\newcommand{\Atilde}{\widetilde{A}}
\newcommand{\Bc}{B_{\mathrm{c}}}
\newcommand{\Btilde}{\widetilde{B}}
\newcommand{\BwideHat}{\widehat{B}}
\newcommand{\calM}{\mathcal{M}}
\newcommand{\calW}{\mathcal{W}}
\newcommand{\CC}{\mathbb{C}}
\newcommand{\Cc}{C_{\mathrm{c}}}
\newcommand{\colspan}{\mathrm{im}}
\newcommand{\Ctilde}{\widetilde{C}}
\newcommand{\diag}{\mathrm{diag}}
\newcommand{\Ec}{E_{\mathrm{c}}}
\newcommand{\Etilde}{\widetilde{E}}
\newcommand{\GL}[1]{\mathrm{GL}(#1)}
\newcommand{\Jc}{\widetilde{J}_{\mathrm{c}}}
\newcommand{\Jf}{\widetilde{J}_{\mathrm{f}}}
\newcommand{\Jhat}{\widehat{J}}
\newcommand{\JwideHat}{\widehat{J}}
\newcommand{\Ker}{\mathrm{ker}}
\newcommand{\Ltilde}{\widetilde{\mathcal{L}}}
\newcommand{\matlab}{MATLAB\textsuperscript{\textregistered}}
\newcommand{\NwideHat}{\widehat{N}}
\newcommand{\Pc}{\mathcal{P}_{\mathrm{c}}}
\newcommand{\Pchat}{\widehat{\mathcal{P}}_{\mathrm{c}}}
\newcommand{\Pctilde}{\widetilde{\mathcal{P}}_{\mathrm{c}}}
\newcommand{\Pf}{\mathcal{P}_{\mathrm{f}}}
\newcommand{\Pfhat}{\widehat{\mathcal{P}}_{\mathrm{f}}}
\newcommand{\Pftilde}{\widetilde{\mathcal{P}}_{\mathrm{f}}}
\newcommand{\PwideHat}{\widehat{P}}
\newcommand{\Qtilde}{\widetilde{Q}}
\newcommand{\QwideHat}{\widehat{Q}}
\newcommand{\rank}{\mathrm{rank}}
\newcommand{\redA}{A_{\mathrm{r}}}
\newcommand{\redB}{B_{\mathrm{r}}}
\newcommand{\redC}{C_{\mathrm{r}}}
\newcommand{\redE}{E_{\mathrm{r}}}
\newcommand{\redG}{G_{\mathrm{r}}}
\newcommand{\redGamma}{\Gamma_{\mathrm{r}}}
\newcommand{\redJ}{J_{\mathrm{r}}}
\newcommand{\redM}{M_{\mathrm{r}}}
\newcommand{\redN}{N_{\mathrm{r}}}
\newcommand{\redPc}{\mathcal{P}_{\mathrm{cr}}}
\newcommand{\redPf}{\mathcal{P}_{\mathrm{fr}}}
\newcommand{\redQ}{Q_{\mathrm{r}}}
\newcommand{\redR}{R_{\mathrm{r}}}
\newcommand{\redTheta}{\Theta_{\mathrm{r}}}

\newcommand{\Rhat}{\widehat{R}}
\newcommand{\RR}{\mathbb{R}}
\newcommand{\Rtilde}{\widetilde{R}}
\newcommand{\RwideHat}{\widehat{R}}
\newcommand{\Sc}{S_{\mathrm{c}}}
\newcommand{\scrA}{\mathscr{A}}
\newcommand{\scrE}{\mathscr{E}}
\newcommand{\Sf}{S_{\mathrm{f}}}
\newcommand{\spectrum}{\sigma}
\newcommand{\SwideHat}{\widehat{S}}
\newcommand{\Tc}{T_{\mathrm{c}}}
\newcommand{\Tf}{T_{\mathrm{f}}}
\newcommand{\trunc}{K_{\ell}}
\newcommand{\uc}{u_{\mathrm{c}}}
\newcommand{\uF}{u_{\mathrm{F}}}
\newcommand{\Vsystilde}{\widetilde{\mathcal{V}}_{\mathrm{sys}}}
\newcommand{\Xhat}{\widehat{X}}
\newcommand{\Xmax}{X_{\mathrm{max}}}
\newcommand{\Xtilde}{\widetilde{X}}
\newcommand{\xwideHat}{\widehat{x}}
\newcommand{\yc}{y_{\mathrm{c}}}

\newcommand{\APc}{A_{\Pc}}
\newcommand{\APctilde}{\Atilde_{\Pc}}

\newcommand{\APftilde}{\Atilde_{\Pf}}

\newcommand{\CPctilde}{\Ctilde_{\Pc}}

\newcommand{\DPctilde}{\widetilde{D}}
\newcommand{\LPc}{L_{\Pc}}
\newcommand{\LPf}{L_{\Pf}}

\allowdisplaybreaks 

\begin{document}

\maketitle

 \centerline{\scshape Tobias Breiten} 
 \medskip
 {\footnotesize
  \centerline{Institute of Mathematics}
    \centerline{Technische Universit\"at Berlin} 
    \centerline{Stra\ss e des 17. Juni 136, 10623 Berlin, Germany}
    \centerline{tobias.breiten@tu-berlin.de}
 } 

 \medskip

 \medskip
  \centerline{\scshape Philipp Schulze}
 \medskip
 {\footnotesize
  \centerline{Institute of Mathematics}
    \centerline{Technische Universit\"at Berlin}
    \centerline{Stra\ss e des 17. Juni 136, 10623 Berlin, Germany}
    \centerline{pschulze@math.tu-berlin.de}
 } 
 
\bigskip

\begin{abstract}
    We present a new balancing-based structure-preserving model reduction technique for linear port-Hamiltonian descriptor systems. The proposed method relies on a modification of a set of two dual generalized algebraic Riccati equations that arise in the context of linear quadratic Gaussian balanced truncation for differential algebraic systems. We derive an a priori error bound with respect to a right coprime factorization of the underlying transfer function thereby allowing for an estimate with respect to the gap metric. We further theoretically and numerically analyze the influence of the Hamiltonian and a change thereof, respectively. With regard to this change of the Hamiltonian, we provide a novel procedure that is based on a recently introduced Kalman--Yakubovich--Popov inequality for descriptor systems. Numerical examples demonstrate how the quality of reduced-order models can significantly be improved by first computing an extremal solution to this inequality.
\end{abstract}


\section{Introduction}

Modeling of complex physical systems often leads to large-scale systems of differential-algebraic equations (DAEs).
Among others, the algebraic constraints may arise from coupling conditions within a network, conservation laws, or constitutive relations, see for instance \cite{KunM06} for an overview of DAE systems.
Performing real-time control or optimization based on large scale DAE systems is often very expensive or even unfeasible which motivates model order reduction (MOR) techniques which aim for replacing the original full-order model (FOM) by a low-dimensional surrogate model.
For a general overview over model reduction techniques, we refer to \cite{Ant05,AntBG20,BenOCW17,HesRS16,QuaMN16,QuaR14} and to \cite{BenS17} for a survey about MOR methods for DAE systems.

A surrogate model obtained via model order reduction is in general required to allow a fast evaluation while still achieving an accurate approximation of the input-output map.
When considering descriptor systems, i.e., control systems with algebraic constraints, it is usually also important to preserve the algebraic constraints in the reduced-order model, since they might be otherwise violated and this may lead to unphysical behavior.
Besides, the FOM often has further important properties like stability and passivity and in this case it is desirable to preserve these properties as well.
In this context, structure-preserving model order reduction for port-Hamiltonian (pH) systems is a frequently used approach to ensure a stable and passive reduced-order model, see e.g.~\cite{BreMS21,BreU21,ChaBG16,Fuj08,GifWPL14,GugPBS12,PolS12} for the case without algebraic constraints and \cite{BeaGM19,EggKLMM18,HauMM19} for the DAE case.
The Port-Hamiltonian structure is particularly useful for model reduction since it ensures passivity and under certain assumptions for the Hamiltonian also stability, see for instance \cite[sec.~8]{SchJ14}.
Another important property is that the port-Hamiltonian structure is invariant under power-preserving interconnection, which makes the pH framework especially well-suited for network modeling, multiphysics problems, and control, see \cite{SchJ14} for a general overview over the topic.

The aforementioned literature on structure-preserving pH DAE model reduction is mainly of interpolatory nature and relies on adjusting the projection matrices by (implicitly) performing the effort-constraint reduction framework from \cite{PolS12}. 
In doing so, the methods not only give up some of the available degrees of freedom but also lack an a priori error bound. 
In this contribution, we therefore rather focus on a linear quadratic Gaussian (LQG) balancing and truncation technique which is inspired by the methods from \cite{BreMS21,MoeRS11}. In particular, based on the method from \cite{BreMS21} and several ideas and proofs from \cite{MoeRS11}, we suggest a new set of generalized algebraic Riccati equations  \eqref{eq:modifiedGAREs} for which we establish theoretical existence results as stated in Theorem \ref{thm:solutionsOfNonsymmetricGAREs} and Theorem \ref{thm:bal_real}. One of our main results is concerned with the construction of a reduced-order port-Hamiltonian descriptor system, see Theorem \ref{thm:rom_is_ph}. We further show that the error (in terms of coprime factors) for such reduced models can be bounded by means of truncated singular values arising in a specifically balanced realization, see Theorem \ref{thm:err_bnd}. 
Finally, in Theorem \ref{thm:optimalpHRealizationOfWCF} we provide a novel technique that allows to exchange a given system Hamiltonian by a different one thereby extending available results for ordinary differential equations (ODEs) to descriptor systems. 
Let us emphasize that for the latter technique, we do not make any explicit assumptions on the index of the system.

The structure of the paper is as follows. In section 2, we collect the necessary linear algebraic and control theoretic background on matrix pencils, descriptor systems, port-Hamiltonian systems, and descriptor Kalman--Yakubovich--Popov linear matrix inequalities (KYP LMIs). 
Section 3 presents the solution theory for a set of modified generalized algebraic Riccati equations. 
Subsequently, in section 4, we adapt the LQG balancing and truncation framework from \cite{MoeRS11} to linear port-Hamiltonian descriptor systems and particularly derive an a priori error bound in terms of right coprime factorizations of the involved transfer functions. 
We further propose a novel method to exchange a given system Hamiltonian by means of solving a KYP LMI for descriptor systems. 
In particular, we provide a theoretical analysis of optimizing the a priori error bound by using an extremal solution of this KYP LMI. 
Numerical experiments for two differential-algebraic port-Hamiltonian systems are shown in section 5.
Finally, we give a brief summary and outlook in section~\ref{sec:conclusion}.
\\

\noindent \textbf{Notation}\\[-2ex]

\noindent We denote the open right and the open left half-plane by $\CC_+$ and $\CC_-$, respectively. 
Furthermore, the sets of invertible and orthogonal $n\times n$ matrices are denoted by $\GL{n}$ and $\mathrm{O}(n)$, respectively.
Besides, the conjugate transpose of a matrix $A$ is denoted with $A^*$, its Moore--Penrose inverse by $A^+$, its spectrum by $\spectrum(A)$, its kernel by $\Ker(A)$, and its column span or image by $\colspan(A)$.

\section{Preliminaries}

This section collects several well-known results about matrix pencils, linear descriptor and port-Hamiltonian systems that are relevant for the subsequent theoretical and numerical results. 

\subsection{Matrix Pencils}\label{subsec:matrixPencils}

The matrix pencil $(E,A)\in\RR^{n\times n}\times \RR^{n\times n}$ is called \emph{regular} if there is an $s\in\CC$ with $\det(sE-A)\neq 0$.
We call $\lambda\in\CC$ a \emph{finite eigenvalue} of such a regular matrix pencil if there exists a vector $v\in \CC^n\setminus\lbrace 0\rbrace$ such that
\begin{equation*}
    \lambda Ev = Av.
\end{equation*}
Furthermore, we call two matrix pencils $(E_1,A_1),(E_2,A_2)\in\RR^{n\times n}\times \RR^{n\times n}$ \emph{strictly equivalent} if there exist $P,\widetilde{P}\in \GL{n}$ with
\begin{equation*}
    E_2 = P^{-1}E_1\widetilde{P} \quad \text{and} \quad A_2 = P^{-1}A_1\widetilde{P}.
\end{equation*}
Especially, it is straightforward to show that strict equivalence is an equivalence relation and that strictly equivalent matrix pencils have the same finite eigenvalues.

Every regular pencil $(E,A)$ is strictly equivalent to a pencil in Weierstraß canonical form, i.e., there exist $P,\widetilde{P}\in \GL{n}$ such that
\begin{equation}
    \label{eq:WCF}
    P^{-1}E\widetilde{P} = 
    \begin{bmatrix}
        I & 0\\
        0 & N
    \end{bmatrix}
    ,\quad P^{-1}A\widetilde{P} = 
    \begin{bmatrix}
        \widetilde{J}   & 0\\
        0           & I
    \end{bmatrix}
    ,
\end{equation}
where $\widetilde{J}$ and $N$ are in (real) Jordan canonical form and $N$ is nilpotent, see for instance \cite[Ch.~XII,\;\S~2]{Gan59}.
In particular, the eigenvalues of $\widetilde{J}$ are the finite eigenvalues of the pencil $(E,A)$, whereas the second block row corresponds to infinite eigenvalues.
The regular pair $(E,A)$ is called \emph{impulse-free} if $N=0$ holds.
A consequence of the Weierstraß canonical form is that the regular pencil $(E^\top,A^\top)$ is strictly equivalent to $(E,A)$, similarly as the transpose of a matrix is similar to the matrix itself.

\begin{definition}
    A subspace $\mathcal{V}\subset \RR^n$ is called \emph{deflating subspace} of the matrix pencil $(E,A)\in\RR^{n\times n}\times \RR^{n\times n}$ if for a full rank matrix $V\in\RR^{n\times k}$ with $\colspan(V)=\mathcal{V}$ there exist $\ell\leq k$, $U\in\RR^{n\times\ell}$, and $\widehat{E},\widehat{A}\in\RR^{\ell\times k}$ with $EV = U\widehat{E}$ and  $AV = U\widehat{A}$.
    If additionally $(E,A)$ is regular and if there exists $\Jhat\in \RR^{k\times k}$ with $AV = EV\Jhat$, where $\Jhat$ is in Jordan canonical form and consists of all Jordan blocks of $(E,A)$ corresponding to the finite eigenvalues in $\CC_-$, then we call $\mathcal{V}$ the \emph{stable deflating subspace} of the matrix pencil $(E,A)$.
\end{definition}

\subsection{Linear Descriptor Systems}

We introduce some basic definitions for linear time-invariant systems of the form
\begin{equation}\label{eq:dae}
    \begin{aligned} 
        E\dot{x}    &= Ax+ Bu,\quad x(0)=0,\\ 
        y           &= Cx,
    \end{aligned}
\end{equation}
with $(E,A,B,C)\in \RR^{n\times n}\times \RR^{n\times n}\times \RR^{n \times m}\times \RR^{p \times n}$.

\begin{definition}
    \label{def:controllabilityNotionsForDescriptorSystems}
    Let $(E,A,B,C)\in \RR^{n\times n}\times \RR^{n\times n}\times \RR^{n \times m}\times \RR^{p \times n}$ be given with regular $(E,A)$ and $r \vcentcolon= \rank(E)$.
    \begin{enumerate}[(i)]
        \item We call the triple $(E,A,B)$ \emph{impulse controllable}, if
            \begin{equation*}
                \rank\left(
                \begin{bmatrix}
                    E & AS & B
                \end{bmatrix}
                \right) = n,
            \end{equation*}
            where $S\in\RR^{n\times (n-r)}$ satisfies $\colspan(S) = \Ker(E)$.
        \item We call the triple $(E,A,B)$ \emph{strongly stabilizable} if it is  impulse controllable and satisfies
            \begin{equation}
                \label{eq:behavioralStabilizability}
                \rank\left(
                \begin{bmatrix}
                    \lambda E-A & B
                \end{bmatrix}
                \right) = n\quad \text{for all }\lambda\in\overline{\CC_+}.
            \end{equation}
        \item We call the triple $(E,A,B)$ \emph{strongly anti-stabilizable} if it is  impulse controllable and satisfies
            \begin{equation*}
                \rank\left(
                \begin{bmatrix}
                    \lambda E-A & B
                \end{bmatrix}
                \right) = n\quad \text{for all }\lambda\in\overline{\CC_-}.
            \end{equation*}
        \item We call the triple $(E,A,B)$ \emph{strongly controllable} if it is strongly stabilizable and strongly anti-stabilizable.
        \item We call the triple $(E, A, C)$ \emph{impulse observable} if $(E^\top, A^\top, C^\top)$ is impulse controllable.
        \item We call the triple $(E, A, C)$ \emph{strongly detectable} if $(E^\top, A^\top, C^\top)$ is strongly stabilizable.
    \end{enumerate}
\end{definition}

\begin{definition}
    We call the descriptor system given by $(E,A,B,C)\in \RR^{n\times n}\times \RR^{n\times n}\times \RR^{n \times m}\times \RR^{p \times n}$ semi-explicit if $E$ has the form
    \begin{equation*}
        E = 
        \begin{bmatrix}
            I_r & 0\\
            0   & 0
        \end{bmatrix}
    \end{equation*}
    with $r = \rank(E)$.
\end{definition}

\subsection{Port-Hamiltonian Descriptor Systems}\label{subsec:pHDAEs}

We consider port-Hamiltonian descriptor systems of the form
\begin{equation}
    \label{eq:pHEQ}
    \begin{aligned} 
        E\dot{x}    &= \underbrace{(J-R)Q}_{:=A}x+ Bu,\quad x(0)=0,\\ 
        y           &= \underbrace{B^\top Q}_{:=C} x ,
    \end{aligned}
\end{equation}
where $E,J,R,Q\in \RR^{n\times n}$, $B\in \RR^{n \times m}$ and $E^\top Q=Q^\top E\succeq 0$, $J^\top = -J$, $R=R^\top \succeq 0$, $r \vcentcolon= \rank(E)$, together with a Hamiltonian function $\mathcal{H}(x) \vcentcolon= \frac12x^\top E^\top Qx$.

One important property of this port-Hamiltonian structure is that it is closed under equivalence transformations. In fact for $S,T\in\GL{n}$ by using the coordinate transformation $\widetilde{x} = Tx$ and by multiplying the first equation of \eqref{eq:pHEQ} from the left by $S$, we obtain the equivalent system
\begin{equation}
    \label{eq:pHEQ_transformed}
    \begin{aligned} 
        \Etilde\dot{\widetilde{x}}    &= (\widetilde{J}-\Rtilde)\Qtilde\widetilde{x}+\Btilde u,\quad \widetilde{x}(0)=0,\\ 
        y                           &= \Btilde^\top \Qtilde \widetilde{x},
    \end{aligned}
\end{equation}
with $\Etilde=SET^{-1}$, $\widetilde{J}=SJS^\top$, $\Rtilde=SRS^\top$, $\Qtilde = S^{-\top}QT^{-1}$, $\Atilde \vcentcolon= (\widetilde{J}-\Rtilde)\Qtilde = SAT^{-1}$, and $\Btilde = SB$.
Especially, 
\begin{equation}
    \label{eq:transformedETQ}
    \Etilde^\top\Qtilde = T^{-\top}E^\top S^\top S^{-\top}QT^{-1} = T^{-\top}E^\top QT^{-1}
\end{equation}
is symmetric and positive semidefinite.
A particularly important class of equivalence transformations is given by 
\begin{equation}
    \label{eq:EquivalenceTransformationToSemiExplicit}
    S = 
    \begin{bmatrix}
        \Sigma_1^{-\frac12} & 0\\
        0                   & S_{22}
    \end{bmatrix}
    U^\top = 
    \begin{bmatrix}
        \Sigma_1^{-\frac12}U_1^\top\\
        S_{22}U_2^\top
    \end{bmatrix}
    ,\quad 
    T = 
    \begin{bmatrix}
        \Sigma_1^{\frac12} & 0\\
        0                   & T_{22}
    \end{bmatrix}
    V^\top = 
    \begin{bmatrix}
        \Sigma_1^{\frac12}V_1^\top\\
        T_{22}V_2^\top
    \end{bmatrix}
    ,
\end{equation}
where $S_{22},T_{22}\in\GL{n-r}$ and $\Sigma_1\in\GL{r}$, $U = [U_1\;\;U_2],V = [V_1\;\;V_2]\in\mathrm{O}(n)$ are obtained from the singular value decomposition of $E$ via
\begin{equation}
    \label{eq:svdOfE}
    E = 
    \begin{bmatrix}
        U_1 & U_2
    \end{bmatrix}
    \begin{bmatrix}
        \Sigma_1    & 0\\
        0           & 0
    \end{bmatrix}
    \begin{bmatrix}
        V_1^\top\\
        V_2^\top
    \end{bmatrix}
    .
\end{equation}
Especially, each equivalence transformation based on transformation matrices of the form \eqref{eq:EquivalenceTransformationToSemiExplicit} yields a semi-explicit port-Hamiltonian system.
 
\begin{remark}\label{rem:pH-structure}
    Note that for $Q=I_n$ in  \eqref{eq:pHEQ}, we obtain the special form
    \begin{equation}
        \label{eq:pH}
        \begin{aligned} 
            E\dot{x}    &= {(J-R)}x+ Bu,\quad x(0)=0,\\ 
            y           &={B^\top} x.
        \end{aligned}
    \end{equation} 
    While such a formulation with $Q=I_n$ is sometimes used in the literature, see for instance \cite{AltMU21,MehV21}, we focus on the more general formulation \eqref{eq:pHEQ}. 
    This is mainly due to the fact that \eqref{eq:pH} is closed under congruence but not under general equivalence transformations. However, the latter transformations are of particular interest for the balancing procedure considered in section~\ref{subsec:pHLQGBT}.
\end{remark}
 
\begin{remark}
    Note that the formulation \eqref{eq:pHEQ} particularly includes port-Hamiltonian systems of the form
    \begin{equation}
        \label{eq:pHWithFeedthrough}
        \begin{aligned}
            \dot{\xwideHat}   &= (\JwideHat-\RwideHat)\QwideHat\xwideHat+(\BwideHat-\PwideHat)u, \quad \xwideHat(0)=0, \\
            y                 &= (\BwideHat+\PwideHat)^\top \QwideHat\xwideHat+(\SwideHat+\NwideHat)u.
        \end{aligned}
    \end{equation}
    with $\QwideHat=\QwideHat^\top\succ 0$ and
    \begin{equation*}
        \begin{bmatrix}
            \JwideHat       & \BwideHat\\
            -\BwideHat^\top & \NwideHat
        \end{bmatrix}
        = -
         \begin{bmatrix}
            \JwideHat       & \BwideHat\\
            -\BwideHat^\top & \NwideHat
        \end{bmatrix}
        ^\top,\quad
        \begin{bmatrix}
            \RwideHat       & \PwideHat\\
            \PwideHat^\top  & \SwideHat
        \end{bmatrix}
        =
        \begin{bmatrix}
            \RwideHat       & \PwideHat\\
            \PwideHat^\top  & \SwideHat
        \end{bmatrix}
        ^\top \succeq 0.
    \end{equation*}
    This class of port-Hamiltonian ODE systems with feedthrough term has, for instance, been considered in \cite{BeaMV19}.
    By introducing the new state variables $x_1\vcentcolon= \xwideHat$, $x_2\vcentcolon= y$, and $x_3 \vcentcolon= u$, straightforward calculations reveal that \eqref{eq:pHWithFeedthrough} is equivalent to a port-Hamiltonian descriptor system of the form \eqref{eq:pHEQ} with
    \begin{align*}
        x &= 
        \begin{bmatrix}
            x_1\\
            x_2\\
            x_3
        \end{bmatrix}
        ,\quad E =
        \begin{bmatrix}
            I & 0 & 0\\
            0 & 0 & 0\\
            0 & 0 & 0
        \end{bmatrix}
        ,\quad J =
        \begin{bmatrix}
            \JwideHat       & \BwideHat     & 0\\
            -\BwideHat^\top & -\NwideHat    & I\\
            0               & -I            & 0
        \end{bmatrix}
        ,\quad R =
        \begin{bmatrix}
            \RwideHat       & \PwideHat & 0\\
            \PwideHat^\top  & \SwideHat & 0\\
            0               & 0         & 0
        \end{bmatrix}
        ,\\
        Q &=
        \begin{bmatrix}
            \QwideHat   & 0 & 0\\
            0           & 0 & I\\
            0           & I & 0
        \end{bmatrix}
        ,\quad B =
        \begin{bmatrix}
            0\\
            0\\
            I
        \end{bmatrix}
        .
    \end{align*}
\end{remark}

For the framework presented in the subsequent sections, we usually consider pH descriptor systems which are strongly stabilizable and detectable.
We emphasize that these assumptions are not very restrictive since every regular pH descriptor system may be reduced to a minimal one which is especially strongly stabilizable and detectable as well as port-Hamiltonian, see sections~4 and 5 in \cite{CheGH23} for the details.
A characterization of strongly stabilizable and detectable pH descriptor systems is provided in the following.

\begin{proposition}
    \label{thm:characterizationOfStronglyStabilizableAndDetectablePHSystems}
	The port-Hamiltonian system \eqref{eq:pHEQ} is regular and strongly stabilizable and detectable if and only if the following conditions are satisfied:
	\begin{enumerate}[(i)]
		\item \label{itm:invertibilityOfQ}$Q$ is invertible,
		\item \label{itm:regularityCondition}$\rank([E\;\; J\;\; R])=n$,
		\item \label{itm:behavioralStabilizabilityCondition}$\rank([i\omega E-JQ\;\; R\;\; B])=n$ for all $\omega\in\RR$, and
		\item \label{itm:impulseControllabilityCondition}$\rank([E\;\; JQV_2\;\; R\;\; B])=n$ with $V_2$ as in \eqref{eq:svdOfE}.
	\end{enumerate}
\end{proposition}

\begin{proof}
	``$\Rightarrow$'': Let the pH system \eqref{eq:pHEQ} be regular and strongly stabilizable and detectable.
	The strong detectability implies that the only vector $v\in\RR^n$ satisfying
	\begin{equation*}
		0 = v^\top A^\top = v^\top Q^\top (J-R)^\top\quad\text{and}\quad0=v^\top C^\top = v^\top Q^\top B
	\end{equation*}
	is the zero vector.
	Therefore, the kernel of $Q$ may only consist of the zero vector, i.e., \eqref{itm:invertibilityOfQ} is satisfied.
	
	Next, we show that the regularity of $(E,A)$ implies \eqref{itm:regularityCondition}.
	If \eqref{itm:regularityCondition} is violated, then there exists $v\in\RR^n\setminus\lbrace 0\rbrace$ with $v^\top E = v^\top J = v^\top R=0$ and, therefore, $(E,(J-R)Q)$ is singular.
	By contraposition, we conclude that the regularity of $(E,A)$ implies \eqref{itm:regularityCondition}.
	Using very similar arguments, one may show that strong stabilizability of $(E,A,B)$ implies the rank conditions \eqref{itm:behavioralStabilizabilityCondition} and \eqref{itm:impulseControllabilityCondition}.
	Therefore, a regular and strongly stabilizable and detectable pH system of the form \eqref{eq:pHEQ} needs to satisfy the conditions \eqref{itm:invertibilityOfQ}--\eqref{itm:impulseControllabilityCondition}.
	
	``$\Leftarrow$'': In the following, let the conditions \eqref{itm:invertibilityOfQ}--\eqref{itm:impulseControllabilityCondition} be satisfied.
	First, we show via contraposition that this implies the regularity of $(E,A)$.
	If $(E,A)$ is singular, then for each $\lambda\in\CC$ there exists $v\in\CC^n\setminus\lbrace 0\rbrace$ with
	\begin{equation*}
		(\lambda E-(J-R)Q)v=0,
	\end{equation*}
	which in turn implies
	\begin{equation}
		\label{eq:regularityConditionProof}
		\lambda v^*Q^\top Ev-v^*Q^\top JQv+v^*Q^\top RQv=0.
	\end{equation}
	Due to $E^\top Q=Q^\top E\succeq 0$, $R=R^\top\succeq 0$, and $J=-J^\top$, we infer that $v^*Q^\top Ev$ and $v^*Q^\top RQv$ are real and non-negative and that $v^*Q^\top JQv$ is purely imaginary.
	Since \eqref{eq:regularityConditionProof} has to also hold for $\lambda\in\CC_+$, we conclude that there exists a non-zero vector $v\in\Ker(E^\top Q)\cap\Ker(RQ)$ satisfying $v^*Q^\top JQv=0$.
	Moreover, this vector satisfies
	\begin{align*}
		Q^\top JQv  &= -(\lambda E^\top Q-Q^\top (J-R)Q)v = -(\lambda Q^\top E-Q^\top (J-R)Q)v\\
                    &= -Q^\top (\lambda E-(J-R)Q)v = 0,
	\end{align*}
	i.e., $v\in \Ker(E^\top Q)\cap\Ker(RQ)\cap \Ker(Q^\top JQ)$.
	Since $v$ is non-zero, using the invertibility of $Q$, this implies
	\begin{align*}
		\rank(
		\begin{bmatrix}
			E & J & R
		\end{bmatrix}				
		) &= \rank\Bigg(
		\begin{bmatrix}
			E^\top\\
			J^\top\\
			R^\top
		\end{bmatrix}
		\Bigg) = \rank\Bigg(
		\begin{bmatrix}
			E^\top\\
			-J\\
			R
		\end{bmatrix}
		\Bigg)\\
        &= \rank\Bigg(
		\begin{bmatrix}
			I_n & & \\
			& -Q^\top & \\
			& & I_n
		\end{bmatrix}
		\begin{bmatrix}
			E^\top\\
			-J\\
			R
		\end{bmatrix}
		Q\Bigg)\\
        &= \rank\Bigg(
		\begin{bmatrix}
			E^\top Q\\
			Q^\top JQ\\
			RQ
		\end{bmatrix}
		\Bigg) <n.
	\end{align*}
	Thus, we have shown that the singularity of $(E,A)$ implies that \eqref{itm:regularityCondition} is violated or, equivalently, that \eqref{itm:regularityCondition} implies the regularity of $(E,A)$.
	
	Next, we show via contraposition that $(E,A,B)$ is impulse controllable.
	If $(E,A,B)$ is not impulse controllable, then there exists $v\in\RR^n\setminus\lbrace 0\rbrace$ with
	\begin{equation*}
		E^\top v = 0,\quad V_2^\top Q^\top (J^\top-R^\top)v=0,\quad B^\top v=0.
	\end{equation*}
	This implies that there exists $\alpha\in\RR^{n-r}\setminus\lbrace 0\rbrace$ with 
    \begin{equation}
        \label{eq:proofOfImpulseControllability_-1}
        V_2^\top Q^\top (J^\top-R^\top)U_2\alpha=0,\quad B^\top U_2\alpha=0.
    \end{equation}
    Exploiting in addition the equality
    \begin{equation}
        \label{eq:proofOfImpulseControllability_0}
        \begin{aligned}
            QV_2    &= (U_1U_1^\top+U_2U_2^\top)QV_2 = EE^+QV_2+U_2U_2^\top QV_2 \\
                    &= (E^+)^\top E^\top QV_2+U_2U_2^\top QV_2 = (E^+)^\top Q^\top EV_2+U_2U_2^\top QV_2 = U_2U_2^\top QV_2,
        \end{aligned}
    \end{equation}
    we obtain
    \begin{equation}
        \label{eq:proofOfImpulseControllability_1}
        0 = V_2^\top Q^\top U_2U_2^\top(J^\top-R^\top)U_2\alpha.
    \end{equation}
    Moreover, the relation
    \begin{equation*}
        U_1^\top QV_2 = \Sigma_1^{-1}V_1^\top V_1\Sigma_1 U_1^\top QV_2 = \Sigma_1^{-1}V_1^\top E^\top QV_2 = \Sigma_1^{-1}V_1^\top Q^\top EV_2 = 0
    \end{equation*}
    implies
    \begin{align*}
        \det(U^\top QV) &= \det\Big(
        \begin{bmatrix}
            U_1^\top QV_1 & U_1^\top QV_2\\
            U_2^\top QV_1 & U_2^\top QV_2
        \end{bmatrix}
        \Big) = \det\Big(
        \begin{bmatrix}
            U_1^\top QV_1 & 0\\
            U_2^\top QV_1 & U_2^\top QV_2
        \end{bmatrix}
        \Big)\\
        &= \det(U_1^\top QV_1)\det(U_2^\top QV_2).
    \end{align*}
    Therefore, $U_2^\top QV_2$ is invertible due to the invertiblity of $Q$.
    Hence, \eqref{eq:proofOfImpulseControllability_1} implies 
    \begin{equation}
        \label{eq:proofOfImpulseControllability_2}
        U_2^\top(J^\top-R^\top)U_2\alpha = 0.
    \end{equation}
    Multiplying this equation from the left by $\alpha^\top$ yields $\alpha^\top U_2^\top RU_2\alpha=0$ and thus $RU_2\alpha=0$.
    Combining this with \eqref{eq:proofOfImpulseControllability_2} and the second equation in \eqref{eq:proofOfImpulseControllability_-1} yields $\alpha\in \Ker(B^\top U_2)\cap\Ker(RU_2)\cap\Ker(U_2^\top JU_2)$.
    We conclude that $v=U_2\alpha$ satisfies 
    \begin{equation*}
        E^\top v = 0,\quad B^\top v=0,\quad Rv=0,\quad U_2^\top Jv=0,
    \end{equation*}
    which implies together with \eqref{eq:proofOfImpulseControllability_0} the inequality
    \begin{align*}
        &\rank(
        \begin{bmatrix}
            E & JQV_2 & R & B
        \end{bmatrix}
        ) = \rank(
        \begin{bmatrix}
            E & JU_2U_2^\top QV_2 & R & B
        \end{bmatrix}
        )\\
        &= \rank\Bigg(
        \begin{bmatrix}
            E & JU_2 & R & B
        \end{bmatrix}
        \begin{bmatrix}
            I_n & & & \\
            & U_2^\top QV_2 & & \\
            & & I_n & \\ 
            & & & I_m
        \end{bmatrix}
        \Bigg)\\
        &= \rank(
        \begin{bmatrix}
            E & JU_2 & R & B
        \end{bmatrix}
        )<n,
    \end{align*}
    i.e., the rank condition \eqref{itm:impulseControllabilityCondition} is not satisfied.
    Thus, by contraposition we conclude that $(E,A,B)$ is impulse controllable.

    We proceed to show the impulse observability of $(E,A,C)$.
    If $(E,A,C)$ is not impulse observable, then there exists $w\in\RR^n\setminus\lbrace 0\rbrace$ with
    \begin{equation*}
        Ew=0,\quad U_2^\top(J-R)Qw=0,\quad B^\top Qw=0.
    \end{equation*}
    The first equation implies $E^\top Qw=Q^\top Ew=0$ and, using \eqref{eq:proofOfImpulseControllability_0} and the invertibility of $U_2^\top QV_2$, the second equation implies 
    \begin{equation*}
        V_2^\top Q^\top (J-R)Qw = V_2^\top Q^\top U_2U_2^\top (J-R)Qw = 0.
    \end{equation*}
    Then, introducing $v\vcentcolon= Qw$, the proof of the impulse observability of $(E,A,C)$ follows the same lines as the proof of the impulse controllability of $(E,A,B)$.

    Next, we show by contradiction that \eqref{eq:behavioralStabilizability} is satisfied.
    If \eqref{eq:behavioralStabilizability} was not satisfied, there would exist $\lambda\in\overline{\CC_+}$ and $v\in\CC^n\setminus\lbrace 0\rbrace$ with
    \begin{equation}
        \label{eq:proofOfImpulseControllability_3}
        (\bar{\lambda} E^\top-Q^\top(J^\top-R^\top))v=0,\quad B^\top v=0.
    \end{equation}
    The first equation implies
    \begin{equation}
        \label{eq:proofOfImpulseControllability_4}
        \bar{\lambda} v^*Q^{-\top} E^\top v-v^*J^\top v+v^*R^\top v=0.
    \end{equation}
    Since $Q^{-\top} E^\top = Q^{-\top}(E^\top Q)Q^{-1}$ and $R$ are symmetric and positive semi-definite, $v^*Q^{-\top} E^\top v$ and $v^*R^\top v$ are real and non-negative.
    Moreover, $v^*J^\top v$ is purely imaginary due to the skew-symmetry of $J$.
    Therefore, due to $\lambda\in\overline{\CC_+}$, the real part of \eqref{eq:proofOfImpulseControllability_4} yields $Rv=0$.
    In the case where the real part of $\lambda$ is positive, we obtain in addition $Q^{-\top} E^\top v = 0$, i.e., $v\in\Ker(R)\cap\Ker(E^\top)$.
    Combining this with \eqref{eq:proofOfImpulseControllability_3} yields $v\in\Ker(R)\cap\Ker(E^\top)\cap\Ker(Q^\top J)$, which contradicts the regularity of $(E,(J-R)Q)$.
    Therefore, we must have $\bar\lambda=i\omega$ for some $\omega\in\RR$.
    Hence, $Rv=0$ and \eqref{eq:proofOfImpulseControllability_3} imply $v\in\Ker(i\omega E^\top-Q^\top J^\top)$ and thus
    \begin{equation*}
        \rank(
        \begin{bmatrix}
           i\omega E-JQ & R & B 
        \end{bmatrix}
        ) <n,
    \end{equation*}
    which contradicts condition \eqref{itm:behavioralStabilizabilityCondition}.
    Consequently, we have shown that $(E,A,B)$ must satisfy \eqref{eq:behavioralStabilizability} and is thus strongly stabilizable.

    We conclude the proof by showing
    \begin{equation}
        \label{eq:behavioralDetectability}
        \rank([\lambda E^\top-A^\top\;\; C^\top])=n\quad\text{for all }\lambda\in\overline{\CC_+}
    \end{equation}
    in order to show that $(E,A,C)$ is strongly detectable.
    If this rank condition was not satisfied, there would exist $\lambda\in\overline{\CC_+}$ and $w\in\CC^n\setminus\lbrace 0\rbrace$ with
    \begin{equation}
        \label{eq:proofOfImpulseControllability_5}
        (\bar\lambda E-(J-R)Q)w=0,\quad B^\top Qw=0.
    \end{equation}
    The first equation implies
    \begin{align*}
        (\bar\lambda E^\top-Q^\top(J-R))Qw  &= (\bar\lambda E^\top Q-Q^\top(J-R) Q)w\\
                                        &= (\bar\lambda Q^\top E-Q^\top(J-R)Q)w = Q^\top(\bar\lambda E-(J-R)Q)w\\
                                        &= 0.
    \end{align*}
    Then, introducing $v\vcentcolon= Qw$, the proof of \eqref{eq:behavioralDetectability} follows the same lines as the proof of \eqref{eq:behavioralStabilizability}.
\end{proof}

\subsection{A Generalized KYP LMI}\label{subsec:generalizedKYPLMI}

In this subsection we summarize those results from \cite[ch.~3]{Voi15} which are relevant for this work, especially for the considerations in section~\ref{subsec:optimalQ}.
To this end, we consider a general regular and impulse-controllable control system given by $(E,A,B,C)\in\RR^{n\times n}\times \RR^{n\times n}\times \RR^{n\times m}\times \RR^{p\times n}$ and introduce the associated \emph{system space} via
\begin{equation*}
    \mathcal{V}_{\mathrm{sys}} \vcentcolon= \left\lbrace 
    \begin{bmatrix}
        x\\
        u
    \end{bmatrix}
    \in\RR^{n+m}\mid Ax+Bu\in\colspan(E)
    \right\rbrace.
\end{equation*}
Furthermore, for general matrices $\widehat{A},\widehat{B}\in\RR^{(n+m)\times (n+m)}$ we write $\widehat{A}\succeq_{\mathcal{V}_{\mathrm{sys}}}\widehat{B}$ if $\widehat{A}$ and $\widehat{B}$ satisfy
\begin{equation}
    \label{eq:matrixIneqOnASubspace}
    v^\top \widehat{A}v \geq v^\top \widehat{B} v\quad \text{for all }v\in\mathcal{V}_{\mathrm{sys}}.
\end{equation}
Analogously, we use the notation $\widehat{A}=_{\mathcal{V}_{\mathrm{sys}}}\widehat{B}$ for the case when equality holds in \eqref{eq:matrixIneqOnASubspace}.

In \cite{Voi15}, the author investigated solutions of the generalized KYP LMI
\begin{equation}
    \label{eq:generalizedKYPLMI}
    \begin{bmatrix}
        W_{11}-A^\top X-X^\top A    & W_{12}-X^\top B\\
        W_{12}^\top-B^\top X        & W_{22}
    \end{bmatrix}
    \succeq_{\mathcal{V}_{\mathrm{sys}}} 0,
    \quad E^\top X = X^\top E.
\end{equation}
with weighting matrices $W_{11}=W_{11}^\top\in\RR^{n\times n}$, $W_{12}\in\RR^{n\times m}$, and $W_{22}=W_{22}^\top\in\RR^{m\times m}$.
Actually, in \cite{Voi15} a different sign is used for $A$ and $B$, but for this work it is more reasonable to use the signs as in \eqref{eq:generalizedKYPLMI}.
The following observations demonstrate the relevance of \eqref{eq:generalizedKYPLMI} in the context of port-Hamiltonian descriptor systems.
When we consider the special case of a port-Hamiltonian control system as in \eqref{eq:pHEQ}, we note that $Q$ is a solution of \eqref{eq:generalizedKYPLMI} with weighting matrices $W_{11}=0$, $W_{12} = C^\top$, and $W_{22}=0$, i.e., $Q$ solves
\begin{equation}
    \label{eq:generalizedKYP_II}
    W(X) = 
    \begin{bmatrix}
        -A^\top X-X^\top A & C^\top-X^\top B\\
        C-B^\top X & 0
    \end{bmatrix}
    \succeq_{\mathcal{V}_{\mathrm{sys}}} 0,
    \quad E^\top X = X^\top E.
\end{equation}
In fact, $Q$ solves the matrix inequality in \eqref{eq:generalizedKYP_II} even on the whole space $\RR^{n+m}$, i.e., it is a solution of
\begin{equation}
    \label{eq:generalizedKYP_II_unrestricted}
    W(X) = 
    \begin{bmatrix}
        -A^\top X-X^\top A & C^\top-X^\top B\\
        C-B^\top X & 0
    \end{bmatrix}
    \succeq 0,
    \quad E^\top X = X^\top E.
\end{equation}
This follows from the defining properties of a port-Hamiltonian system, see the beginning of section~\ref{subsec:pHDAEs}.

On the other hand, if there exists $X\in \GL{n}$ which satisfies \eqref{eq:generalizedKYP_II_unrestricted} and $E^\top X\succeq 0$, then the system \eqref{eq:pHEQ} can be equivalently written as
\begin{equation}
    \label{eq:pHEX}
    \begin{aligned}
        E\dot{x}    &= \underbrace{(J_X-R_X)X}_{=A}x+ Bu,\quad x(0)=0,\\ 
        y           &= \underbrace{B^\top X}_{=C} x ,
    \end{aligned}
\end{equation}
with $J_X \vcentcolon= \frac12(AX^{-1}-X^{-\top}A^\top) = -J_X^\top$ and $R_X \vcentcolon= -\frac12(AX^{-1}+X^{-\top}A^\top) = R_X^\top\succeq 0$.
The resulting system has the same port-Hamiltonian structure as \eqref{eq:pHEQ}, but with the Hamiltonian $\mathcal{H}_X(x) \vcentcolon= \frac12x^\top E^\top Xx$.
Furthermore, we note that in the more general case that there exists a possibly non-invertible solution $X$ to \eqref{eq:generalizedKYP_II}, the explicit construction of a port-Hamiltonian formulation as in \eqref{eq:pHEX} using $X^{-1}$ is not possible, but the Hamiltonian $\mathcal{H}_X$ still satisfies the dissipation inequality
\begin{align*}
    \frac{\mathrm{d}\mathcal{H}_X}{\mathrm{d}t} &= x^\top E^\top X\dot{x} = x^\top X^\top E\dot{x} = x^\top X^\top\left(Ax+Bu\right) = x^\top X^\top Ax+x^\top X^\top Bu\\ 
    &= \frac12x^\top\left(X^\top A+A^\top X\right)x+x^\top C^\top u \leq y^\top u.
\end{align*}
Here, we implicitly assumed that there exists an input $u$ which admits a solution $x$ to \eqref{eq:pHEQ} and this is in particular the case if $(E,A)$ is regular, cf.~\cite[Thm.~2.12]{KunM06}.
Moreover, the dissipation inequality with $E^\top X\succeq 0$ implies passivity and, thus, since $Q$ is a solution of \eqref{eq:generalizedKYP_II} with $E^\top Q\succeq 0$, every regular port-Hamiltonian system of the form \eqref{eq:pHEQ} is passive.

In \cite{Voi15}, the author shows that the existence of extremal solutions of \eqref{eq:generalizedKYPLMI} is closely related to solutions $(Y,K,L)\in\RR^{n\times n}\times \RR^{q\times n}\times \RR^{q\times m}$ of the Lur'e equation
\begin{equation}
     \label{eq:Lure}
     \begin{bmatrix}
         W_{11}-A^\top Y-Y^\top A   & W_{12}-Y^\top B\\
         W_{12}^\top-B^\top Y       & W_{22}
     \end{bmatrix}
     =_{\mathcal{V}_{\mathrm{sys}}}
     \begin{bmatrix}
         K^\top\\
         L^\top
     \end{bmatrix}
     \begin{bmatrix}
         K & L
     \end{bmatrix}
     ,\quad E^\top Y=Y^\top E,
\end{equation}
where $(Y,K,L)$ have to additionally satisfy
\begin{equation*}
    \rank_{\RR(s)}\left(
    \begin{bmatrix}
        -sE-A   & -B\\
        K       & L
    \end{bmatrix}
    \right) = n+q.
\end{equation*}
Here, the block matrix in the last equality is regarded as a matrix with entries in $\RR(s)$, which is the field of rational functions with coefficients in $\RR$.

It immediately follows that $Y$ solves \eqref{eq:generalizedKYPLMI} if $(Y,K,L)$ is a solution of the Lur'e equation \eqref{eq:Lure}.
We call a solution of \eqref{eq:Lure} \emph{stabilizing} if 
\begin{equation*}
     \rank\left(
     \begin{bmatrix}
         -\lambda E-A    & -B\\
         K               & L
     \end{bmatrix}
     \right) = n+q \quad \text{for all } \lambda\in\CC_+.
\end{equation*}

The following theorem, which is a special case of Theorem~3.5.3 from \cite{Voi15}, establishes the existence of stabilizing solutions of the Lur'e equation \eqref{eq:Lure}.
\begin{theorem}
     \label{thm:existenceOfStabilizingSolutionsOfTheLureEquation}
     Assume we have given $(E,A,B)\in\RR^{n\times n}\times \RR^{n\times n}\times \RR^{n\times m}$ with regular pencil $(E,A)$ and impulse-controllable triple $(E,A,B)$ as well as the matrices $W_{11}=W_{11}^\top\in\RR^{n\times n}$, $W_{12}\in\RR^{n\times m}$, and $W_{22}=W_{22}^\top\in\RR^{m\times m}$.
     Furthermore, let the associated KYP LMI \eqref{eq:generalizedKYPLMI} be solvable.
     If $(E,A,B)$ is strongly anti-stabilizable, then there exists a stabilizing solution to the Lur'e equation \eqref{eq:Lure}.
\end{theorem}
\begin{proof}
    Using our notation, Theorem~3.5.3 from \cite{Voi15} states the result for the case that $(E,-A,-B)$ is strongly stabilizable.
    From Definition~\ref{def:controllabilityNotionsForDescriptorSystems} it follows that strong stabilizability of $(E,-A,-B)$ is equivalent to strong anti-stabilizability of $(E,A,B)$, which concludes the proof.
\end{proof}

Based on the concept of stabilizing solutions of the Lur'e equation, we can make use of Theorem~3.5.4 from \cite{Voi15} to obtain a maximal solution of \eqref{eq:generalizedKYPLMI}.

\begin{theorem}
     \label{thm:extremalSolutionOfKYPLMI}
     Let $(E,A,B)\in\RR^{n\times n}\times \RR^{n\times n}\times \RR^{n\times m}$ with regular pencil $(E,A)$ and impulse-controllable triple $(E,A,B)$ as well as the matrices $W_{11}=W_{11}^\top\in\RR^{n\times n}$, $W_{12}\in\RR^{n\times m}$, and $W_{22}=W_{22}^\top\in\RR^{m\times m}$ be given.
     Furthermore, assume that the associated KYP LMI \eqref{eq:generalizedKYPLMI} is solvable.
     If there exists additionally a stabilizing solution $(Y,K,L)$ of the Lur'e equation \eqref{eq:Lure}, then $E^\top Y\succeq E^\top X$ holds for all solutions $X$ of \eqref{eq:generalizedKYPLMI}.
\end{theorem}

\section{Modified Generalized Algebraic Riccati Equations}

In \cite{MoeRS11} the authors proposed an LQG balanced truncation approach for unstructured linear time-invariant descriptor systems based on generalized algebraic Riccati equations (GAREs).
For the case that there is no direct feedthrough term, these GAREs read
\begin{subequations}
    \label{eq:originalGAREs}
    \begin{align}
        \label{eq:originalControlGARE}
        A^\top\Pc+\Pc^\top A-\Pc^\top BB^\top\Pc+C^\top C &= 0, \quad E^\top\Pc-\Pc^\top E = 0,\\
        \label{eq:originalFilterGARE}
        A\Pf^\top+\Pf A^\top-\Pf C^\top C\Pf^\top+BB^\top &= 0, \quad E\Pf^\top-\Pf E^\top = 0.
    \end{align}
\end{subequations}
We emphasize that these GAREs are dual to each other in the sense that $\Pf$ is a solution of \eqref{eq:originalFilterGARE} if and only if $\Pf^\top$ is a solution of \eqref{eq:originalControlGARE} associated with the dual system given by $(E^\top,A^\top,C^\top,B^\top)$.
This observation allows to investigate existence and uniqueness of solutions for one of the two equations and to transfer these results to the other equation by duality arguments, see for instance \cite[sec.~3]{MoeRS11}.

To derive a structure-preserving model reduction method for pH systems, we extend the idea presented in \cite{BreMS21} to the DAE case by adding $2R$ to the constant term of \eqref{eq:originalFilterGARE} and we consider the resulting modified GAREs
\begin{subequations}
    \label{eq:modifiedGAREs}
    \begin{align}
        \label{eq:modifiedControlGARE}
        A^\top\Pc+\Pc^\top A-\Pc^\top BB^\top\Pc+C^\top C &= 0, \quad E^\top\Pc-\Pc^\top E = 0,\\
        \label{eq:modifiedFilterGARE}
        A\Pf^\top+\Pf A^\top-\Pf C^\top C\Pf^\top+BB^\top+2R &= 0, \quad E\Pf^\top-\Pf E^\top = 0,
    \end{align}
\end{subequations}
associated to the port-Hamiltonian system \eqref{eq:pHEQ}.

\begin{remark}\label{rem:alter_mod_gare}
  Comparing \eqref{eq:originalGAREs} and \eqref{eq:modifiedGAREs}, we see that only the generalized filter Riccati equation has been changed while the generalized control Riccati equation is left unaltered. 
  Let us emphasize that the following results can similarly be obtained by keeping the filter Riccati equation \eqref{eq:originalFilterGARE} and instead changing the constant term in \eqref{eq:originalControlGARE} from $C^\top C$ to $C^\top C + 2Q^\top RQ$, see \cite[Remark 3]{BreMS21} for a similar discussion in the ODE case.
\end{remark}

We call a solution $\Pc$ of \eqref{eq:modifiedControlGARE} \emph{stabilizing} if the pencil $(E,A-BB^\top\Pc)$ is regular, impulse-free, and all its finite eigenvalues lie in $\CC_-$.
Analogously, we call a solution $\Pf$ of \eqref{eq:modifiedFilterGARE} stabilizing if the pencil $(E,A-\Pf C^\top C)$ is regular, impulse-free, and all its finite eigenvalues lie in $\CC_-$.

\begin{remark}
    \label{rem:relationToMoeRS}
    When ignoring the $2R$ term in \eqref{eq:modifiedFilterGARE}, the GAREs in \eqref{eq:modifiedGAREs} are a special case of those considered in \cite{MoeRS11} by setting the feedthrough term to zero.
    On the other hand, even when including the $2R$ term, the modified filter equation \eqref{eq:modifiedFilterGARE} can be considered as an unmodified one by using
    \begin{equation}
        \label{eq:Bhat}
        \widehat{B} \vcentcolon= 
        \begin{bmatrix}
            B & \sqrt{2}R^{\frac12}
        \end{bmatrix}
        ,
    \end{equation}
    as input matrix, where we make use of the positive semidefiniteness of $R$.
    Moreover, if the triple $(E,A,B)$ is strongly stabilizable, then also the triple $(E,A,\widehat{B})$ is strongly stabilizable.
    This observation allows us to directly transfer some of the results from \cite{MoeRS11} about the existence and uniqueness of solutions of the unmodified GAREs without the $2R$ term to the modified equations in \eqref{eq:modifiedGAREs}.
    Since this transfer of the results usually follows the same arguments, we only carry this out explicitly when deriving Corollary~\ref{cor:semidefinitenessFilter} from Lemma~\ref{lem:semidefinitenessControl}.
    For the statements of Lemma \ref{lem:reductionToARE}, we avoid a detailed proof and instead refer to the respective result from the literature corresponding to the unmodified case without the $2R$ term.
\end{remark}

The following lemma is a special case of Lemma 3{.}4 from \cite{MoeRS11} with $D=0$.

\begin{lemma}
    \label{lem:semidefinitenessControl}
    Consider the system \eqref{eq:pHEQ} with regular $(E,A)$, strongly stabilizable $(E,A,B)$, and strongly detectable $(E,A,C)$.
    Then, a solution $\Pc$ of \eqref{eq:modifiedControlGARE} is stabilizing if and only if $E^\top \Pc$ is positive semidefinite.
\end{lemma}

\begin{corollary}
    \label{cor:semidefinitenessFilter}
    Consider the system \eqref{eq:pHEQ} with regular $(E,A)$, strongly stabilizable $(E,A,B)$, and strongly detectable $(E,A,C)$.
    Then, a solution $\Pf$ of \eqref{eq:modifiedFilterGARE} is stabilizing if and only if $E\Pf^\top$ is positive semidefinite.
\end{corollary}

\begin{proof}
    First, we note that \eqref{eq:modifiedFilterGARE} can be written as 
    \begin{equation}
        \label{eq:modifiedFilterGAREAsUnmodifiedOne}
        A\Pf^\top+\Pf A^\top-\Pf C^\top C\Pf^\top+\widehat{B}\widehat{B}^\top = 0, \quad E\Pf^\top-\Pf E^\top = 0
    \end{equation}
    with $\widehat{B}$ as defined in \eqref{eq:Bhat}.
    By duality we infer that $\Pf$ is a stabilizing solution of \eqref{eq:modifiedFilterGAREAsUnmodifiedOne} if and only if $\Pf^\top$ is a stabilizing solution of \eqref{eq:modifiedControlGARE} associated with the system given by $(E^\top,A^\top,C^\top,\widehat{B}^\top)$.
    The claim then follows from Lemma~\ref{lem:semidefinitenessControl}, since strong stabilizability of $(E,A,B)$ implies strong detectability of $(E^\top,A^\top,\widehat{B}^\top)$, cf.~Remark~\ref{rem:relationToMoeRS}.
\end{proof}

The following lemma establishes a relation between the solutions of the GAREs \eqref{eq:modifiedGAREs} of equivalent systems.
It is a straightforward extension of Lemma~3.5 from \cite{MoeRS11} and follows from simple calculations which are omitted here.

\begin{lemma}
    \label{lem:transformedGramians}
    Consider the equivalent port-Hamiltonian realizations \eqref{eq:pHEQ} and \eqref{eq:pHEQ_transformed} as well as the respective GAREs of the form \eqref{eq:modifiedGAREs}.
    Then, $\Pc$ and $\Pf$ are stabilizing solutions of the GAREs associated to \eqref{eq:pHEQ} if and only if $S^{-\top}\Pc T^{-1}$ and $S\Pf T^\top$ are stabilizing solutions of the GAREs associated to \eqref{eq:pHEQ_transformed}.
\end{lemma}

An important ingredient for investigating the existence of solutions to the modified filter GARE \eqref{eq:modifiedFilterGARE} is the matrix pencil
\begin{equation}
    \label{eq:evenPencil}
    H_{\Pf} \vcentcolon= \Bigg(\underbrace{
    \begin{bmatrix}
        0  & -E^\top\\
        E  & 0
    \end{bmatrix}
    }_{=\vcentcolon \mathcal{E}}, \underbrace{
    \begin{bmatrix}
        -C^\top C   & -A^\top\\
        -A          & BB^\top+2R
    \end{bmatrix}
    }_{=\vcentcolon \mathcal{A}}\Bigg).
\end{equation}
Similarly as in \cite{MoeRS11} we observe that there is a relation between deflating subspaces of $H_{\Pf}$ and \eqref{eq:modifiedFilterGARE} in the sense that \eqref{eq:modifiedFilterGARE} is equivalent to
\begin{align*}
    \mathcal{A}
    \begin{bmatrix}
        \Pf^\top\\
        -I_n
    \end{bmatrix}
    = 
    \begin{bmatrix}
        I_n\\
        \Pf
    \end{bmatrix}
    \left(A^\top-C^\top C\Pf^\top\right), \quad
    \mathcal{E}
    \begin{bmatrix}
        \Pf^\top\\
        -I_n
    \end{bmatrix}
    = 
    \begin{bmatrix}
        I_n\\
        \Pf
    \end{bmatrix}
    E^\top.
\end{align*}

\begin{lemma}
    \label{lem:reductionToARE}
    Let $(E,A,B,C)$ be a semi-explicit realization of \eqref{eq:pHEQ} and let $(E,A)$ be regular, $(E,A,B)$ be strongly stabilizable, and $(E,A,C)$ be strongly detectable.
    Let $V_{11}, V_{21}\in\RR^{r\times r}$ and $V_{12}, V_{22}\in\RR^{(n-r)\times r}$ be such that
    \begin{equation*}
        \mathcal{V} \vcentcolon=
        \colspan\left(
        \begin{bmatrix}
            V_{11}\\
            V_{12}\\
            V_{21}\\
            V_{22}
        \end{bmatrix}
        \right)
    \end{equation*}
    is the stable deflating subspace of the pencil $H_{\Pf}$ defined in \eqref{eq:evenPencil}.
    Then, the following assertions hold.
    \begin{enumerate}[(i)]
        \item The matrix $V_{21}$ is invertible.
        \item The matrix $\Pf\in\RR^{n\times n}$ is a stabilizing solution of \eqref{eq:modifiedFilterGARE} if and only if it is given by
            \begin{equation*}
                \Pf = 
                \begin{bmatrix}
                    V_{11} V_{21}^{-1}               & 0\\
                    (V_{12}-\mathcal{P}_0^\top V_{22})V_{21}^{-1} & \mathcal{P}_0^\top
                \end{bmatrix}
                ^\top,
            \end{equation*}
            where $\mathcal{P}_0\in\RR^{(n-r)\times(n-r)}$ is a solution of the algebraic Riccati equation
            \begin{equation*}
                A_{22}\mathcal{P}_0^\top+\mathcal{P}_0A_{22}^\top+B_2B_2^\top+2R_{22}-\mathcal{P}_0C_2^\top C_2\mathcal{P}_0^\top=0,
            \end{equation*}
            where we used the block partition
            \begin{equation*}
                A = 
                \begin{bmatrix}
                    A_{11} & A_{12}\\
                    A_{21} & A_{22}
                \end{bmatrix}
                ,\quad R = 
                \begin{bmatrix}
                    R_{11} & R_{12}\\
                    R_{21} & R_{22}
                \end{bmatrix}
                ,\quad B = 
                \begin{bmatrix}
                    B_1\\
                    B_2
                \end{bmatrix}
                ,\quad C = 
                \begin{bmatrix}
                    C_1 & C_2
                \end{bmatrix}
            \end{equation*}
            with $A_{22},R_{22}\in\RR^{(n-r)\times(n-r)}$, $B_2\in\RR^{(n-r)\times m}$, $C_2\in\RR^{m\times(n-r)}$ and the sizes of the other blocks are accordingly.
    \end{enumerate}
\end{lemma}

\begin{proof}
    The claim follows directly from Lemma~3.10 from \cite{MoeRS11}, by duality and by using the extended input matrix $\widehat{B}$ as in the proof of Corollary~\ref{cor:semidefinitenessFilter}.
\end{proof}

\begin{theorem}
    \label{thm:solutionsOfNonsymmetricGAREs}
    Consider the system \eqref{eq:pHEQ} with regular pair $(E,A)$, strongly stabilizable triple $(E,A,B)$, and strongly detectable triple $(E,A,C)$.
    Then, the following assertions are true.
    \begin{enumerate}[(i)]
        \item There exist stabilizing solutions of the GAREs \eqref{eq:modifiedGAREs}.
        \item Any two stabilizing solutions $\mathcal{P}_{\mathrm{c},1},\mathcal{P}_{\mathrm{c},2}$ of \eqref{eq:modifiedControlGARE} and any two stabilizing solutions $\mathcal{P}_{\mathrm{f},1},\mathcal{P}_{\mathrm{f},2}$ of \eqref{eq:modifiedFilterGARE} satisfy
            \begin{equation*}
                E^\top\mathcal{P}_{\mathrm{c},1} = E^\top\mathcal{P}_{\mathrm{c},2} \quad\text{and}\quad E\mathcal{P}_{\mathrm{f},1}^\top = E\mathcal{P}_{\mathrm{f},2}^\top,
            \end{equation*}
            respectively.
        \item A particular stabilizing solution of \eqref{eq:modifiedFilterGARE} is given by $Q^{-\top}$.
        \item There exist stabilizing solutions $\Pc$ and $\Pf$ of \eqref{eq:modifiedGAREs} such that $I_n+\Pf\Pc^{\top}$ is invertible.
    \end{enumerate}
\end{theorem}

\begin{proof}
    \begin{enumerate}[(i)]
        \item This statement immediately follows from (iv).
        \item The statement for the control GARE \eqref{eq:modifiedControlGARE} is a special case of Lemma~2 from \cite{KawTK99}.
            The statement for the modified filter GARE \eqref{eq:modifiedFilterGARE} follows then by duality using similar arguments as in the proof of Corollary~\ref{cor:semidefinitenessFilter}.
        \item First, we note that the invertibility of $Q$ follows from Proposition~\ref{thm:characterizationOfStronglyStabilizableAndDetectablePHSystems}.
            The fact that $Q^{-\top}$ solves the first matrix equation in \eqref{eq:modifiedFilterGARE} can be comprehended by the calculation
            \begin{align*}
                    &AQ^{-1}+Q^{-\top}A^\top-Q^{-\top}C^\top CQ^{-1}+BB^\top+2R \\
                    &= J-R+J^\top-R^\top-B B^\top+BB^\top+2R = 0.
            \end{align*}
            For the second matrix equality in \eqref{eq:modifiedFilterGARE}, we note that the symmetry of $E^\top Q$ implies the symmetry of $Q^{-\top}E^\top$ which immediately yields that $Q^{-\top}$ satisfies the second matrix equality in \eqref{eq:modifiedFilterGARE}.
            Analogously, the positive semidefiniteness of $E^\top Q$ implies the positive semidefiniteness of $Q^{-\top}E^\top=EQ^{-1}$ and, thus, we infer by Corollary~\ref{cor:semidefinitenessFilter} that $Q^{-\top}$ is a stabilizing solution of \eqref{eq:modifiedFilterGARE}.
        \item The proof follows along the lines of the proof of Theorem~3.2 from \cite{MoeRS11} with only slight adaptations to the port-Hamiltonian setting.
            By Lemma~\ref{lem:transformedGramians}, $\Pc$ and $\Pf$ are stabilizing solutions of the modified GAREs \eqref{eq:modifiedGAREs} associated with $(E,A,B,C)$ if and only if $\Pctilde \vcentcolon= S^{-\top}\Pc T^{-1}$ and $\Pftilde \vcentcolon= S\Pf T^\top$ are stabilizing solutions of the modified GAREs \eqref{eq:modifiedGAREs} associated with the transformed system \eqref{eq:pHEQ_transformed}.
            Especially, we have
            \begin{equation*}
                I_n+\Pf\Pc^{\top} = S^{-1}\left(I_n+\Pftilde\Pctilde^\top\right)S
            \end{equation*}
            and, consequently, it is sufficient to investigate the invertibility of $I_n+\Pftilde\Pctilde^\top$, where $S$ and $T$ are chosen such that the transformed system \eqref{eq:pHEQ_transformed} is in semi-explicit form.
            Furthermore, we can choose without loss of generality $S$ and $T$ such that the lower right block $\Atilde_{22}$ of $\Atilde$ is symmetric and negative semidefinite, for instance, by choosing $S_{22}=-V_2^\top A^\top U_2$ and $T_{22}=I_{n-r}$ in \eqref{eq:EquivalenceTransformationToSemiExplicit}.
            
            Let $V_{11\mathrm{c}},V_{21\mathrm{c}},V_{11\mathrm{f}},V_{21\mathrm{f}}\in\RR^{r\times r}$ and $V_{12\mathrm{c}},V_{22\mathrm{c}},V_{12\mathrm{f}},V_{22\mathrm{f}}\in\RR^{(n-r)\times r}$ be such that
            \begin{equation*}
                \mathcal{V}_{\mathrm{c}} \vcentcolon=
                \colspan\left(
                \begin{bmatrix}
                    V_{11\mathrm{c}}\\
                    V_{12\mathrm{c}}\\
                    V_{21\mathrm{c}}\\
                    V_{22\mathrm{c}}
                \end{bmatrix}
                \right)\quad\text{and}\quad
                \mathcal{V}_{\mathrm{f}} \vcentcolon=
                \colspan\left(
                \begin{bmatrix}
                    V_{11\mathrm{f}}\\
                    V_{12\mathrm{f}}\\
                    V_{21\mathrm{f}}\\
                    V_{22\mathrm{f}}
                \end{bmatrix}
                \right)
            \end{equation*}
            are the stable deflating subspaces of the pencils
            \begin{align*}
                &H_{\Pctilde} \vcentcolon= \left(
                \begin{bmatrix}
                    0               & -\Etilde\\
                    \Etilde^\top  & 0
                \end{bmatrix}
                ,
                \begin{bmatrix}
                    -\Btilde\Btilde^\top & -\Atilde\\
                    -\Atilde^\top          & \Ctilde^\top \Ctilde
                \end{bmatrix}
                \right)\\
                \text{and}\quad &H_{\Pftilde} \vcentcolon= \left(
                \begin{bmatrix}
                    0           & -\Etilde^\top\\
                    \Etilde   & 0
                \end{bmatrix}
                ,
                \begin{bmatrix}
                    -\Ctilde^\top \Ctilde   & -\Atilde^\top\\
                    -\Atilde^\top             & \Btilde\Btilde^\top+2\Rtilde
                \end{bmatrix}
                \right),
            \end{align*}
            respectively.
            Then, by Lemma~\ref{lem:reductionToARE} and by Lemma~3.10 from \cite{MoeRS11}, $\Pctilde$ and $\Pftilde$ are stabilizing solutions of the GAREs \eqref{eq:modifiedGAREs} associated with $(\Etilde,\Atilde,\Btilde,\Ctilde)$ if and only if they are of the form
            \begin{equation*}
                \Pctilde = 
                \begin{bmatrix}
                    V_{11\mathrm{c}} V_{21\mathrm{c}}^{-1}               & 0\\
                    (V_{12\mathrm{c}}-\mathcal{P}_{0\mathrm{c}}V_{22\mathrm{c}})V_{21\mathrm{c}}^{-1} & \mathcal{P}_{0\mathrm{c}}
                \end{bmatrix}
                ,\;
                \Pftilde = 
                \begin{bmatrix}
                    V_{11\mathrm{f}} V_{21\mathrm{f}}^{-1}               & 0\\
                    (V_{12\mathrm{f}}-\mathcal{P}_{0\mathrm{f}}^\top V_{22\mathrm{f}})V_{21\mathrm{f}}^{-1} & \mathcal{P}_{0\mathrm{f}}^\top
                \end{bmatrix}
                ^\top,
            \end{equation*}
            where $\mathcal{P}_{0\mathrm{c}}$ and $\mathcal{P}_{0\mathrm{f}}$ solve 
            \begin{equation}
                \label{eq:helperAREs}
                \begin{aligned}
                    \Atilde_{22}^\top\mathcal{P}_{0\mathrm{c}}+\mathcal{P}_{0\mathrm{c}}^\top\Atilde_{22}+\Ctilde_2^\top\Ctilde_2-\mathcal{P}_{0\mathrm{c}}^\top\Btilde_2 \Btilde_2^\top\mathcal{P}_{0\mathrm{c}}=0,\\
                    \Atilde_{22}\mathcal{P}_{0\mathrm{f}}^\top+\mathcal{P}_{0\mathrm{f}}\Atilde_{22}^\top+\Btilde_2\Btilde_2^\top+2\Rtilde_{22}-\mathcal{P}_{0\mathrm{f}}\Ctilde_2^\top \Ctilde_2\mathcal{P}_{0\mathrm{f}}^\top=0,
                \end{aligned}
            \end{equation}
            respectively.
            To assess the invertibility of
            \begin{equation*}
                I_n+\Pftilde\Pctilde^\top = 
                \begin{bmatrix}
                    I_r+\left(V_{11\mathrm{c}} V_{21\mathrm{c}}^{-1}V_{11\mathrm{f}} V_{21\mathrm{f}}^{-1}\right)^\top & *\\
                    0 & I_{n-r}+\mathcal{P}_{0\mathrm{f}}\mathcal{P}_{0\mathrm{c}}^\top
                \end{bmatrix}
                ,
            \end{equation*}
            we show that there exist solutions of \eqref{eq:helperAREs} such that the diagonal blocks of $I_n+\Pftilde\Pctilde^\top$ are invertible.
            First, by Lemma~\ref{lem:semidefinitenessControl} and by Corollary~\ref{cor:semidefinitenessFilter}, we infer that $V_{11\mathrm{c}} V_{21\mathrm{c}}^{-1}$ and $V_{11\mathrm{f}} V_{21\mathrm{f}}^{-1}$ are symmetric and positive semidefinite.
            Consequently, all eigenvalues of their product are nonnegative and, thus, $I_r+\left(V_{11\mathrm{c}} V_{21\mathrm{c}}^{-1}V_{11\mathrm{f}} V_{21\mathrm{f}}^{-1}\right)^\top$ is invertible.
            Furthermore, since the triple $(\Etilde,\Atilde,\Btilde)$ is impulse controllable, we have
            due to the semi-explicit structure
            \begin{equation*}
                \rank\left(
                \begin{bmatrix}
                    I_r & 0 & \Atilde_{12} & \Btilde_1\\
                    0   & 0 & \Atilde_{22} & \Btilde_2
                \end{bmatrix}
                \right) = n.
            \end{equation*}
            Hence, we obtain $\rank([\Atilde_{22}\;\; \Btilde_2])=n-r$ and with similar arguments using the impulse observability of $(\Etilde,\Atilde,\Ctilde)$ we get $\rank([\Atilde_{22}^\top\;\; \Ctilde_2^\top])=n-r$.
            Furthermore, since $\Atilde_{22}$ has by construction only real and non-negative eigenvalues, we infer that the pair $(\Atilde_{22},\Btilde_2)$ is stabilizable and the pair $(\Atilde_{22},\Ctilde_2)$ is detectable.
            Finally, using well-known results for standard algebraic Riccati equations, see for instance \cite[Thm.~4.1]{Won68}, and using the extended input matrix $[\Btilde_2\;\; \sqrt{2}\Rtilde_{22}^{\frac12}]$, it follows that the modified algebraic Riccati equations \eqref{eq:helperAREs} have symmetric and positive semidefinite solutions $\mathcal{P}_{0\mathrm{c}}$ and $\mathcal{P}_{0\mathrm{f}}$. 
            Thus, there exist solutions of the modified GAREs \eqref{eq:modifiedGAREs} associated with $(\Etilde,\Atilde,\Btilde,\Ctilde)$ satisfying $I_n+\Pftilde\Pctilde^\top\in\GL{n}$.
    \end{enumerate}
\end{proof}

\begin{remark}\label{rem:I_Q_Pc_inv_not_needed}
  In what follows, the solution $\mathcal{P}_{\mathrm{f}}=Q^{-\top}$ will be of particular relevance for a model reduction approach preserving the pH structure.
  While it is not clear whether there exists a solution $\Pc$ such that $I_n+Q^{-\top} \Pc^\top$ is invertible, we only require the existence of a general pair $(\Pc,\Pf)$ of stabilizing solutions to \eqref{eq:modifiedGAREs} with $I_n+\Pf \Pc^\top\in\GL{n}$. 
  This existence result is first of all essential for proving the error bound in Theorem~\ref{thm:err_bnd}, but it is not part of a computational procedure, cf.~the discussion in \cite[Remark 2]{MoeRS11}.
\end{remark}

A consequence of Theorem~\ref{thm:solutionsOfNonsymmetricGAREs} is that the stabilizing solutions of the modified GAREs \eqref{eq:modifiedGAREs} may be used for constructing a stabilizing controller, which is itself passive.
This statement is specified in the following theorem and extends a similar assertion in the ODE case, cf.~Theorem~1 in \cite{BreMS21}.

\begin{theorem}
  Consider the system \eqref{eq:pHEQ} with regular pair $(E,A)$, strongly stabilizable triple $(E,A,B)$, and strongly detectable triple $(E,A,C)$. 
  If $\Pc$ is a stabilizing solution of \eqref{eq:modifiedControlGARE}, then the controller defined by 
  \begin{align*}
     \Ec=E, \ \ \Ac=A-BB^\top \Pc-Q^{-\top} C^\top C, \ \ \Bc = Q^{-\top}C^\top , \ \ \Cc = B^\top \Pc
  \end{align*}
   is regular and passive.
   Moreover, the associated closed-loop system obtained by $u=-\yc$ and $\uc=y$ is asymptotically stable, regular, and impulse-free.
\end{theorem}

\begin{proof}
    We start by proving the regularity of the pencil $(\Ec,\Ac)$ which we show by contradiction.
    The singularity of $(\Ec,\Ac)$ implies that for any $\lambda\in\RR_{>0}$ there exists $v\in\RR^n\setminus\lbrace 0\rbrace$ with $(\lambda E-\Ac)v =(\lambda \Ec-\Ac)v = 0$.
    This in turn implies
    \begin{equation}
        \label{eq:proofController}
        \lambda v^\top \Pc^\top Ev = v^\top\Pc^\top\Ac v = \frac12v^\top(\Pc^\top\Ac+\Ac^\top\Pc)v.
    \end{equation}
    Furthermore, since $\Pc$ is a solution of \eqref{eq:modifiedControlGARE}, we obtain
    \begin{equation}
        \label{eq:proofControllerTwo}
        \begin{aligned}
            &A_c^\top \Pc + \Pc^\top A_c\\
            &=(A-BB^\top\Pc-Q^{-\top} C^\top C)^\top \Pc + \Pc^\top (A-BB^\top\Pc-Q^{-\top} C^\top C) \\
            &= A^\top \Pc + \Pc^\top A - \Pc^\top BB^\top \Pc - \Pc^\top BB^\top \Pc - C^\top C Q^{-1} \Pc - \Pc^\top Q^{-\top} C^\top C \\
            &=-C^\top C -\Pc^\top BB^\top \Pc - C^\top C Q^{-1}\Pc - \Pc^\top Q^{-\top} C^\top C \\
            &= - (C^\top +\Pc^\top B)(C^\top +\Pc^\top B)^\top \preceq 0.
        \end{aligned}
    \end{equation}
    Together with \eqref{eq:proofController} and the positive semidefiniteness of $\Pc^\top E$, cf.~Lemma~\ref{lem:semidefinitenessControl}, this implies $v^\top \Pc^\top Ev = 0$ and $(C +B^\top \Pc)v=0$.
    Then, using $(\lambda \Ec-\Ac)v = 0$ we obtain
    \begin{align*}
        (\lambda E-A)v &= \left(\lambda E-A+B(C+B^\top \Pc )\right)v = (\lambda E-(A-BB^\top\Pc-BC))v\\
        &= (\lambda E-(A-BB^\top\Pc-Q^{-\top}C^\top C))v = (\lambda \Ec-\Ac)v = 0,
    \end{align*}
    i.e., $\lambda>0$ is a generalized eigenvalue of $(E,A) = (E,(J-R)Q)$.
    This is however a contradiction to the dissipative Hamiltonian structure of $(E,A)$, cf.~\cite[Cor.~4.11]{MehMW18}, and thus $(\Ec,\Ac)$ must be regular.

    Next, we show that the controller is passive.
    As outlined in section~\ref{subsec:generalizedKYPLMI} it is sufficient to show that the KYP LMI \eqref{eq:generalizedKYP_II_unrestricted} associated to the system given by $(\Ec, \Ac, \Bc, \Cc)$ has a solution $X$ satisfying $\Ec^\top X\succeq 0$.
    In the following, we show that $X=\Pc$ provides such a solution of \eqref{eq:generalizedKYP_II_unrestricted} associated to $(\Ec, \Ac, \Bc, \Cc)$. 
    First, we obtain that 
    \begin{align*}
        B_c^\top \Pc = (Q^{-\top} C^\top )^\top \Pc = C Q^{-1} \Pc = B^\top \Pc= C_{\mathrm{c}}.
    \end{align*}
    Furthermore, since $\Pc$ is a stabilizing solution of \eqref{eq:modifiedControlGARE}, Lemma~\ref{lem:semidefinitenessControl} implies $E^{\top}\Pc=\Pc^\top E \succeq 0$ and together with \eqref{eq:proofControllerTwo} this yields the passivity of the controller.
 
 With regard to the closed-loop system
 \begin{align*}
    \underbrace{
     \begin{bmatrix} 
        E & 0 \\ 
        0 & E_{\mathrm{c}} 
    \end{bmatrix} 
    }_{=\vcentcolon \scrE}
    \begin{bmatrix} 
        \dot{x} \\ 
        \dot{x}_{\mathrm{c}} 
    \end{bmatrix} 
    = \underbrace{
    \begin{bmatrix} 
        A                   & -BB^\top \Pc \\ 
        Q^{-\top}C^\top C   & A-BB^\top \Pc - Q^{-\top} C^\top C 
    \end{bmatrix} 
    }_{=\vcentcolon \scrA}
    \begin{bmatrix} 
        x \\ x_{\mathrm{c}} 
    \end{bmatrix} 
 \end{align*}
 note that a state space transformation with $S_1=T_1=\begin{bmatrix} I_n & 0 \\ -I_n & I_n \end{bmatrix}$ leads to the strictly equivalent matrix pencil
 \begin{align*}
     (S_1\scrE T_1^{-1},S_1\scrA T_1^{-1}) = \left( 
     \begin{bmatrix} 
        E & 0 \\ 
        0 & E
    \end{bmatrix}
    , 
    \begin{bmatrix} 
        A-BB^\top \Pc & -BB^\top \Pc \\ 
        0 & A-Q^{-\top} C^\top C 
    \end{bmatrix} 
    \right).
 \end{align*}
 The regularity of the controller then follows from the fact that $\Pc$ and $Q^{-\top}$ are stabilizing solutions of \eqref{eq:modifiedControlGARE} and \eqref{eq:modifiedFilterGARE}, respectively.
 Furthermore, this fact implies that all finite eigenvalues of $(\Ec,\Ac)$ are in the open left half-plane and due to the regularity of $(\Ec,\Ac)$ this is equivalent to asymptotic stability of the controller, see for instance \cite[Prop.~2.1]{DuLM13}.
Lastly, this fact also implies that the pencils $(E,A-BB^\top \Pc)$ and $(E,A-Q^{-\top} C^\top C )$ are impulse-free, i.e., there exist $\Sc,\Tc,\Sf,\Tf\in\GL{n}$ such that
\begin{align*}
    \left(\Sc E\Tc^{-1},\Sc(A-BB^\top \Pc)\Tc^{-1}\right) &= 
    \left(
    \begin{bmatrix}
        I_r & 0\\
        0   & 0
    \end{bmatrix}
    ,
    \begin{bmatrix}
        \Jc & 0\\
        0   & I_{n-r}
    \end{bmatrix}
    \right),\\ 
    \left(\Sf E\Tf^{-1},\Sf(A-Q^{-\top} C^\top C)\Tf^{-1}\right) &= 
    \left(
    \begin{bmatrix}
        I_r & 0\\
        0   & 0
    \end{bmatrix}
    ,
    \begin{bmatrix}
        \Jf & 0\\
        0   & I_{n-r}
    \end{bmatrix}
    \right),
\end{align*}
where $\Jc$ and $\Jf$ are in Jordan canonical form.
Thus, performing another state space transformation to the closed-loop system with $S_2=\begin{bmatrix} \Sc & 0 \\ 0 & \Sf \end{bmatrix}$ and $T_2=\begin{bmatrix} \Tc & 0 \\ 0 & \Tf \end{bmatrix}$ shows that $(\scrE,\scrA)$ is strictly equivalent to
\begin{equation*}
    \left( 
     \begin{bmatrix} 
        I_r &   &       & \\ 
            & 0 &       & \\
            &   & I_r   & \\
            &   &       & 0
    \end{bmatrix}
    , 
    \begin{bmatrix} 
        \Jc & 0         & \widetilde{K}_{11}    & \widetilde{K}_{12}\\
        0   & I_{n-r}   & \widetilde{K}_{21}    & \widetilde{K}_{22}\\
        0   & 0         & \Jf               & 0\\
        0   & 0         & 0                 & I_{n-r}
    \end{bmatrix} 
    \right).
 \end{equation*}
 with $\begin{bmatrix} \widetilde{K}_{11} & \widetilde{K}_{12}\\ \widetilde{K}_{21} & \widetilde{K}_{22} \end{bmatrix}\vcentcolon= -\Sc BB^\top \Pc\Tf^{-1}$.
 Another state space transformation with 
 \begin{equation*}
     S_3 = 
     \begin{bmatrix}
         I_r    & 0         & 0     & -\widetilde{K}_{12}\\
         0      & 0         & I_r   & 0\\
         0      & I_{n-r}   & 0     & -\widetilde{K}_{22}\\
         0      & 0         & 0     & I_{n-r}
     \end{bmatrix}
     ,\quad
     T_3^{-1} = 
     \begin{bmatrix}
         I_r    & 0                 & 0         & 0\\
         0      & -\widetilde{K}_{21}   & I_{n-r}   & 0\\
         0      & I_r               & 0         & 0\\
         0      & 0                 & 0         & I_{n-r}
     \end{bmatrix}
 \end{equation*}
 shows that $(\scrE,\scrA)$ is strictly equivalent to
 \begin{equation*}
    \left( 
     \begin{bmatrix} 
        I_r &       &   & \\ 
            & I_r   &   & \\
            &       & 0 & \\
            &       &   & 0
    \end{bmatrix}
    , 
    \begin{bmatrix} 
        \Jc & \widetilde{K}_{11}    & 0         & 0\\
        0   & \Jf               & 0         & 0\\
        0   & 0                 & I_{n-r}   & 0\\
        0   & 0                 & 0         & I_{n-r}
    \end{bmatrix} 
    \right).
 \end{equation*}
 Finally, performing a similarity transformation, which transforms $\begin{bmatrix} \Jc & \widetilde{K}_{11}\\ 0 & \Jf \end{bmatrix}$ to Jordan canonical form, leads to the Weierstraß canonical form of $(\scrE,\scrA)$ and in particular reveals that $(\scrE,\scrA)$ is impulse-free.
\end{proof}

\section{Structure-Preserving LQG Balanced Truncation}

The goal of this section is to derive a structure-preserving port-Hamiltonian model reduction scheme based on balancing and truncation of the modified GAREs \eqref{eq:modifiedGAREs}, thereby extending the results from \cite{BreMS21} to the DAE case. 
As is common for LQG balancing based methods, we further establish an a priori $\mathcal{H}_\infty$ error bound for suitably defined right coprime factorizations. 
Since this error bound implicitly depends on the chosen Hamiltonian, we conclude this section with a theoretical discussion on how to ideally choose the Hamiltonian in order to minimize this bound. 

\subsection{LQG Balancing and Structure-Preserving Truncation}\label{subsec:pHLQGBT}

As a first step towards a balancing procedure, we show that the solutions $\Pf$ and $\Pc$ of \eqref{eq:modifiedGAREs} provide characteristic values that are invariant under specific equivalence transformations.

\begin{proposition}
    \label{prop:LQGcharacteristicValues}
    Consider the port-Hamiltonian system \eqref{eq:pHEQ} and let $(E,A,B,C)$ be a semi-explicit realization with regular pair $(E,A)$, strongly stabilizable triple $(E,A,B)$, and strongly detectable triple $(E,A,C)$.
    Moreover, let $\Pc$ and $\Pf$ be stabilizing solutions of \eqref{eq:modifiedControlGARE} and \eqref{eq:modifiedFilterGARE}, respectively.
    Then, the eigenvalues of the matrix
    \begin{equation}
        \label{eq:characteristicLQGMatrix}
        \begin{bmatrix}
            I_r & 0
        \end{bmatrix}
        \Pf E^\top \Pc^\top E
        \begin{bmatrix}
            I_r\\
            0
        \end{bmatrix}
    \end{equation}
    are invariants of the system in the sense that they are independent of the particular semi-explicit realization.
\end{proposition}

\begin{proof}
    The proof is similar to the proof of Theorem~4.3(iii) from \cite{MoeRS11}.
    We consider an equivalence transformation defined by matrices $S,T\in\GL{n}$ with the constraint 
    \begin{equation}
        \label{eq:preservingSemiExplicitness}
        \underbrace{
        \begin{bmatrix}
            S_{11} & S_{12}\\
            S_{21} & S_{22}
        \end{bmatrix}
        }_{\vcentcolon=S}\underbrace{
        \begin{bmatrix}
            I_r & 0\\
            0   & 0
        \end{bmatrix}
        }_{=E}\underbrace{
        \begin{bmatrix}
            \widehat{T}_{11} & \widehat{T}_{12}\\
            \widehat{T}_{21} & \widehat{T}_{22}
        \end{bmatrix}
        }_{\vcentcolon=T^{-1}} = 
        \begin{bmatrix}
            I_r & 0\\
            0   & 0
        \end{bmatrix}
        ,
    \end{equation}
    i.e., we only consider state space transformations which preserve the semi-explicit structure of $E$.
    From \eqref{eq:preservingSemiExplicitness} we find that such a transformation preserves the semi-explicit structure if and only if $S_{11}$ and $\widehat{T}_{11}$ are invertible with $S_{11}=\widehat{T}_{11}^{-1}$ and $S_{21}$ and $\widehat{T}_{12}$ are zero.
    Another consequence of the semi-explicit structure is that $\Pf E^\top$ and $\Pc^\top E$ have the block structures
    \begin{equation*}
        \Pf E^\top = 
        \begin{bmatrix}
            \mathcal{P}_{\mathrm{f}11}  & 0\\
            0                           & 0
        \end{bmatrix}
        \quad\text{and}\quad \Pc^\top E = 
        \begin{bmatrix}
            \mathcal{P}_{\mathrm{c}11}  & 0\\
            0                           & 0
        \end{bmatrix}
    \end{equation*}
    where the matrices $\mathcal{P}_{\mathrm{f}11},\mathcal{P}_{\mathrm{c}11}\in\RR^{r\times r}$ are symmetric and positive semidefinite according to Lemma~\ref{lem:semidefinitenessControl} and Corollary~\ref{cor:semidefinitenessFilter}.
            
    To prove the claim, we consider the transformed system \eqref{eq:pHEQ_transformed} with $S,T\in\GL{n}$ satisfying \eqref{eq:preservingSemiExplicitness}.
    By Lemma~\ref{lem:transformedGramians}, stabilizing solutions of the GAREs \eqref{eq:modifiedGAREs} associated with the transformed system are given by $\Pctilde = S^{-\top}\Pc T^{-1}$ and $\Pftilde = S\Pf T^\top$, respectively.
    Consequently, we obtain that the matrix
    \begin{align*}
        &  
        \begin{bmatrix}
            I_r & 0
        \end{bmatrix}
        \Pftilde\Etilde^\top \Pctilde^\top\Etilde
        \begin{bmatrix}
            I_r\\
            0
        \end{bmatrix}
        = 
        \begin{bmatrix}
            I_r & 0
        \end{bmatrix}
        S\Pf T^\top T^{-\top}E^\top S^\top T^{-\top}\Pc^\top S^{-1}SET^{-1}
        \begin{bmatrix}
            I_r\\
            0
        \end{bmatrix}
        \\
        &= 
        \begin{bmatrix}
            I_r & 0
        \end{bmatrix}
        S\Pf E^\top S^\top T^{-\top}\Pc^\top ET^{-1}
        \begin{bmatrix}
            I_r\\
            0
        \end{bmatrix}
        \\
        &= 
        \begin{bmatrix}
            I_r & 0
        \end{bmatrix}
        \begin{bmatrix}
            S_{11}  & S_{12}\\
            0       & S_{22}
        \end{bmatrix}
        \begin{bmatrix}
            \mathcal{P}_{\mathrm{f}11}  & 0\\
            0                           & 0
        \end{bmatrix}
        \begin{bmatrix}
            I_r & *\\
            *   & *
        \end{bmatrix}
        \begin{bmatrix}
            \mathcal{P}_{\mathrm{c}11}  & 0\\
            0                           & 0
        \end{bmatrix}
        \begin{bmatrix}
            \widehat{T}_{11} & 0\\
            \widehat{T}_{21} & \widehat{T}_{22}
        \end{bmatrix}
        \begin{bmatrix}
            I_r\\
            0
        \end{bmatrix}
        \\
        &= S_{11}
        \begin{bmatrix}
            I_r & 0
        \end{bmatrix}
        \begin{bmatrix}
            \mathcal{P}_{\mathrm{f}11}  & 0\\
            0                           & 0
        \end{bmatrix}
        \begin{bmatrix}
            \mathcal{P}_{\mathrm{c}11}  & 0\\
            0                           & 0
        \end{bmatrix}
        \begin{bmatrix}
            I_r\\
            0
        \end{bmatrix}
        \widehat{T}_{11} = S_{11}
        \begin{bmatrix}
            I_r & 0
        \end{bmatrix}
        \Pf E^\top\Pc^\top E
        \begin{bmatrix}
            I_r\\
            0
        \end{bmatrix}
        S_{11}^{-1}
    \end{align*}
    is similar to the matrix from \eqref{eq:characteristicLQGMatrix} and, thus, has the same eigenvalues.
\end{proof}

\begin{theorem}\label{thm:bal_real}
    Consider the system \eqref{eq:pHEQ} with regular pair $(E,A)$, strongly stabilizable triple $(E,A,B)$, and strongly detectable triple $(E,A,C)$.
    Then, the following assertions are true.
    \begin{enumerate}
        \item[(i)] There exist $S,T\in\GL{n}$ such that the transformed system \eqref{eq:pHEQ_transformed} is semi-explicit and such that all stabilizing solutions $\Pctilde,\Pftilde$ of \eqref{eq:modifiedGAREs} associated to the transformed system \eqref{eq:pHEQ_transformed} satisfy
            \begin{equation*}
                \Pftilde\Etilde^\top = 
                \begin{bmatrix}
                    \Sigma  & 0         & 0 & 0 & 0\\
                    0       & \Sigma_2  & 0 & 0 & 0\\
                    0       & 0         & 0 & 0 & 0\\
                    0       & 0         & 0 & 0 & 0\\
                    0       & 0         & 0 & 0 & 0
                \end{bmatrix}
                ,\quad \Etilde^\top\Pctilde = 
                \begin{bmatrix}
                    \Sigma  & 0 & 0         & 0 & 0\\
                    0       & 0 & 0         & 0 & 0\\
                    0       & 0 & \Sigma_3  & 0 & 0\\
                    0       & 0 & 0         & 0 & 0\\
                    0       & 0 & 0         & 0 & 0
                \end{bmatrix}
            \end{equation*}
            with positive definite, diagonal matrices $\Sigma=\diag(\sigma_1,\ldots,\sigma_k)$, $\Sigma_2$, $\Sigma_3$ and $\sigma_1\ge \sigma_2 \ge \ldots \ge \sigma_k$.
        \item[(ii)] The $\Qtilde$ matrix of the transformed system from (i) has the block structure
            \begin{equation*}
                \Qtilde =
                \begin{bmatrix}
                    \Qtilde_{11}  & 0\\
                    \Qtilde_{21}  & \Qtilde_{22}
                \end{bmatrix}
            \end{equation*}
            and, in particular, $\Qtilde_{11}\in\RR^{r\times r}$ is a diagonal matrix.
    \end{enumerate}
\end{theorem}

\begin{proof}
    \begin{enumerate}
        \item[(i)] The proof follows along the lines of the proof of Theorem~4.3 from \cite{MoeRS11}.
            As outlined in section~\ref{subsec:pHDAEs}, there is an equivalence transformation which transforms the original port-Hamiltonian system \eqref{eq:pHEQ} into a semi-explicit port-Hamiltonian system.
            This intermediate system is denoted with $(\widehat{E},\widehat{A},\widehat{B},\widehat{C})$.
            By Lemmata~\ref{lem:semidefinitenessControl} and \ref{cor:semidefinitenessFilter}, any stabilizing solutions $\Pchat$, $\Pfhat$ of \eqref{eq:modifiedGAREs} associated with the semi-explicit system $(\widehat{E},\widehat{A},\widehat{B},\widehat{C})$ satisfy  $\widehat{E}\Pfhat^\top=\Pfhat\widehat{E}^\top\succeq 0$ and $\widehat{E}^\top\Pchat = \Pchat^\top \widehat{E}\succeq 0$, respectively.
            Combining this with the semi-explicit form of $\widehat{E}$ we obtain 
            \begin{equation*}
                \Pfhat = 
                \begin{bmatrix}
                    \mathcal{P}_{\mathrm{f}11}   & \mathcal{P}_{\mathrm{f}12}\\
                    0                                       & \mathcal{P}_{\mathrm{f}22}
                \end{bmatrix}
                ,\quad \Pchat = 
                \begin{bmatrix}
                    \mathcal{P}_{\mathrm{c}11} & 0\\
                    \mathcal{P}_{\mathrm{c}21} & \mathcal{P}_{\mathrm{c}22}
                \end{bmatrix}
                ,
            \end{equation*}
            where $\mathcal{P}_{\mathrm{f}11}$ and $\mathcal{P}_{\mathrm{c}11}$ are symmetric and positive semidefinite.
            Then, by Theorem~3.22 from \cite{ZhoDG96}, there exists $T_1\in\GL{r}$ with 
            \begin{equation}
                \label{eq:simultaneousDiagonalization}
                T_1 \mathcal{P}_{\mathrm{f}11}T_1^\top = 
                \begin{bmatrix}
                    \Sigma  & 0         & 0 & 0\\
                    0       & \Sigma_2  & 0 & 0\\
                    0       & 0         & 0 & 0\\
                    0       & 0         & 0 & 0
                \end{bmatrix}
                ,\quad T_1^{-\top}\mathcal{P}_{\mathrm{c}11}T_1^{-1} =
                \begin{bmatrix}
                    \Sigma  & 0 & 0         & 0\\
                    0       & 0 & 0         & 0\\
                    0       & 0 & \Sigma_3  & 0\\
                    0       & 0 & 0         & 0
                \end{bmatrix}
                ,
            \end{equation}
            where $\Sigma$, $\Sigma_2$, and $\Sigma_3$ are positive definite, diagonal matrices.
            Especially, $T_1$ can be chosen such that the diagonal entries of $\Sigma$ are in descending order, cf.~the proof of Theorem~3.22 in \cite{ZhoDG96}.
            Finally, from Lemma~\ref{lem:transformedGramians} we infer that using the transformation matrices
            \begin{equation}
                \label{eq:balancingTransformation}
                \widehat{S} = \widehat{T} = 
                \begin{bmatrix}
                    T_1 & 0\\
                    0   & I_{n-r}
                \end{bmatrix}
            \end{equation}
            and applying a corresponding state space transformation to the system given by $(\widehat{E},\widehat{A},\widehat{B},\widehat{C})$, the resulting transformed system given by
            \begin{equation*}
                (\Etilde,\Atilde,\Btilde,\Ctilde) = (\widehat{S}\widehat{E}\widehat{T}^{-1},\widehat{S}\widehat{A}\widehat{T}^{-1},\widehat{S}\widehat{B},\widehat{C}\widehat{T}^{-1})
            \end{equation*}
            satisfies the properties of the claim.
        \item[(ii)] In order to prove the claim, we investigate the state space transformations used in the proof of (i) one after another.
            First, we consider the SVD of $E$ as in \eqref{eq:svdOfE} and the transformation matrices from \eqref{eq:EquivalenceTransformationToSemiExplicit} leading to the semi-explicit port-Hamiltonian realization $(\widehat{E},\widehat{A},\widehat{B},\widehat{C})$.
            We introduce the column splittings $U=[U_1\;\;U_2]$ and $V=[V_1\;\;V_2]$ with $U_1,V_1\in\RR^{n\times r}$ and $U_2,V_2\in\RR^{n\times (n-r)}$ as well as the splitting of $\widehat{Q}=S^{-\top}QT^{-1}$ given by
            \begin{equation*}
                \widehat{Q} = 
                \begin{bmatrix}
                    \widehat{Q}_{11} & \widehat{Q}_{12}\\
                    \widehat{Q}_{21} & \widehat{Q}_{22}
                \end{bmatrix}
            \end{equation*}
            with $\widehat{Q}_{11}\in\RR^{r\times r}$ and the other block sizes are accordingly.
            From the semi-explicit form of $\widehat{E}$ and the symmetry of $\widehat{E}^\top \widehat{Q}$, we immediately obtain that $\widehat{Q}_{11}$ is symmetric and that $\widehat{Q}_{12}=0$ holds.
            Furthermore, by Theorem~\ref{thm:solutionsOfNonsymmetricGAREs} $\widehat{Q}^{-\top}$ is a stabilizing solution of the modified filter GARE \eqref{eq:modifiedFilterGARE} associated with the system $(\widehat{E},\widehat{A},\widehat{B},\widehat{C})$.
            In addition, due to the semi-explicit form of $\widehat{E}$ and due to Theorem~\ref{thm:solutionsOfNonsymmetricGAREs}, the upper left $r\times r$ blocks of all stabilizing solutions of \eqref{eq:modifiedFilterGARE} associated with the system $(\widehat{E},\widehat{A},\widehat{B},\widehat{C})$ coincide.
            This block is called $\mathcal{P}_{\mathrm{f}11}$ in the proof of (i) and in particular coincides with $\widehat{Q}_{11}^{-\top}=\widehat{Q}_{11}^{-1}$, which follows from the fact that $\widehat{Q}$ is block triangular.
            An immediate consequence of this is that $\mathcal{P}_{\mathrm{f}11}$ is invertible and, thus, the third and fourth block rows and columns in the right-hand side of the first equation in \eqref{eq:simultaneousDiagonalization} have zero size.

            Based on $\widehat{Q}$, the second state space transformation leads to $\Qtilde = \widehat{S}^{-\top}\widehat{Q}\widehat{T}^{-1}$, where $\widehat{S}$ and $\widehat{T}$ are as specified in \eqref{eq:balancingTransformation}.
            As a consequence, we obtain
            \begin{align*}
                \Qtilde &= \widehat{S}^{-\top}\widehat{Q}\widehat{T}^{-1} =
                \begin{bmatrix}
                    T_1^{-\top} & 0\\
                    0           & I_{n-r}
                \end{bmatrix}
                \begin{bmatrix}
                    \widehat{Q}_{11} & 0\\
                    \widehat{Q}_{21} & \widehat{Q}_{22}
                \end{bmatrix}
                \begin{bmatrix}
                    T_1^{-1} & 0\\
                    0   & I_{n-r}
                \end{bmatrix}
                \\
                &=
                \begin{bmatrix}
                    T_1^{-\top}\widehat{Q}_{11}T_1^{-1} & 0\\
                    \widehat{Q}_{21}T_1^{-1}            & \widehat{Q}_{22}
                \end{bmatrix}
                =
                \begin{bmatrix}
                    T_1^{-\top}\mathcal{P}_{\mathrm{f}11}^{-1}T_1^{-1}   & 0\\
                    \widehat{Q}_{21}T_1^{-1}                                            & \widehat{Q}_{22}
                \end{bmatrix}
                \\
                &=
                \begin{bmatrix}
                    \left(T_1\mathcal{P}_{\mathrm{f}11}T_1^\top\right)^{-1}  & 0\\
                    \widehat{Q}_{21}T_1^{-1}                                                & \widehat{Q}_{22}
                \end{bmatrix}
                .
            \end{align*}
            The claim then follows from \eqref{eq:simultaneousDiagonalization}.
    \end{enumerate}
\end{proof}

The diagonal matrix $\Sigma$ in Theorem~\ref{thm:bal_real} contains the square roots of the non-zero eigenvalues of the matrix given in \eqref{eq:characteristicLQGMatrix} on the diagonal and as shown in Proposition~\ref{prop:LQGcharacteristicValues} these are invariants of the system \eqref{eq:pHEQ}.
Moreover, these invariants and the balanced realization from Theorem~\ref{thm:bal_real} play an important role for the error bound provided in section~\ref{subsec:errorBound}.

Let us emphasize that the particular structure of $\Qtilde$ in the balanced realization in Theorem \ref{thm:bal_real} is essential for showing that a reduced model obtained by truncation has a port-Hamiltonian realization \eqref{eq:pHEQ}. In particular, we have the following result.
 
\begin{theorem}\label{thm:rom_is_ph}
  Let $(\Etilde,\Atilde,\Btilde,\Ctilde)$ be given with $\Pctilde, \Pftilde$ as in Theorem \ref{thm:bal_real}.
  Then, for $\ell\in\lbrace 1,\ldots,r\rbrace$ and $
    \trunc =
    \begin{bmatrix}
        I_\ell  & 0 & 0\\
        0       & 0 & I_{n-r}
    \end{bmatrix}\in\RR^{(\ell+n-r)\times n}$
the reduced-order model $(\redE,\redA,\redB,\redC)=(\trunc \Etilde \trunc^\top,\trunc \Atilde \trunc^\top,\trunc \Btilde,\Ctilde \trunc^\top)$ is port-Hamiltonian.
\end{theorem}

\begin{proof}
    With slight abuse of notation (w.r.t.~$\Qtilde$) from Theorem \ref{thm:bal_real}, let us consider the partitioning 
    \begin{equation}\label{eq:aux_part}
        \Qtilde =
        \begin{bmatrix}
            \Qtilde_{11}  & 0                 & 0\\
            0               & \Qtilde_{22}    & 0\\
            \Qtilde_{31}  & \Qtilde_{32}    & \Qtilde_{33}
        \end{bmatrix}
        ,\quad 
        \widetilde{J} =
        \begin{bmatrix}
            \widetilde{J}_{11}  & \widetilde{J}_{12}    & \widetilde{J}_{13}\\
            \widetilde{J}_{21}  & \widetilde{J}_{22}    & \widetilde{J}_{23}\\
            \widetilde{J}_{31}  & \widetilde{J}_{32}    & \widetilde{J}_{33}
        \end{bmatrix}
        ,\quad 
        \Rtilde =
        \begin{bmatrix}
            \Rtilde_{11}  & \Rtilde_{12}    & \Rtilde_{13}\\
            \Rtilde_{21}  & \Rtilde_{22}    & \Rtilde_{23}\\
            \Rtilde_{31}  & \Rtilde_{32}    & \Rtilde_{33}
        \end{bmatrix}
        ,
    \end{equation}
    where $\Qtilde_{11},\widetilde{J}_{11},\Rtilde_{11}\in \RR^{\ell\times\ell}$, $\Qtilde_{33},\widetilde{J}_{33},\Rtilde_{33}\in \RR^{(n-r)\times(n-r)}$, and the sizes of the other blocks are accordingly. 
    Then, straightforward calculations yield
    \begin{align*}
    \redA &= \trunc (\widetilde{J}-\Rtilde)\Qtilde \trunc^\top \\
    & =  \begin{bmatrix}
        (\widetilde{J}_{11}-\Rtilde_{11})\Qtilde_{11}+(\widetilde{J}_{13}-\Rtilde_{13})\Qtilde_{31} &  (\widetilde{J}_{13}-\Rtilde_{13})\Qtilde_{33}\\
        (\widetilde{J}_{31}-\Rtilde_{31})\Qtilde_{11}+(\widetilde{J}_{33}-\Rtilde_{33})\Qtilde_{31} &  (\widetilde{J}_{33}-\Rtilde_{33})\Qtilde_{33}
    \end{bmatrix} \\
    &= 
    \Big(
    \underbrace{
    \begin{bmatrix}
        \widetilde{J}_{11} & \widetilde{J}_{13}\\
        \widetilde{J}_{31} & \widetilde{J}_{33}
    \end{bmatrix}
    }_{=\redJ}-\underbrace{
    \begin{bmatrix}
        \Rtilde_{11} & \Rtilde_{13}\\
        \Rtilde_{31} & \Rtilde_{33}
    \end{bmatrix}
    }_{=\redR}
    \Big)
    \underbrace{
    \begin{bmatrix}
        \Qtilde_{11} & 0\\
        \Qtilde_{31} & \Qtilde_{33}
    \end{bmatrix}
    }_{=\redQ}.
    \end{align*}
    Note that $\redJ=\trunc \widetilde{J} \trunc^\top$ and $\redR=\trunc \Rtilde \trunc^\top$ implying $\redJ=-\redJ^\top$ as well as $\redR=\redR^\top \succeq 0$. Since $\Etilde$ is in semi-explicit form, we further obtain that 
    \begin{align*}
        \redE^\top \redQ=\trunc \Etilde \trunc^\top \redQ = 
        \begin{bmatrix}
            I_\ell  & 0 \\ 
            0       & 0 
        \end{bmatrix} 
        \begin{bmatrix}
        \Qtilde_{11} & 0\\
        \Qtilde_{31} & \Qtilde_{33}
        \end{bmatrix} = \begin{bmatrix} \Qtilde_{11}^\top & \Qtilde_{31}^\top \\ 0 & \Qtilde_{33}^\top  \end{bmatrix} 
        \begin{bmatrix} 
            I_\ell & 0 \\ 
            0 & 0
        \end{bmatrix} 
        =\redQ^\top \redE.
    \end{align*}
    Furthermore, it holds that 
    \begin{align*}
        \redE^\top \redQ &= \begin{bmatrix} \Qtilde_{11} & 0 \\ 0 & 0\end{bmatrix} = 
        \begin{bmatrix} 
            I_\ell  & 0 \\ 
            0       & 0 \\ 
            0       & I_{n-r} 
        \end{bmatrix}
        ^\top \begin{bmatrix} \Qtilde_{11} & 0 & 0 \\ 0 & \Qtilde_{22} & 0 \\ 0 & 0 & 0 \end{bmatrix} 
        \begin{bmatrix} 
            I_\ell  & 0 \\    
            0       & 0 \\ 
            0       & I_{n-r} 
        \end{bmatrix}  
        = \trunc \Etilde^\top \Qtilde \trunc^\top \\
        &\succeq 0 .
    \end{align*}
    Finally, since $(\Etilde,\Atilde,\Btilde,\Ctilde)$ is obtained by an equivalance transformation as in \eqref{eq:pHEQ_transformed} with a partitioning of $\Btilde$ and $\Ctilde$ according to \eqref{eq:aux_part} we find that
    \begin{align*}
        \redC &= \Ctilde\trunc^\top = \Btilde^\top \Qtilde\trunc^\top = 
        \begin{bmatrix} 
            \Btilde_1^\top & \Btilde_2^\top & \Btilde_3^\top
        \end{bmatrix} 
        \begin{bmatrix} \Qtilde_{11} & 0 \\ 0 & 0 \\ \Qtilde_{31} & \Qtilde_{33} \end{bmatrix} \\
        &= 
        \begin{bmatrix} 
            \Btilde_1^\top \Qtilde_{11} + \Btilde_3^\top \Qtilde_{31} & \Btilde_3^\top \Qtilde_{33} 
        \end{bmatrix}
        \\
        &= \begin{bmatrix} 
            \Btilde_1^\top & \Btilde_3^\top 
        \end{bmatrix} 
        \begin{bmatrix} \Qtilde_{11} & 0 \\ \Qtilde_{31} & \Qtilde_{33} \end{bmatrix} = \redB^\top \redQ,
    \end{align*}
    which shows the assertion.
\end{proof}

\begin{remark}
  Note that the previous result is a consequence of the particular balancing w.r.t.~the modified Riccati equations \eqref{eq:modifiedGAREs}. 
  Indeed, since $\Pf=Q^{-\top}$ is a stabilizing solution of the modified filter GARE \eqref{eq:modifiedFilterGARE}, we implicitly performed an \emph{effort constraint} balanced truncation which preserves the port-Hamiltonian structure. 
  For a more detailed discussion, we refer to the ODE case \cite[sec.~3.1]{BreMS21}.
\end{remark}

Theorem~\ref{thm:rom_is_ph} provides a structure-preserving model reduction method for port-Hamiltonian DAE systems of the form \eqref{eq:pHEQ} by truncating a balanced realization.
Similarly as in \cite{MoeRS11}, the balanced system never needs to be explicitly computed and the balancing and truncation step can be combined.
The computational steps for obtaining a pH reduced-order model as in Theorem~\ref{thm:rom_is_ph} are summarized in Algorithm~\ref{alg:MOR}.
The calculations showing that this algorithm indeed yields a ROM as in Theorem~\ref{thm:rom_is_ph} are provided in appendix~\ref{sec:algorithm}.

\begin{algorithm}[H]
    \caption{Structure-preserving LQG-BT for port-Hamiltonian DAEs}
    \label{alg:MOR}
    \begin{algorithmic}[1]
        \REQUIRE $E, J, R, Q\in \RR^{n\times n}, B\in \RR^{n\times m}$ as in \eqref{eq:pHEQ} satisfying the assumptions of Theorem~\ref{thm:bal_real}
        \ENSURE Port-Hamiltonian ROM $(\redE,\;\redA=(\redJ-\redR)\redQ,\;\redB,\;\redC=\redB^\top \redQ)$ \\[0.1cm]
        \STATE Compute an SVD of $E$ as in \eqref{eq:svdOfE}.
        \STATE Compute full-column-rank matrices $\LPc,\LPf$ such that $E^\top \Pc = \LPc \LPc^\top$ and $E\Pf^\top = \LPf\LPf^\top$, where $\Pc$ and $\Pf = Q^{-\top}$ are stabilizing solutions of \eqref{eq:modifiedGAREs}.
        \STATE Compute a singular value decomposition 
        \begin{equation*}
            \LPc^\top E^+\LPf = 
            \begin{bmatrix}
                \widetilde{U}_1 & \widetilde{U}_2
            \end{bmatrix}
            \begin{bmatrix}
                \widetilde{\Sigma}_1    & 0\\
                0                   & 0
            \end{bmatrix}
            \begin{bmatrix}
                \widetilde{V}_1^\top\\
                \widetilde{V}_2^\top
            \end{bmatrix}
        \end{equation*}
        with $\widetilde{\Sigma}_1\succ 0$ and determine a suitable dimension $\ell$ for the truncation step based on the singular value decay.
        \STATE Assemble the projection matrices 
        \begin{align*}
            S_\ell &\vcentcolon= [(E^+)^\top \LPc\widetilde{U}_\ell\widetilde{\Sigma}_\ell^{-\frac12}\;\;U_2] \in\RR^{n\times(\ell+n-r)},\\
            T_\ell &\vcentcolon= [E^+\LPf\widetilde{V}_\ell\widetilde{\Sigma}_\ell^{-\frac12}\;\;V_2]\in\RR^{n\times(\ell+n-r)},\\
            W_\ell &\vcentcolon= [EE^+\LPf\widetilde{V}_\ell\widetilde{\Sigma}_\ell^{-\frac12}\;\;U_2]\in\RR^{n\times(\ell+n-r)},
        \end{align*}
        where $\widetilde{U}_\ell$ and $\widetilde{V}_\ell$ contain the first $\ell$ columns of $\widetilde{U}_1$ and $\widetilde{V}_1$, respectively, and $\widetilde{\Sigma}_\ell$ the first $\ell$ columns and rows of $\widetilde{\Sigma}_1$.
        \STATE Obtain the ROM matrices via $\redE = S_\ell^\top ET_\ell$, $\redJ = S_\ell^\top ES_\ell$, $\redR = S_\ell^\top RS_\ell$, $\redQ = W_\ell^\top QT_\ell$, and $\redB = S_\ell^\top R$.
    \end{algorithmic}
\end{algorithm}

\subsection{A Priori Error Bound for Right Coprime Factors}\label{subsec:errorBound}

In this subsection, we establish a classical a priori $\mathcal{H}_\infty$ bound for the error between the normalized right coprime factorizations of the FOM transfer function $G$ and the ROM transfer function $\redG$. 
As is well-known, see, e.g., \cite[sec.~2 and sec.~4]{MoeRS11}, this in particular allows to bound the so-called \emph{directed gap} between $G$ and $\redG$, offering a way to measure the quality of a reduced system even if the involved systems are unstable. For a more detailed discussion on these concepts, we refer to, e.g., \cite{Liuetal97,Mey90,MoeRS11}.

We start by adapting Lemma~3.13 from \cite{MoeRS11} to our setting, to provide a relation between certain solutions of generalized Lyapunov inequalities and an $\mathcal{H}_\infty$ error bound between the original system and the truncated one.

\begin{lemma}\label{lem:gen_lyap_ineq}
    Consider a semi-explicit descriptor system given by $(E,A,B,C)\in\RR^{n\times n}\times \RR^{n\times n}\times \RR^{n\times m}\times \RR^{p\times n}$ with $r=\rank(E)$ and regular and impulse-free pair $(E,A)$ whose finite eigenvalues lie in the open left half-plane.
    Let $P_1,P_2\in \RR^{n\times n}$ fulfill the following generalized Lyapunov inequalities 
 \begin{equation}\label{eq:lyap_ineq_errbnd}
 \begin{aligned}
  AP_1^\top + P_1A^\top + BB^\top &\preceq 0, \quad EP_1^\top-P_1E^\top = 0, \\
  A^\top P_2 + P_2^\top A + C^\top C &\preceq 0, \quad E^\top P_2-P_2^\top E = 0.
 \end{aligned}
 \end{equation}
Furthermore, assume that for some diagonal matrices $\Gamma=\mathrm{diag}(\gamma_1,\dots,\gamma_r)\in \RR^{r \times r}$ and $\Theta=\mathrm{diag}(\theta_1,\dots,\theta_r)\in \RR^{r \times r}$ with positive elements $\gamma_1,\dots,\gamma_r$ and non-negative elements $\theta_1,\dots,\theta_r$ such that 
\begin{align*}
 \gamma_1 \theta_1 \ge \gamma_2 \theta_2 \ge \cdots \ge \gamma_{\ell} \theta_{\ell} \ge \gamma_{\ell +1 } \theta_{\ell+1} \ge \dots \ge \gamma_r \theta_r\ge 0
\end{align*}
there holds 
\begin{align*}
 P_1=\begin{bmatrix} \Gamma & \ast \\ 0 & \ast \end{bmatrix}, \quad 
 P_2=\begin{bmatrix} \Theta & 0 \\ \ast & \ast \end{bmatrix}.
\end{align*}
    Then, for $\ell\in\lbrace 1,\ldots,r\rbrace$,
    \begin{equation}
        \label{eq:truncationMat}
        \trunc = 
        \begin{bmatrix} 
            I_{\ell}    & 0 & 0 \\ 
            0           & 0 & I_{n-r} 
        \end{bmatrix}
        \in\RR^{(n+\ell-r)\times n},
    \end{equation}
    and the matrices $\widetilde{E}=\trunc E \trunc^\top$, $\widetilde{A}=\trunc A \trunc^\top$, $\widetilde{B}=\trunc B$, $\widetilde{C}=C\trunc^\top$, $\widetilde{P}_1=\trunc P_1\trunc^\top$, $\widetilde{P}_2=\trunc P_2 \trunc^\top$, there holds 
\begin{equation}\label{eq:lyap_ineq_errbnd_sim}
\begin{aligned}
 \widetilde{A} \widetilde{P}_1^\top + \widetilde{P}_1 \widetilde{A}^\top + \widetilde{B}\widetilde{B}^\top &\preceq 0, \quad  \widetilde{E}\widetilde{P}_1^\top - \widetilde{P}_1\widetilde{E}^\top = 0, \\
 \widetilde{A}^\top \widetilde{P}_2+ \widetilde{P}_2^\top \widetilde{A} + \widetilde{C}^\top\widetilde{C} &\preceq 0, \quad  \widetilde{E}^\top \widetilde{P}_2 - \widetilde{P}_2^\top\widetilde{E}= 0.
\end{aligned}
\end{equation}
Moreover, the transfer functions $G(s) = C(sE-A)^{-1}B$ and $\widetilde{G}(s) = \widetilde{C}(s\widetilde{E}-\widetilde{A})^{-1}\widetilde{B}$ satisfy
\begin{equation}
    \label{eq:errorBound}
 \| G-\widetilde{G}\|_{\mathcal{H}_\infty} \le 2 \sum_{k=\ell+1}^r \sqrt{\gamma_k \theta_k}.
\end{equation}

\end{lemma}

\begin{proof}
 The proof follows along the lines of those of \cite[Lemma 3.13]{MoeRS11}. For exposition, we provide the main arguments. Pre- and postmultiplication of the first inequality in \eqref{eq:lyap_ineq_errbnd} with $\trunc$ and $\trunc^\top$, respectively, yields 
 \begin{align*}
  \trunc AP_1^\top \trunc^\top + \trunc P_1A^\top \trunc^\top + \trunc BB^\top \trunc^\top &\preceq 0.
 \end{align*}
The first inequality in \eqref{eq:lyap_ineq_errbnd_sim} is then obtained by utilizing that $\trunc^\top \trunc P_1^\top \trunc^\top = P_1^\top \trunc^\top$ which follows from the partitioning of $P_1$.
The rest of \eqref{eq:lyap_ineq_errbnd_sim} follows by similar arguments.

For the error bound \eqref{eq:errorBound}, we interpret the difference $G-\widetilde{G}$ as a transfer function of an ODE system. Note that since $(E,A,B,C)$ is semi-explicit and impulse-free, assuming a partitioning 
\begin{align*}
 E=\begin{bmatrix} I_r & 0 \\ 0 & 0 \end{bmatrix},\quad  A=\begin{bmatrix} A_{11} & A_{12} \\ A_{21} & A_{22} \end{bmatrix},\quad B=\begin{bmatrix} B_1 \\ B_2 \end{bmatrix}, \quad C=\begin{bmatrix} C_1 & C_2 \end{bmatrix}
\end{align*}
it follows that $A_{22}$ is invertible, see for instance \cite[Prop.~2.1]{Ria08}.
This allows for the following calculations 
\begin{align*}
    &G(s)=C (sE-A)^{-1}B = \begin{bmatrix} C_1 & C_2 \end{bmatrix} 
    \begin{bmatrix} 
        sI_r-A_{11} & -A_{12} \\ 
        -A_{21}     & -A_{22} 
    \end{bmatrix}
    ^{-1} \begin{bmatrix} B_1 \\ B_2\end{bmatrix} \\
    &= 
    \begin{bmatrix} 
        C_1^\top\\
        C_2^\top
    \end{bmatrix} 
    ^\top
    \begin{bmatrix} 
        (sI_r-M)^{-1}                       & -(sI_r-M)^{-1}A_{12} A_{22}^{-1} \\
        -A_{22}^{-1} A_{21} (sI_r-M)^{-1}   & A_{22}^{-1}A_{21} (sI_r-M)^{-1} A_{12} A_{22}^{-1}-A_{22}^{-1} 
    \end{bmatrix} 
    \begin{bmatrix} B_1 \\ B_2\end{bmatrix} \\
    &=(C_1-C_2A_{22}^{-1} A_{21}) (sI_r-M)^{-1} (B_1-A_{12} A_{22}^{-1}B_2) - C_2 A_{22}^{-1}B_2
\end{align*}
where $M\vcentcolon=A_{11}-A_{12}A_{22}^{-1}A_{21}$. 
Similarly, we obtain the following partitioning for $(\widetilde{E},\widetilde{A},\widetilde{B},\widetilde{C})$:
\begin{align*}
 \widetilde{E}&=\trunc\begin{bmatrix} I_r & 0 \\ 0 & 0 \end{bmatrix}\trunc^\top = \begin{bmatrix} I_{\ell} & 0 \\ 0 & 0 \end{bmatrix},\quad  \widetilde{A}=\trunc\begin{bmatrix} A_{11} & A_{12} \\ A_{21} & A_{22} \end{bmatrix}\trunc^\top = \begin{bmatrix} \widetilde{A}_{11} & \widetilde{A}_{12} \\  \widetilde{A}_{21} & \widetilde{A}_{22} \end{bmatrix}, \\ 
 &\quad \widetilde{B}=\trunc\begin{bmatrix} B_1 \\ B_2 \end{bmatrix}=\begin{bmatrix} \widetilde{B}_1 \\ \widetilde{B}_2 \end{bmatrix}, \quad \widetilde{C}= \begin{bmatrix} C_1 & C_2 \end{bmatrix}\trunc^\top = \begin{bmatrix} \widetilde{C}_1 & \widetilde{C}_2 \end{bmatrix} .
\end{align*}
Using $\widehat{S}\vcentcolon=\begin{bmatrix} I_{\ell} & 0\end{bmatrix}\in\RR^{\ell,r}$ this yields 
\begin{align*}
 \widetilde{A}_{11} &= \widehat{S}A_{11} \widehat{S}^\top, \ \  \widetilde{A}_{12} = \widehat{S} A_{12},\ \  \widetilde{A}_{21} = A_{21} \widehat{S}^\top , \ \  \widetilde{B}_1  = \widehat{S} B_1 , \ \  \widetilde{C}_1 = C_1 \widehat{S}^\top, \\
 \widetilde{A}_{22} &= A_{22}, \ \ \widetilde{B}_2 =B_2, \ \ \widetilde{C}_2=C_2.
\end{align*}
Since $\widetilde{E}$ is in semi-explicit form and $\widetilde{A}_{22}$ invertible, $(\widetilde{E},\widetilde{A})$ is regular, see e.g.~\cite[Prop.~2.1]{Ria08}, and thus the associated transfer function $\widetilde{G}$ is well-defined and given by
\begin{align*}
 \widetilde{G}(s) &= \widetilde{C} (s\widetilde{E}-\widetilde{A})^{-1} \widetilde{B} \\
 &=(\widetilde{C}_1-\widetilde{C}_2\widetilde{A}_{22}^{-1} \widetilde{A}_{21}) (sI_\ell-\widetilde{M})^{-1} (\widetilde{B}_1-\widetilde{A}_{12} \widetilde{A}_{22}^{-1}\widetilde{B}_2) - \widetilde{C}_2 \widetilde{A}_{22}^{-1}\widetilde{B}_2 \\
 &= (C_1-C_2A_{22}^{-1} A_{21})\widehat{S}^\top (sI_\ell-\widehat{S}M\widehat{S}^\top)^{-1} \widehat{S} (B_1-A_{12} A_{22}^{-1}B_2) - C_2 A_{22}^{-1}B_2.
\end{align*}
with $\widetilde{M} \vcentcolon= \widetilde{A}_{11}-\widetilde{A}_{12}\widetilde{A}_{22}^{-1}\widetilde{A}_{21}$.
Finally, note that we can pre- and postmultiply the inequalities in \eqref{eq:lyap_ineq_errbnd} by the matrices 
\begin{align*}
    \begin{bmatrix} 
        I_r & -A_{12}A_{22}^{-1} 
    \end{bmatrix}
    , \quad 
    \begin{bmatrix} 
        I_r \\ 
        -A_{22}^{-\top } A_{12}^\top 
    \end{bmatrix}
    , \quad 
    \begin{bmatrix} 
        I_r & -A_{21}^\top A_{22}^{-\top }  
    \end{bmatrix} 
    , \quad
    \begin{bmatrix} 
        I_r \\  
        -A_{22}^{-1}A_{21} 
    \end{bmatrix}
\end{align*}
 in order to obtain 
 \begin{align*}
  &M\Gamma + \Gamma M^\top+ (B_1-A_{12} A_{22}^{-1}B_2)(B_1-A_{12} A_{22}^{-1}B_2)^\top \preceq 0, \\
  &M^\top \Theta + \Theta M + (C_1-C_2A_{22}^{-1}A_{21} )^\top (C_1-C_2A_{22}^{-1}A_{21} ) \preceq 0.
 \end{align*}
The result now follows with arguments used for ODE systems as in the proofs of \cite[Theorem III.3]{DamB14} and \cite[Theorem~8]{BreMS21}. In particular, let us emphasize that the latter proof does not require invertibility of $\Theta$ but only of $\Gamma.$
\end{proof}

While the error bound provided in Lemma~\ref{lem:gen_lyap_ineq} implies $\widetilde{G}\in\mathcal{H}_\infty^{p\times m}$, it does in general not guarantee that all finite eigenvalues of the truncated pencil $(\widetilde{E},\widetilde{A})$ are in $\CC_-$, since $(\widetilde{E},\widetilde{A},\widetilde{B},\widetilde{C})$ is not necessarily a minimal realization.
The following Lemma provides sufficient conditions under which all finite eigenvalues of the truncated pencil are in $\CC_-$.

\begin{lemma}
    \label{lem:lyapIneqs}
    Let the regular, impulse-free, and semi-explicit matrix pencil $(E,A)\in\RR^{n\times n}\times\RR^{n\times n}$ with $r\vcentcolon=\rank(E)$ and all finite eigenvalues in $\CC_-$ be given with the partitioning
    \begin{equation}
        \label{eq:EAthreeBlocks}
        E = 
        \begin{bmatrix}
            I_\ell  & 0             & 0\\
            0       & I_{r-\ell}    & 0\\
            0       & 0             & 0
        \end{bmatrix}
        ,\quad A = 
        \begin{bmatrix}
            A_{11} & A_{12} & A_{13}\\
            A_{21} & A_{22} & A_{23}\\
            A_{31} & A_{32} & A_{33}
        \end{bmatrix}
        ,
    \end{equation}
    $\ell\in\lbrace 1,\ldots,r\rbrace$, $A_{11}\in\RR^{\ell\times \ell}$, $A_{22}\in\RR^{(r-\ell)\times(r-\ell)}$, and the sizes of the other blocks are accordingly.
    Furthermore, assume that there exist matrices $\Theta,\Gamma\in\RR^{n\times n}$ with
    \begin{equation*}
        \Theta = 
        \begin{bmatrix}
            \Theta_{11} & 0             & \Theta_{13}\\
            0           & \Theta_{22}   & \Theta_{23}\\
            0           & 0             & \Theta_{33}
        \end{bmatrix}
        ,\quad \Gamma = 
        \begin{bmatrix}
            \Gamma_{11} & 0             & 0\\
            0           & \Gamma_{22}   & 0\\
            \Gamma_{31} & \Gamma_{32}   & \Gamma_{33}
        \end{bmatrix}
        ,
    \end{equation*}
    where the block sizes are as in \eqref{eq:EAthreeBlocks}, satisfying $\Theta_{11}=\Theta_{11}^\top$, $\Gamma_{11}=\Gamma_{11}^\top\succ 0$, and 
    \begin{align}
        \label{eq:lyapODEIneq_I}
        A\Theta^\top+\Theta A^\top  &\preceq 0,\\
        \label{eq:lyapODEIneq_II}
        A^\top \Gamma+\Gamma^\top A &\preceq 0.
    \end{align}
    Besides, based on the truncation matrix $\trunc$ as defined in \eqref{eq:truncationMat}, we consider the truncated pencil $(\redE,\redA) \vcentcolon= (\trunc E\trunc^\top, \trunc A\trunc^\top)$ and the corresponding truncated quantities $\redTheta \vcentcolon = \trunc\Theta\trunc^\top$ and $\redGamma \vcentcolon = \trunc\Gamma\trunc^\top$.
    If $\spectrum(\Theta_{11}\Gamma_{11})\cap\spectrum(\Theta_{22}\Gamma_{22}^\top)=\emptyset$ holds, then all finite eigenvalues of the reduced pencil $(\redE,\redA)$ are in $\CC_-$.
\end{lemma}

\begin{proof}
    First, we note that, since $(E,A)$ is regular and impulse-free, the semi-explicit structure of $E$ implies that $A_{33}$ is invertible.
    Thus, since $A_{33}$ is not affected by the truncation step, also the reduced pencil $(\redE,\redA)$ is regular and impulse-free, cf.~\cite[Prop.~2.1]{Ria08}.
    Consequently, the finite eigenvalues of $(\redE,\redA)$ are well-defined and we start by showing that they are all contained in $\overline{\CC_-}$.
    To this end, let $(\lambda,v)$ be an eigenpair of $(\redE,\redA)$ and note that \eqref{eq:lyapODEIneq_I} and \eqref{eq:lyapODEIneq_II} imply the existence of $B,C\in\RR^{n\times n}$ with the partitioning
    \begin{equation*}
        B =
        \begin{bmatrix}
            B_1\\
            B_2\\
            B_3
        \end{bmatrix}
        ,\quad C = 
        \begin{bmatrix}
            C_1 & C_2 & C_3
        \end{bmatrix}
        ,\quad \redB \vcentcolon= \trunc B,\quad \redC = C\trunc^\top,
    \end{equation*}
    $B_1,C_1^\top\in\RR^{\ell\times n}$, $B_2,C_2^\top\in\RR^{(r-\ell)\times n}$, $B_3,C_3^\top\in\RR^{(n-r)\times n}$, satisfying
    \begin{align}
        \label{eq:lyapODEEq_I}
        A\Theta^\top+\Theta A^\top  &= -BB^\top,\\
        \label{eq:lyapODEEq_II}
        A^\top \Gamma+\Gamma^\top A &= -C^\top C.
    \end{align}
    Multiplying \eqref{eq:lyapODEEq_II} from the left by $v^*\trunc$ and from the right by $\trunc^\top v$ yields
    \begin{align*}
        -v^*\redC^\top\redC v &= v^*\redA^\top\redGamma v+v^*\redGamma^\top \redA v = \overline{\lambda}v^*\redE^\top\redGamma v+\lambda v^*\redGamma^\top\redE v\\
        &= 2\Re(\lambda)
        \begin{bmatrix}
            v_1^* & v_2^*
        \end{bmatrix}
        \begin{bmatrix}
            \Gamma_{11} & 0\\
            0           & 0
        \end{bmatrix}
        \underbrace{
        \begin{bmatrix}
            v_1\\
            v_2
        \end{bmatrix}
        }_{\vcentcolon= v} = 2\Re(\lambda)v_1^*\Gamma_{11}v_1.
    \end{align*}
    Since the left-hand side is real and non-positive and since $\Gamma_{11}$ is positive definite, this equation implies $\Re(\lambda)\le 0$ or $v_1=0$.
    If $v_1=0$ held, then $v$ would be in $\Ker(\redE)$ and, thus, due to the identity $\lambda\redE v = \redA v$ also in $\Ker(\redA)$.
    However, since $(\redE,\redA)$ is regular, this is a contradiction and, hence, $\Re(\lambda)\le 0$ must hold.
    
    It remains to show that $(\redE,\redA)$ has no eigenvalues on the imaginary axis, which we prove by contradiction.
    For this purpose, let $i\omega$ with $\omega\in\RR$ be a finite eigenvalue of $(\redE,\redA)$ and let $V\in\CC^{(n+\ell-r)\times k}$ be a matrix with full column rank satisfying $\Ker(i\omega\redE-\redA) = \colspan(V)$.
    Especially, this implies $[A_{31}\;\;A_{33}]V = 0$.
    Furthermore, multiplying \eqref{eq:lyapODEEq_II} from the left by $V^*\trunc$ and from the right by $\trunc^\top V$ yields
    \begin{equation*}
        -V^* \redC^\top \redC V = V^*\redGamma^\top\redA V+V^*\redA^\top\redGamma V = i\omega V^*\redGamma^\top\redE V-i\omega V^*\redE^\top\redGamma V = 0,
    \end{equation*}
    which implies $\redC V = 0$.
    Using this and multiplying \eqref{eq:lyapODEEq_II} from the left by $\trunc$ and from the right by $\trunc^\top V$, we obtain
    \begin{equation*}
        0 = \redGamma^\top\redA V+\redA^\top\redGamma V = i\omega\redGamma^\top\redE V+\redA^\top\redGamma V = (\redA^\top+i\omega\redE^\top)\redGamma V.
    \end{equation*}
    Especially, this identity implies $[A_{13}^\top\;\;A_{33}^\top]\redGamma V = 0$.
    Then, multiplying \eqref{eq:lyapODEEq_I} from the left by $V^*\redGamma^\top\trunc$ and from the right by $\trunc^\top\redGamma V$ yields
    \begin{align*}
        -V^*\redGamma^\top\redB \redB^\top\redGamma V &= V^*\redGamma^\top\redA\redTheta^\top \redGamma V+V^*\redGamma^\top\redTheta\redA^\top\redGamma V\\
        &= i\omega V^*\redGamma^\top\redE\redTheta^\top\redGamma V-i\omega V^*\redGamma^\top\redTheta\redE^\top\redGamma V = 0,
    \end{align*}
    which implies $\redB^\top\redGamma V = 0$.
    Using this and multiplying \eqref{eq:lyapODEEq_I} from the left by $\trunc$ and from the right by $\trunc^\top\redGamma V$, we obtain
    \begin{equation*}
        0 = \redA\redTheta^\top\redGamma V+\redTheta\redA^\top\redGamma V = \redA\redTheta^\top\redGamma V-i\omega\redTheta\redE^\top\redGamma V = (\redA-i\omega\redE)\redTheta^\top\redGamma V.
    \end{equation*}
    Consequently, the columns of $\redTheta^\top\redGamma V$ lie in $\Ker(i\omega\redE-\redA) = \colspan(V)$ and, thus, there exists $\Xi\in\CC^{k\times k}$ satisfying $\redTheta^\top\redGamma V = V\Xi$.
    For the following arguments, we consider the partitioning $V^\top = [V_1^\top\;\;V_2^\top]$ with $V_1\in\RR^{\ell\times k}$ and $V_2\in\RR^{(n-r)\times k}$.
    Especially, we note that the first block row of $\redTheta^\top\redGamma V = V\Xi$ reads $\Theta_{11}\Gamma_{11} V_1 = V_1\Xi$.
    Moreover, $V_1$ must have full column rank since otherwise any non-zero vector $\alpha\in\RR^k$ in $\Ker(V_1)$ would give rise to a corresponding non-zero vector $V\alpha$ lying in $\Ker(\redE)\cap\Ker(\redA)$ which would contradict the regularity of $(\redE,\redA)$.
    Using this fact and $\Theta_{11}\Gamma_{11} V_1 = V_1\Xi$, we infer that $\spectrum(\Xi)\subseteq \spectrum(\Theta_{11}\Gamma_{11})$ holds.
    Furthermore, multiplying \eqref{eq:lyapODEEq_II} from the left by $\Theta_{22}[0\;\; I_{r-\ell}\;\;0]$ and from the right by $\trunc^\top V$ yields
    \begin{align*}
        0 &= \Theta_{22}
        \begin{bmatrix}
            \Gamma_{22}^\top & \Gamma_{32}^\top
        \end{bmatrix}
        \begin{bmatrix}
            A_{21} & A_{23}\\
            A_{31} & A_{33}
        \end{bmatrix}
        V+\Theta_{22}
        \begin{bmatrix}
            A_{12}^\top & A_{32}^\top
        \end{bmatrix}
        \redGamma V\\
        &= \Theta_{22}\Gamma_{22}^\top
        \begin{bmatrix}
            A_{21} & A_{23}
        \end{bmatrix}
        V+\Theta_{22}
        \begin{bmatrix}
            A_{12}^\top & A_{32}^\top
        \end{bmatrix}
        \redGamma V
    \end{align*}
    and, then, multiplying \eqref{eq:lyapODEEq_I} from the left by $[0\;\; I_{r-\ell}\;\;0]$ and from the right by $\trunc^\top \redGamma V$ we obtain
    \begin{align*}
        0 &=
        \begin{bmatrix}
            A_{21} & A_{23}
        \end{bmatrix}
        \redTheta^\top\redGamma V+
        \begin{bmatrix}
            \Theta_{22} & \Theta_{23}
        \end{bmatrix}
        \begin{bmatrix}
            A_{12}^\top & A_{32}^\top\\
            A_{13}^\top & A_{33}^\top
        \end{bmatrix}
        \redGamma V\\
        &= 
        \begin{bmatrix}
            A_{21} & A_{23}
        \end{bmatrix}
        \redTheta^\top\redGamma V+\Theta_{22}
        \begin{bmatrix}
            A_{12}^\top & A_{32}^\top
        \end{bmatrix}
        \redGamma V\\
        &= 
        \begin{bmatrix}
            A_{21} & A_{23}
        \end{bmatrix}
        V\Xi-\Theta_{22}\Gamma_{22}^\top
        \begin{bmatrix}
            A_{21} & A_{23}
        \end{bmatrix}
        V.
    \end{align*}
    Thus, we obtain a homogeneous Sylvester equation for $[A_{21}\;\;A_{23}]V$ and, since $\Xi$ and $\Theta_{22}\Gamma_{22}^\top$ have no common eigenvalues, the solution is unique and given by $[A_{21}\;\;A_{23}]V = 0$.
    Consequently, we have
    \begin{equation*}
        \begin{bmatrix}
            A_{11} & A_{12} & A_{13}\\
            A_{21} & A_{22} & A_{23}\\
            A_{31} & A_{32} & A_{33}
        \end{bmatrix}
        \begin{bmatrix}
            V_1\\
            0\\
            V_2
        \end{bmatrix}
        = i\omega
        \begin{bmatrix}
            I_\ell  & 0             & 0\\
            0       & I_{r-\ell}    & 0\\
            0       & 0             & 0
        \end{bmatrix}
        \begin{bmatrix}
            V_1\\
            0\\
            V_2
        \end{bmatrix}
        ,
    \end{equation*}
    which is a contradiction to the assumption that all finite eigenvalues of $(E,A)$ are in $\CC_-$.
    Thus, in total we obtain that all finite eigenvalues of $(\redE,\redA)$ are in $\CC_-$.
\end{proof}

Let us recall a well-known result from \cite{Liuetal97} which provides a (right) coprime factorization of $G(s)$, i.e., transfer functions $M,N$ such that $G(s)=N(s)M(s)^{-1}.$

\begin{theorem}(\cite{Liuetal97})
    \label{thm:coprimeRealization}
    Consider the system \eqref{eq:pHEQ} with regular pair $(E,A)$, strongly stabilizable triple $(E,A,B)$, and strongly detectable triple $(E,A,C)$, and let $\Pc$ be a stabilizing solution of \eqref{eq:modifiedControlGARE}. 
    Then a normalized right coprime factorization of $G$ is given by
    \begin{align*}
        M(s)&=I_m-B^\top \Pc(sE-\APc)^{-1}B,\\
        N(s)&=C(sE-\APc)^{-1}B,
    \end{align*}
    where $\APc\vcentcolon=A-BB^\top \Pc.$
\end{theorem}

Based on such a factorization, we can specify a pair of generalized Lyapunov (in)equalities associated with the state space system described by $[M^\top\;\; N^\top]^\top$.

\begin{theorem}\label{thm:nrcf_lyap}
    Consider the system \eqref{eq:pHEQ} with regular pair $(E,A)$, strongly stabilizable triple $(E,A,B)$, and strongly detectable triple $(E,A,C)$, and let $\Pc$ and $\Pf$ be stabilizing solutions of \eqref{eq:modifiedGAREs} such that $I_n+\Pf\Pc^\top$ is invertible. 
    Then the matrices $\mathcal{L} \vcentcolon= (I_n+\Pf\Pc^\top)^{-1}\Pf$  and $\Pc$ satisfy
    \begin{subequations}
        \label{eq:gen_lyap_ineq}
        \begin{align}
            \label{eq:gen_lyap_ineq_1}
            \APc \mathcal{L}^\top + \mathcal{L} \APc^\top + BB^\top &\preceq 0, \quad E\mathcal{L}^\top = \mathcal{L}E^\top, \\
            \label{eq:gen_lyap_ineq_2}
            \APc^\top \Pc + \Pc^\top \APc + 
            \begin{bmatrix} 
                -\Pc^\top B & C^\top
            \end{bmatrix} 
            \begin{bmatrix}  
                -\Pc^\top B & C^\top
            \end{bmatrix}
            ^\top  &= 0, \quad  E^\top \Pc = \Pc^\top E
        \end{align}
    \end{subequations}
    with $\APc$ as defined in Theorem~\ref{thm:coprimeRealization}.
\end{theorem}

\begin{proof}
    First, we note that the existence of stabilizing solutions $\Pc$ and $\Pf$ satisfying $I_n+\Pf\Pc^\top\in\GL{n}$ is guaranteed by Theorem~\ref{thm:solutionsOfNonsymmetricGAREs}.
    Furthermore, since \eqref{eq:gen_lyap_ineq_2} is just another way of writing \eqref{eq:modifiedControlGARE}, it is obviously satisfied and thus we only focus on \eqref{eq:gen_lyap_ineq_1}. 
    Due to the invertibility of $I_n+\Pf\Pc^\top$ we may equivalently show that
    \begin{align*}
        (I_n+\Pf\Pc^\top) \APc \Pf^\top+\Pf \APc^\top  (I_n+\Pc\Pf^\top)+ (I_n+\Pf\Pc^\top) BB^\top (I_n+\Pc\Pf^\top)
    \end{align*}
    is negative semidefinite.
    Indeed, we obtain
    \begin{align*}
        & (I_n+\Pf\Pc^\top) \APc \Pf^\top +\Pf \APc^\top  (I_n+\Pc\Pf^\top) + (I_n+\Pf\Pc^\top) BB^\top (I_n+\Pc\Pf^\top) \\
        &= (I_n+\Pf\Pc^\top) A \Pf^\top + \Pf A^\top (I_n+\Pc\Pf^\top) \\ 
        &\quad - (I_n+\Pf\Pc^\top)BB^\top \Pc \Pf^\top - \Pf \Pc^\top BB^\top (I_n+\Pc\Pf^\top)  \\
        &\quad + BB^\top + \Pf \Pc^\top BB^\top+ BB^\top \Pc \Pf^\top + \Pf \Pc^\top BB^\top \Pc \Pf^\top \\
        &= A{\Pf}^\top + {\Pf} A^\top  + BB^\top  \\
        &\quad - \Pf \Pc^\top BB^\top \Pc \Pf^\top+{\Pf}{\Pc}^\top A {\Pf}^\top+ {\Pf}A^\top {\Pc} {\Pf}^\top \\
        &= -2R + \Pf\left( C^\top C + {\Pc}^\top A  +A^\top  \Pc -  {\Pc}^\top BB^\top  \Pc\right) \Pf^\top \\
        &= -2R \preceq 0 .
    \end{align*}
    Since $E^\top \Pc=\Pc^\top E$ and $E\Pf^\top=\Pf E^\top$ it follows that $(I_n+\Pf\Pc^\top ) E \Pf^\top = \Pf E^\top (I_n+\Pc\Pf^\top)$, which, combined with the invertibility of $I_n+\Pc \Pf^\top $, implies $E\mathcal{L}^\top = \mathcal{L}E^\top.$
\end{proof}

By combining the previous results from this subsection, we arrive at one of our main results.

\begin{theorem}\label{thm:err_bnd}
    Let $(\Etilde,\Atilde,\Btilde,\Ctilde)$ be given with $\Pctilde$, $\Pftilde$, and $\Sigma$ as in Theorem \ref{thm:bal_real}. 
    Furthermore, assume that there exists $\ell\in\lbrace 1,\ldots, k-1\rbrace$ such that $\sigma_\ell>\sigma_{\ell+1}$ and 
    consider the truncated system given by 
    \begin{equation*}
        (\redE,\redA,\redB,\redC)=(\trunc\Etilde \trunc^\top,\trunc\Atilde\trunc^\top,\trunc\Btilde,\Ctilde\trunc^\top).
    \end{equation*}
    with $\trunc$ as defined in \eqref{eq:truncationMat}.
    If $(\redE,\redA)$ is regular, then there holds
    \begin{equation}
        \label{eq:errorBoundForCoprimeFactors}
        \left\lVert  
        \begin{bmatrix}
            M\\
            N
        \end{bmatrix}
        -
        \begin{bmatrix}
            \redM\\
            \redN
        \end{bmatrix}
        \right\rVert_{\mathcal{H}_\infty} \le 2 \sum_{i=\ell+1}^k \frac{\sigma_i}{\sqrt{1+\sigma_i^2}},
    \end{equation}
    where the coprime factors $(M,N)$ and $(\redM,\redN)$ are constructed as in Theorem~\ref{thm:coprimeRealization}: $(M,N)$ based on the realization $(\Etilde,\Atilde,\Btilde,\Ctilde)$ and $\Pctilde$, $(\redM,\redN)$ based on $(\redE,\redA,\redB,\redC)$ and $\redPc \vcentcolon= \trunc\Pctilde\trunc^\top$.
\end{theorem}

\begin{proof}
    First of all, we may assume without loss of generality that $\Pctilde$ and $\Pftilde$ are such that $I_n+\Pftilde\Pctilde^\top$ is invertible as in Theorem~\ref{thm:solutionsOfNonsymmetricGAREs}(iv) since the normalized right coprime factors are unique up to multiplication by an orthogonal matrix from the right, cf.~\cite[Lem.~3.1]{Vid84}.
    Thus, for the error bound it is not relevant which of the stabilizing solutions of the GARE \eqref{eq:modifiedControlGARE} is chosen for constructing the coprime factors.
    By Theorem~\ref{thm:coprimeRealization}, a realization of $[M^\top\;\; N^\top]^\top$ is given by $(\Etilde,\APctilde,\Btilde,\CPctilde,\DPctilde)$ with $\APctilde \vcentcolon= \Atilde-\Btilde\Btilde^\top\Pctilde$,
    \begin{equation*}
        \CPctilde \vcentcolon= 
        \begin{bmatrix}
            -\Btilde^\top \Pctilde\\
            \Ctilde
        \end{bmatrix}
        ,\quad \DPctilde \vcentcolon= 
        \begin{bmatrix}
            I_m\\
            0
        \end{bmatrix}
        \in\RR^{2m\times m}.
    \end{equation*}
    Especially, since $\Pctilde$ is a stabilizing solution of the control GARE \eqref{eq:modifiedControlGARE} associated with $(\Etilde,\Atilde,\Btilde,\Ctilde)$, the pencil $(\Etilde,\APctilde)$ is regular and impulse-free and all its finite eigenvalues lie in $\CC_-$.
    Furthermore, by Theorem~\ref{thm:nrcf_lyap}, $\Ltilde \vcentcolon= (I_n+\Pftilde\Pctilde^\top)^{-1}\Pftilde$ and $\Pctilde$ satisfy the generalized Lyapunov (in-)equalities
    \begin{equation}
        \label{eq:lyapIneqsInProof}
        \begin{aligned}
            \APctilde \Ltilde^\top + \Ltilde \APctilde^\top + \Btilde\Btilde^\top &\preceq 0, \quad \Etilde\Ltilde^\top = \Ltilde\Etilde^\top, \\
            \APctilde^\top \Pctilde + \Pctilde^\top \APctilde + \CPctilde^\top\CPctilde &= 0, \quad  \Etilde^\top \Pctilde = \Pctilde^\top \Etilde.
        \end{aligned}
    \end{equation}
    Besides, since the upper left $r\times r$ block of $\Pftilde$ is unique due to the semi-explicit form of $\Etilde$ and Theorem~\ref{thm:solutionsOfNonsymmetricGAREs}(ii), this block equals $\Qtilde_{11}^{-\top}$ by Theorem~\ref{thm:solutionsOfNonsymmetricGAREs}(iii) and Theorem~\ref{thm:bal_real}(ii).
    Consequently, also the upper left $r\times r$ block of $\Ltilde$ is invertible, which allows us to apply Lemma~\ref{lem:gen_lyap_ineq} to the system $(\Etilde,\APctilde,\Btilde,\CPctilde,\DPctilde)$ yielding the error bound \eqref{eq:errorBoundForCoprimeFactors}.
    Note that the feedthrough term $\DPctilde$ is not affected by the truncation and, thus, just cancels out in \eqref{eq:errorBoundForCoprimeFactors}.
    
    It remains to show that $\redM$ and $\redN$ as specified in the claim form indeed a normalized right coprime factorization of the reduced system.
    To this end, we first note that due to \eqref{eq:lyapIneqsInProof}, the pencil $(\Etilde,\APctilde)$ satisfies the requirements of Lemma~\ref{lem:lyapIneqs} and, thus, the reduced pencil $(\redE,\redA-\redB\redB^\top\redPc)$ is regular, impulse-free, and all its finite eigenvalues are in $\CC_-$.
    This shows that the triple $(\redE,\redA,\redB)$ is strongly stabilizable and that $\redPc$ is a stabilizing solution of the control GARE \eqref{eq:modifiedControlGARE} associated with the reduced system given by $(\redE,\redA,\redB,\redC)$.
    Thus, if additionally the triple $(\redE,\redA,\redC)$ is strongly detectable, it follows from Theorem~\ref{thm:coprimeRealization} that $\redM$ and $\redN$ indeed form a normalized right coprime factorization for the reduced system.
    To show the strong detectability of $(\redE,\redA,\redC)$, we observe that, since $\Pftilde$ is a stabilizing solution of the modified filter GARE \eqref{eq:modifiedFilterGARE} associated with $(\Etilde,\Atilde,\Btilde,\Ctilde)$, the pencil $(\Etilde,\APftilde)$ with $\APftilde \vcentcolon= \Atilde-\Pftilde\Ctilde^\top\Ctilde$ is regular and impulse-free and all its finite eigenvalues are in $\CC_-$.
    Furthermore, we note that the construction in the proof of Theorem~\ref{thm:solutionsOfNonsymmetricGAREs}(iv) allows for choosing $\Pftilde$ and $\Pctilde$ such that $I_n+\Pftilde\Pctilde^\top$ and $I_n+\Pftilde^\top\Pctilde$ are invertible.
    Then, the matrices $\Pftilde$ and $\calM \vcentcolon= \Pctilde(I_n+\Pftilde^\top\Pctilde)^{-1}$ satisfy the Lyapunov inequalities
    \begin{align}
        \label{eq:lyapIneqPf_I}
        \APftilde\Pftilde^\top+\Pftilde\APftilde^\top &\preceq 0,\\
        \label{eq:lyapIneqPf_II}
        \APftilde^\top\calM+\calM^\top\APftilde & \preceq 0,
    \end{align}
    where \eqref{eq:lyapIneqPf_I} follows directly from the modified filter GARE \eqref{eq:modifiedFilterGARE} associated with $(\Etilde,\Atilde,\Btilde,\Ctilde)$.
    Furthermore, \eqref{eq:lyapIneqPf_II} is a consequence of the calculation
    \begin{align*}
        &(I_n+\Pftilde^\top\Pctilde)^\top\APftilde^\top\Pctilde+\Pctilde^\top\APftilde(I_n+\Pftilde^\top\Pctilde)\\
        &= \Atilde^\top\Pctilde+\Pctilde^\top \Atilde-\Ctilde^\top\Ctilde\Pftilde^\top\Pctilde-\Pctilde^\top\Pftilde\Ctilde^\top\Ctilde+\Pctilde^\top\Pftilde \Atilde^\top\Pctilde+\Pctilde^\top \Atilde\Pftilde^\top\Pctilde\\
        &\quad-2\Pctilde^\top\Pftilde\Ctilde^\top\Ctilde\Pftilde^\top\Pctilde\\
        &= \Pctilde^\top\left(\Btilde\Btilde^\top+\Pftilde \Atilde^\top+\Atilde\Pftilde^\top-\Pftilde\Ctilde^\top\Ctilde\Pftilde^\top\right)\Pctilde\\
        &\quad-(I_n+\Pftilde^\top\Pctilde)^\top\Ctilde^\top\Ctilde(I_n+\Pftilde^\top\Pctilde)\\
        &\preceq -2\Pctilde^\top\Rtilde\Pctilde \preceq 0.
    \end{align*}
    Thus, the pencil $(\Etilde,\APftilde)$ satisfies the requirements of Lemma~\ref{lem:lyapIneqs} and, thus, the reduced pencil $(\redE,\redA-\redPf\redC^\top\redC)$ with $\redPf \vcentcolon= \trunc \Pftilde\trunc^\top$ is regular, impulse-free, and all its finite eigenvalues are in $\CC_-$.
    Hence, the triple $(\redE,\redA,\redC)$ is strongly detectable, which concludes the proof.
\end{proof}

Let us emphasize that the requirement in Theorem~\ref{thm:err_bnd} that the reduced pencil $(\redE,\redA)$ is regular is in general not automatically satisfied, as Example~\ref{ex:pHDAEMOR} demonstrates.
In fact, this is also true for the unstructured LQG balanced truncation approach presented in \cite{MoeRS11} as is shown in Example~\ref{ex:MoeRS11}.
However, for the special case that the original pencil $(\Etilde,\Atilde)$ is regular and impulse-free, the requirement that $(\redE,\redA)$ is regular is automatically satisfied, cf.~the beginning of the proof of Lemma~\ref{lem:lyapIneqs}.

\begin{example}
    \label{ex:pHDAEMOR}
    We consider a port-Hamiltonian descriptor system as in \eqref{eq:pHEQ} with
    \begin{align*}
        E &= 
        \begin{bmatrix}
            1 & 0 & 0\\
            0 & 1 & 0\\
            0 & 0 & 0
        \end{bmatrix}
        ,\quad J = 
        \begin{bmatrix}
            0                       & \frac{a^2+b^2}{b^2-a^2}   & 0\\
            \frac{a^2+b^2}{a^2-b^2} & 0                         & 1\\
            0                       & -1                        & 0
        \end{bmatrix}
        ,\quad R = 
        \begin{bmatrix}
            1 & 1 & 0\\
            1 & 2 & 0\\
            0 & 0 & 0
        \end{bmatrix}
        ,\\ 
        Q &= 
        \begin{bmatrix}
            a & 0           & 0\\
            0 & b           & 0\\
            0 & b-\frac1b   & 1
        \end{bmatrix}
        ,\quad A = (J-R)Q = 
        \begin{bmatrix}
            -a                      & \frac{2ba^2}{a^2-b^2}  & 0\\
            \frac{2ab^2}{b^2-a^2}    & -(b+\frac1b)          & 1\\
            0                       & -b                    & 0
        \end{bmatrix}
        ,\\ 
        B &= 
        \begin{bmatrix}
            1 & 0 & 0\\
            0 & 0 & 1
        \end{bmatrix}
        ^\top,\quad C = B^\top Q = 
        \begin{bmatrix}
            a & 0 & 0\\
            0 & b-\frac1b   & 1
        \end{bmatrix}
        ,
    \end{align*}
    $a = \sqrt{1+\sqrt2}$, and $b = \sqrt{2+\sqrt5}$.
    Especially, we note that $(E,A)$ is regular but not impulse-free, $(E,A,B)$ is strongly stabilizable, and $(E,A,C)$ is strongly detectable.
    Furthermore, the realization is balanced with respect to the modified GAREs \eqref{eq:modifiedGAREs}, since stabilizing solutions are given by
    \begin{equation*}
        \Pc = 
        \begin{bmatrix}
            \frac1a & 0         & 0\\
            0       & \frac1b   & 0\\
            0       & 0         & 1
        \end{bmatrix}
        ,\quad
        \Pf = Q^{-\top} = 
        \begin{bmatrix}
            \frac1a & 0         & 0\\
            0       & \frac1b   & \frac1{b^2}-1\\
            0       & 0         & 1
        \end{bmatrix}
        .
    \end{equation*}
    After truncating the state corresponding to the characteristic value $\frac1b$, the pencil $(\redE,\redA)$ of the reduced-order model is given by
    \begin{equation*}
        \left(
        \begin{bmatrix}
            1 & 0\\
            0 & 0
        \end{bmatrix}
        , 
        \begin{bmatrix}
            -a  & 0\\
            0   & 0
        \end{bmatrix}
        \right)
    \end{equation*}
    and, thus, singular, since $\redE$ and $\redA$ have a non-trivial common null space.
\end{example}

\begin{example}
    \label{ex:MoeRS11}
    We consider a descriptor system given by $(E,A,B,C)$ with
    \begin{equation*}
        E = 
        \begin{bmatrix}
            1 & 0 & 0\\
            0 & 1 & 0\\
            0 & 0 & 0
        \end{bmatrix}
        ,\quad A = 
        \begin{bmatrix}
            -1  & 0     & 0\\
            0   & -1    & 1\\
            0   & 1     & 0
        \end{bmatrix}
        ,\quad B = 
        \begin{bmatrix}
            1 & 0                   & 0\\
            0 & \frac{\sqrt{3}}2    & 0\\
            0 & -\frac{4}{\sqrt{3}} & 1
        \end{bmatrix}
        =C^\top
    \end{equation*}
    and note that $(E,A)$ is regular but not impulse-free, $(E,A,B)$ is strongly stabilizable, and $(E,A,C)$ is strongly detectable.
    Furthermore, the realization is balanced with respect to the GAREs \eqref{eq:originalGAREs} considered in \cite{MoeRS11}, since stabilizing solutions are given by
    \begin{equation*}
        \Pc = \Pf = 
        \begin{bmatrix}
            \sqrt{2}-1  & 0         & 0\\
            0           & \frac13   & 0\\
            0           & 0         & 1
        \end{bmatrix}
        .
    \end{equation*}
    After truncating the state corresponding to the LQG characteristic value $\frac13$, the pencil $(\redE,\redA)$ of the reduced-order model is given by
    \begin{equation*}
        \left(
        \begin{bmatrix}
            1 & 0\\
            0 & 0
        \end{bmatrix}
        , 
        \begin{bmatrix}
            -1  & 0\\
            0   & 0
        \end{bmatrix}
        \right)
    \end{equation*}
    and, thus, singular.
\end{example}

\subsection{Improvement of the Error Bound by Replacing $Q$}\label{subsec:optimalQ}

In what follows, we investigate the possibility of exchanging the $Q$ matrix in the pH realization \eqref{eq:pHEQ} with the intention of improving, i.e., reducing $\sigma_i$ in the error bound \eqref{eq:errorBoundForCoprimeFactors}, cf.~\cite[sec.~4.3]{BreMS21} for a similar discussion in the ODE case.
More precisely, let us assume that the KYP LMI \eqref{eq:generalizedKYP_II} has a maximal solution $\Xmax$ which we can use instead of $Q$ for the formulation of a port-Hamiltonian system. Then, we have $E^\top \Xmax\succeq E^\top Q$ and this property is also preserved after transformation to a semi-explicit realization, see \eqref{eq:transformedETQ}.
So without loss of generality we may assume that we already started with a semi-explicit realization.
Then, we have
\begin{align*}
    \sigma_i^2 &= \lambda_i\left(
    \begin{bmatrix}
        I_r & 0
    \end{bmatrix}
    \Pf E^\top \Pc^\top E
    \begin{bmatrix}
        I_r\\
        0
    \end{bmatrix}
    \right) = \lambda_i(\mathcal{P}_{\mathrm{f}11}\mathcal{P}_{\mathrm{c}11}) = \lambda_i(\mathcal{P}_{\mathrm{f}11}\mathcal{P}_{\mathrm{c}11}^{\frac12}\mathcal{P}_{\mathrm{c}11}^{\frac12})\\
    &= \lambda_i(\mathcal{P}_{\mathrm{c}11}^{\frac12}\mathcal{P}_{\mathrm{f}11}\mathcal{P}_{\mathrm{c}11}^{\frac12}) = \lambda_i(\mathcal{P}_{\mathrm{c}11}^{\frac12}Q_{11}^{-\top}\mathcal{P}_{\mathrm{c}11}^{\frac12}) \geq \lambda_i(\mathcal{P}_{\mathrm{c}11}^{\frac12}X_{11}^{-\top}\mathcal{P}_{\mathrm{c}11}^{\frac12})\\
    &= \lambda_i(X_{11}^{-\top}\mathcal{P}_{\mathrm{c}11}) = \lambda_i\left(
    \begin{bmatrix}
        I_r & 0
    \end{bmatrix}
    \Xmax^{-\top}E^\top \Pc^\top E
    \begin{bmatrix}
        I_r\\
        0
    \end{bmatrix}
    \right) =\vcentcolon \widehat{\sigma}_i^2
\end{align*}
for every non-zero eigenvalue $\lambda_i$.
Here, $X_{11}$ denotes the upper left $r\times r$ block of $\Xmax$ and we exploited the semi-explicit form of $E$, the resulting symmetry and positive semidefiniteness of $\mathcal{P}_{\mathrm{c}11}$, Theorem~\ref{thm:solutionsOfNonsymmetricGAREs}(ii) and (iii), and the fact that the non-zero eigenvalues of $AB$ and $BA$ coincide for all matrices $A$ and $B$ whose dimensions allow for the matrix products to be well-defined. 
Since the function $f\colon\RR_{\geq 0}\to\RR$ with $f(x)\vcentcolon=\frac{x}{1+x}$ is strictly monotonously increasing, the property $\sigma_i^2\geq \widehat{\sigma}_i^2$ also implies $\frac{\sigma_i^2}{1+\sigma_i^2}\geq \frac{\widehat{\sigma}_i^2}{1+\widehat{\sigma}_i^2}$ which in turn implies that the error bound \eqref{eq:errorBoundForCoprimeFactors} is minimized when $\Xmax$ is used for defining the Hamiltonian. 

For deriving a procedure to determine $\Xmax$, we distinguish two cases and begin with index-1 or impulse-free systems before turning to the general setting by means of feedback equivalence transformations. 

\subsubsection{Procedure for Index-1 Systems}\label{subsubsec:indexOne}

In the following, we assume that \eqref{eq:pHEQ} is regular and has index 1. 
Thus, there exists a state space transformation which transforms $(E,A)$ to Weierstraß canonical form, i.e., there exist $S,T\in\GL{n}$ with
\begin{equation}
    \label{eq:quasiWCFmatrices_indexOneCase}
    \begin{aligned}
        \Etilde &= SET^{-1} = 
        \begin{bmatrix}
            I_r & 0\\
            0   & 0
        \end{bmatrix}
        ,\quad \Atilde = SAT^{-1} = 
        \begin{bmatrix}
            A_{11}  & 0\\
            0       & I_{n-r}
        \end{bmatrix}
        ,\quad \Btilde = SB =
        \begin{bmatrix}
            B_1\\
            B_2
        \end{bmatrix}
        ,\\
        \Ctilde &= CT^{-1} = 
        \begin{bmatrix}
            C_1 & C_2
        \end{bmatrix}
    \end{aligned}
\end{equation}
with $A_{11}\in\RR^{r\times r}$, $B_1,C_1^\top\in\RR^{r\times m}$, and $B_2,C_2^\top\in\RR^{(n-r)\times m}$, cf.~section~\ref{subsec:matrixPencils}.
Especially, $S$ and $T$ can be chosen such that $A_{11}$ is in Jordan canonical form (JCF), but for the following considerations this is not necessary and, thus, we do not enforce $A_{11}$ to be in JCF.
This freedom in choosing $S$ and $T$ is in particular helpful for the numerical examples in section~\ref{sec:numericalExamples}.

As mentioned in section~\ref{subsec:pHDAEs}, the transformed system is still pH with $\Atilde = (\widetilde{J}-\Rtilde)\Qtilde$.
Furthermore, due to the semi-explicit nature of $\Etilde$ we infer that $\Qtilde$ has the block structure $\Qtilde = \left[\begin{smallmatrix} Q_{11} & 0 \\ Q_{21} & Q_{22} \end{smallmatrix}\right]$ with $Q_{11}=Q_{11}^\top \succeq 0.$
Thus, we have $C_1 = B_1^\top Q_{11}+B_2^\top Q_{21}$, $C_2 = B_2^\top Q_{22}$, and
\begin{equation}
    \label{eq:WCFpH}
    \begin{bmatrix}
        A_{11}  & 0\\
        0       & I_{n-r}
    \end{bmatrix}
    =
    \begin{bmatrix}
        (J_{11}-R_{11})Q_{11}+(J_{12}-R_{12})Q_{21} 					& (J_{12}-R_{12})Q_{22}\\
        (-J_{12}^\top-R_{12}^\top)Q_{11}+(J_{22}-R_{22})Q_{21}  & (J_{22}-R_{22})Q_{22}
    \end{bmatrix}
    .
\end{equation}
In particular, the lower right block implies $J_{22}-R_{22},Q_{22}\in\GL{n-r}$ and $Q_{22} = (J_{22}-R_{22})^{-1}$.
Moreover, the invertibility of $Q_{22}$ leads to $J_{12} = R_{12}$ and $A_{11} = (J_{11}-R_{11})Q_{11}$.
In addition, the relation $Q_{22} = (J_{22}-R_{22})^{-1}$ implies
\begin{align*}
	&x^\top (Q_{22}+Q_{22}^\top)x = x^\top ((J_{22}-R_{22})^{-1}+(J_{22}-R_{22})^{-\top})x\\
	&= \underbrace{x^\top (J_{22}-R_{22})^{-\top}}_{= y^\top} (J_{22}-R_{22})^{\top} \underbrace{(J_{22}-R_{22})^{-1} x}_{=\vcentcolon y}\\ &\quad +\underbrace{x^\top(J_{22}-R_{22})^{-\top}}_{= y^\top}(J_{22}-R_{22})\underbrace{(J_{22}-R_{22})^{-1} x}_{= y}\\
	&= y^\top (J_{22}^\top-R_{22}^\top+J_{22}-R_{22})y = -2y^\top R_{22} y \leq 0.
\end{align*}
Consequently, there exists a skew-symmetric matrix $J_Q$ and a symmetric positive semidefinite matrix $R_Q$ satisfying $Q_{22} = J_Q-R_Q$.

Since the transformed system is still pH, $\Qtilde$ is a solution of the KYP LMI
\begin{equation*}
	\begin{bmatrix}
		-\Atilde^\top \Xtilde-\Xtilde^\top \Atilde & \Ctilde^\top-\Xtilde^\top \Btilde\\
		\Ctilde-\Btilde^\top \Xtilde 						& 0
	\end{bmatrix}
	\succeq 0,\quad \Etilde^\top \Xtilde = \Xtilde^\top \Etilde,
\end{equation*}
cf.~section~\ref{subsec:generalizedKYPLMI}.
Here, $\Etilde^\top \Xtilde = \Xtilde^\top \Etilde$ is satisfied if and only if $X_{11}=X_{11}^\top$ and $X_{12}=0$, where $X_{11}\in\RR^{r\times r}$ and $X_{12}\in\RR^{r\times(n-r)}$ denote the upper left and right block of $\Xtilde$, respectively.
Furthermore, the matrix inequality implies
\begin{equation}
	\label{eq:reducedKYPLMI_index1}
	\begin{aligned}
		&
		\begin{bmatrix}
			I_r & 0\\
			0 	& -B_2\\
			0	& I_m
		\end{bmatrix}
		^\top
		\begin{bmatrix}
			-\Atilde^\top \Xtilde-\Xtilde^\top \Atilde & \Ctilde^\top-\Xtilde^\top \Btilde\\
			\Ctilde-\Btilde^\top \Xtilde 						& 0
		\end{bmatrix}
		\begin{bmatrix}
			I_r & 0\\
			0 	& -B_2\\
			0	& I_m
		\end{bmatrix}
		\\
		&=
		\begin{bmatrix}
			-A_{11}^\top X_{11}-X_{11}A_{11} & C_1^\top-X_{11}B_1\\
			C_1-B_1^\top X_{11} & -C_2B_2-B_2^\top C_2^\top
		\end{bmatrix}
		\succeq 0.
	\end{aligned}
\end{equation}
In particular, this matrix inequality is satisfied by $Q_{11}$.
Furthermore, the following lemma yields sufficient conditions for the existence of a maximal solution of this KYP LMI.

\begin{lemma}
    \label{lem:existenceMaximalSolutionOfReducedKYPLMI_index1}
    Consider the system \eqref{eq:pHEQ} with regular and impulse-free pair $(E,A)$ and strongly anti-stabilizable triple $(E,A,B)$.
    Furthermore, consider the corresponding transformed system $(\Etilde,\Atilde,\Btilde,\Ctilde)$ as in \eqref{eq:quasiWCFmatrices_indexOneCase} as well as the associated reduced KYP LMI \eqref{eq:reducedKYPLMI_index1}.
    Then, there exists a maximizing solution $\Xmax=\Xmax^\top$ of \eqref{eq:reducedKYPLMI_index1} satisfying $\Xmax\succeq X_{11}$ for all $X_{11}=X_{11}^\top$ solving \eqref{eq:reducedKYPLMI_index1}.
\end{lemma}

\begin{proof}
    The proof follows along the lines of the proof of Lemma~\ref{lem:existenceMaximalSolutionOfReducedKYPLMI} and is therefore omitted.
\end{proof}

The following theorem demonstrates how to obtain a port-Hamiltonian realization based on a maximal solution of \eqref{eq:reducedKYPLMI_index1}.

\begin{theorem}
    \label{thm:optimalpHRealizationOfWCF_indexOne}
    Consider the system \eqref{eq:pHEQ} with regular and impulse-free pair $(E,A)$, strongly stabilizable triple $(E,A,B)$, and strongly detectable triple $(E,A,C)$.
    Furthermore, consider the corresponding transformed system given by $(\Etilde,\Atilde,\Btilde,\Ctilde)$ as in \eqref{eq:quasiWCFmatrices_indexOneCase} and let there exist a maximal solution $\Xmax$ to the associated reduced KYP LMI \eqref{eq:reducedKYPLMI_index1} in the sense of Lemma~\ref{lem:existenceMaximalSolutionOfReducedKYPLMI_index1}.
	Besides, let 
	\begin{equation*}
		B_2^\top = 
		\begin{bmatrix}
			U_1 & U_2
		\end{bmatrix}
		\begin{bmatrix}
			\Sigma 	& 0\\
			0 			& 0
		\end{bmatrix}
		\begin{bmatrix}
			V_1^\top\\
			V_2^\top
		\end{bmatrix}
	\end{equation*}
	with $\rank(\Sigma) = \rank(B_2) =\vcentcolon \widetilde{r}$ be a singular value decomposition of $B_2^\top$.
	Moreover, let us define
	\begin{align*}
		X_{21} &\vcentcolon= V_1V_1^\top Q_{21}+(B_2^\top)^+B_1^\top (Q_{11}-\Xmax),\\ X_{22}& \vcentcolon= V_1V_1^\top Q_{22}-V_2V_2^\top Q_{22}^\top V_1V_1^{\top}+V_2KV_2^\top,\\
		 \Jhat_{11} &\vcentcolon= \frac12\left(A_{11}\Xmax^{-1}-\Xmax^{-1}A_{11}^\top\right),\quad \Rhat_{11} \vcentcolon= -\frac12\left(A_{11}\Xmax^{-1}+\Xmax^{-1}A_{11}^\top\right)
	\end{align*}
	with some arbitrary matrix $K\in\RR^{(n-r-\widetilde{r})\times(n-r-\widetilde{r})}$ satisfying $K+K^\top\prec 0$.
	Here, $(B_2^\top)^+ = V_1\Sigma^{-1}U_1^\top$ denotes the Moore--Penrose inverse of $B_2^\top$.
	
	Then, the following assertions hold.
	\begin{enumerate}[(i)]
		\item The matrix $X_{22}$ is invertible.
		\item The matrices $\Atilde$ and $\Ctilde$ can be written as
			\begin{equation*}
				\Atilde = 
				\begin{bmatrix}
        				A_{11}  & 0\\
        				0       & I_{n-r}
    				\end{bmatrix}
    				= \Bigg(\underbrace{
    				\begin{bmatrix}
    					\Jhat_{11} & \Jhat_{12}\\
    					\Jhat_{21} & \Jhat_{22}
    				\end{bmatrix}
    				}_{=\vcentcolon \Jhat} -\underbrace{
    				\begin{bmatrix}
    					\Rhat_{11} & \Rhat_{12}\\
    					\Rhat_{21} & \Rhat_{22}
    				\end{bmatrix}
    				}_{=\vcentcolon \Rhat}\Bigg) \underbrace{
    				\begin{bmatrix}
    					\Xmax & 0\\
    					X_{21} 					& X_{22}
    				\end{bmatrix}
    				}_{=\vcentcolon \Xhat}
			\end{equation*}
			and $\Ctilde = \Btilde^\top \Xhat$, where $\Jhat_{12},\Jhat_{21},\Jhat_{22},\Rhat_{12},\Rhat_{21},\Rhat_{22}$ are defined via
			\begin{equation*}
			  \begin{aligned}
				\Jhat_{12} &= -\Jhat_{21}^\top = \Rhat_{12} = \Rhat_{21}^\top \vcentcolon= \frac12 \Xmax^{-1}X_{21}^\top X_{22}^{-\top},\\ \Jhat_{22}& \vcentcolon= \frac12\left(X_{22}^{-1}-X_{22}^{-\top}\right),\quad \Rhat_{22} \vcentcolon= -\frac12\left(X_{22}^{-1}+X_{22}^{-\top}\right).
			  \end{aligned}
			\end{equation*}
			Furthermore, $\Jhat$ is skew-symmetric and $\Rhat$ is symmetric and positive semidefinite.
	\end{enumerate}
\end{theorem}

\begin{proof}
    First, we emphasize that the transformation to \eqref{eq:quasiWCFmatrices_indexOneCase} is well-defined due to the impulse-freeness of $(E,A)$.
    Furthermore, since $Q_{11}$ is a solution of \eqref{eq:reducedKYPLMI_index1} and positive definite due to the invertibility of $Q$, the invertibility of $\Xmax\succeq Q_{11}$ is guaranteed.
   \begin{enumerate}[(i)]
		\item Let $\alpha\in\RR^{n-r}$ be in the kernel of $X_{22}$.
			Furthermore, define $\beta \vcentcolon= V_1^\top \alpha\in\RR^{\widetilde{r}}$ and $\gamma \vcentcolon= V_2^\top \alpha\in\RR^{n-r-\widetilde{r}}$ yielding $\alpha = V_1\beta+V_2\gamma$.
			Since $\alpha$ is in the kernel of $X_{22}$, we have
			\begin{equation*}
				V_1V_1^\top Q_{22}\alpha-V_2V_2^\top Q_{22}^\top V_1V_1^{\top}\alpha+V_2KV_2^\top\alpha = 0
			\end{equation*}
			which is equivalent to
			\begin{align*}
				0 &= V_1^\top Q_{22}\alpha = V_1^\top Q_{22}\left(V_1\beta+V_2\gamma\right),\\
				0 &= KV_2^\top\alpha-V_2^\top Q_{22}^\top V_1V_1^{\top}\alpha = K\gamma-V_2^\top Q_{22}^\top V_1\beta
			\end{align*}
			or in matrix notation
			\begin{equation}
				\label{eq:linearSystemForProvingInvertibility}
				\begin{bmatrix}
					V_1^\top Q_{22}V_1 				& V_1^\top Q_{22}V_2\\
					-V_2^\top Q_{22}^\top V_1 	& K
				\end{bmatrix}
				\begin{bmatrix}
					\beta\\
					\gamma
				\end{bmatrix}
				= 0.
			\end{equation}
			Multiplying this equation from the left by $[\beta^\top\;\; \gamma^\top]$ yields
			\begin{equation*}
				0 = \beta^\top V_1^\top Q_{22}V_1\beta+\gamma^\top K\gamma = -\beta^\top V_1^\top R_QV_1\beta+\frac12\gamma^\top (K+K^\top)\gamma,
			\end{equation*}
			where we used $Q_{22} = J_Q-R_Q$, see the beginning of section~\ref{subsubsec:indexOne}.
			Since $K+K^\top \prec 0$ and since $R_Q \succeq 0$, this equation implies $\gamma=0$ and $R_QV_1\beta = 0$.
			Using these findings and \eqref{eq:linearSystemForProvingInvertibility} gives us $V^\top J_QV_1\beta = 0$ and, hence, $J_QV_1\beta=0$ and, thus, altogether $Q_{22}V_1\beta = 0$.
			Since $Q_{22}$ is invertible and since $V_1$ has full column rank, it follows $\beta=0$ and, hence, altogether $\alpha = 0$.
			Consequently, since its kernel consists just of the zero vector, $X_{22}$ is invertible.
		\item Straightforward calculations yield
			\begin{align*}
				&	
				\left(
    				\begin{bmatrix}
    					\Jhat_{11} & \Jhat_{12}\\
    					\Jhat_{21} & \Jhat_{22}
    				\end{bmatrix}
    				-
    				\begin{bmatrix}
    					\Rhat_{11} & \Rhat_{12}\\
    					\Rhat_{21} & \Rhat_{22}
    				\end{bmatrix}
    				\right)
    				\begin{bmatrix}
    					\Xmax & 0\\
    					X_{21} 					& X_{22}
    				\end{bmatrix}
    				\\
    				&=
    				\begin{bmatrix}
    					A_{11} & 0\\
    					-X_{22}^{-1}X_{21}\Xmax^{-1} \Xmax+X_{22}^{-1}X_{21} & I_{n-r}
    				\end{bmatrix}
    				= 
    				\begin{bmatrix}
    					A_{11} & 0\\
    					0 & I_{n-r}
    				\end{bmatrix}
    				= \Atilde.
			\end{align*}
			Before we show the relation $\Ctilde = \Btilde^\top \Xhat$, let us once more consider the KYP LMI \eqref{eq:reducedKYPLMI_index1} which in particular implies that for all $x_1\in\RR^r$ and $x_2\in\Ker(B_2)=\colspan(U_2)\subseteq \RR^m$ we have
			\begin{align*}
				0 &\leq 
				\begin{bmatrix}
					x_1^\top & x_2^\top
				\end{bmatrix}
				\begin{bmatrix}
					-A_{11}^\top \Xmax-\Xmax A_{11} & C_1^\top-\Xmax B_1\\
					C_1-B_1^\top \Xmax & -C_2B_2-B_2^\top C_2^\top
				\end{bmatrix}
				\begin{bmatrix}
					x_1\\
					x_2
				\end{bmatrix}
				\\
				&= -2x_1^\top\Xmax A_{11}x_1+2x_1^\top(C_1^\top-\Xmax B_1)x_2-2x_2^\top C_2B_2x_2\\
				&= -2x_1^\top\Xmax A_{11}x_1+2x_1^\top(Q_{11}B_1+Q_{21}^\top B_2-\Xmax B_1)x_2\\
				&= -2x_1^\top\Xmax A_{11}x_1+2x_1^\top(Q_{11}-\Xmax )B_1x_2.
			\end{align*}
			Especially, this can only hold for all $x_1\in\RR^r$ if $(Q_{11}-\Xmax )B_1x_2=0$.
			Thus, we showed $\Ker(B_2)=\colspan(U_2)\subseteq \Ker((Q_{11}-\Xmax )B_1)$.
			Using this, we obtain
			\begin{align*}
				&\Btilde^\top \Xhat =
				\begin{bmatrix}
					B_1^\top & B_2^\top
				\end{bmatrix}
				\begin{bmatrix}
					\Xmax & 0\\
    					X_{21} 					& X_{22}
				\end{bmatrix}
				= 
				\begin{bmatrix}
					B_1^\top \Xmax+B_2^\top X_{21} & B_2^\top X_{22}
				\end{bmatrix}
				\\
				&=
				\begin{bmatrix}
					B_1^\top \Xmax+B_2^\top V_1V_1^\top Q_{21}+B_2^\top(B_2^\top)^+B_1^\top (Q_{11}-\Xmax) & B_2^\top X_{22}
				\end{bmatrix}
				\\
				&= 
				\begin{bmatrix}
					B_1^\top \Xmax+B_2^\top Q_{21}+U_1U_1^\top B_1^\top (Q_{11}-\Xmax) & B_2^\top X_{22}
				\end{bmatrix}
				\\
				&=
				\begin{bmatrix}
					B_1^\top \Xmax+B_2^\top Q_{21}+(I_m-U_2U_2^\top)B_1^\top (Q_{11}-\Xmax) & B_2^\top X_{22}
				\end{bmatrix}
				\\
				&=
				\begin{bmatrix}
					B_1^\top \Xmax+B_2^\top Q_{21}+B_1^\top (Q_{11}-\Xmax) & B_2^\top Q_{22}
				\end{bmatrix}
				\\
				&=
				\begin{bmatrix}
					B_2^\top Q_{21}+B_1^\top Q_{11} & B_2^\top \left(V_1V_1^\top Q_{22}-V_2V_2^\top Q_{22}^\top V_1V_1^{\top}+V_2KV_2^\top\right)
				\end{bmatrix}
				\\
				&=
				\begin{bmatrix}
					C_1 & B_2^\top Q_{22}
				\end{bmatrix}
				=
				\begin{bmatrix}
					C_1 & C_2
				\end{bmatrix}
				= \Ctilde.
			\end{align*}
			It remains to show the properties of $\Jhat$ and $\Rhat$.
			The skew-symmetry of $\Jhat$ and the symmetry of $\Rhat$ follow directly from the definition.
			To show the positive semidefiniteness of $\Rhat$ we first observe the identity
			\begin{align*}
				&2\Rhat = 2
				\begin{bmatrix}
					\Rhat_{11} & \Rhat_{12}\\
    					\Rhat_{21} & \Rhat_{22}
				\end{bmatrix}
				=
				\begin{bmatrix}
					-\left(A_{11}\Xmax^{-1}+\Xmax^{-1}A_{11}^\top\right)    & \Xmax^{-1}X_{21}^\top X_{22}^{-\top}\\
					X_{22}^{-1}X_{21}\Xmax^{-1}                             & \left(X_{22}^{-1}+X_{22}^{-\top}\right)
				\end{bmatrix}
				\\
				&= -
				\begin{bmatrix}
					\Xmax^{-1} 	& 0\\
					0 									& X_{22}^{-1}
				\end{bmatrix}
				\underbrace{
				\begin{bmatrix}
					\Xmax A_{11}+A_{11}^\top \Xmax 	& -X_{21}^\top\\
					-X_{21}																				& X_{22}+X_{22}^\top
				\end{bmatrix}
				}_{=\vcentcolon \widehat{K}}
				\begin{bmatrix}
					\Xmax^{-1} 	& 0\\
					0 									& X_{22}^{-\top}
				\end{bmatrix}
				.
			\end{align*}
			Consequently, $\Rhat\succeq0$ is equivalent to $\widehat{K}\preceq0$.
			To show that $\widehat{K}\preceq0$, let $x_1\in\RR^r$ and $x_2\in\RR^{n-r}$ be arbitrary.
			Furthermore, let $\beta \in\RR^{\widetilde{r}}$ and $\gamma\in\RR^{n-r-\widetilde{r}}$ be such that $x_2 = V_1\beta+V_2\gamma$.
			Then, there holds
			\begin{align*}
				&
				\begin{bmatrix}
					x_1^\top & \beta^\top V_1^\top+\gamma^\top V_2^\top
				\end{bmatrix}
				\begin{bmatrix}
					\Xmax A_{11}+A_{11}^\top \Xmax 	& -X_{21}^\top\\
					-X_{21}																				& X_{22}+X_{22}^\top
				\end{bmatrix}
				\begin{bmatrix}
					x_1\\
					V_1\beta+V_2\gamma
				\end{bmatrix}
				\\
				&= x_1^\top (\Xmax A_{11}+A_{11}^\top \Xmax)x_1-2(\beta^\top V_1^\top+\gamma^\top V_2^\top) X_{21}x_1\\ &\quad +(\beta^\top V_1^\top+\gamma^\top V_2^\top)(X_{22}+X_{22}^\top)(V_1\beta+V_2\gamma)\\
				&= 2x_1^\top\Xmax A_{11}x_1+2(\beta^\top V_1^\top+\gamma^\top V_2^\top)X_{22}(V_1\beta+V_2\gamma)\\
				&\quad-2(\beta^\top V_1^\top+\gamma^\top V_2^\top)\left(V_1V_1^\top Q_{21}+(B_2^\top)^+B_1^\top (Q_{11}-\Xmax)\right)x_1\\
				&= 2x_1^\top\Xmax A_{11}x_1-2\beta^\top V_1^\top\left(Q_{21}+(B_2^\top)^+B_1^\top (Q_{11}-\Xmax) \right)x_1\\
				&\quad+2\left(\beta^\top V_1^\top Q_{22}(V_1\beta+V_2\gamma)-\gamma^\top V_2^\top Q_{22}^\top V_1\beta+\gamma^\top K\gamma\right)\\
				&= 2x_1^\top\Xmax A_{11}x_1-2\beta^\top \Sigma^{-1}U_1^\top B_2^\top\left(Q_{21}+(B_2^\top)^+B_1^\top (Q_{11}-\Xmax) \right)x_1\\
				&\quad+2\left(\beta^\top \Sigma^{-1}U_1^\top B_2^\top Q_{22} B_2U_1\Sigma^{-1}\beta+\gamma^\top K\gamma\right)\\
				&= 2x_1^\top\Xmax A_{11}x_1-2\beta^\top \Sigma^{-1}U_1^\top\left(B_2^\top Q_{21}+B_1^\top (Q_{11}-\Xmax) \right)x_1\\
				&\quad+\beta^\top \Sigma^{-1}U_1^\top B_2^\top (Q_{22}+Q_{22}^\top) B_2U_1\Sigma^{-1}\beta+\gamma^\top (K+K^\top)\gamma\\
				&\leq x_1^\top (\Xmax A_{11}+A_{11}^\top \Xmax)x_1-2\beta^\top \Sigma^{-1}U_1^\top\left(C_1-B_1^\top \Xmax \right)x_1\\
				&\quad+\beta^\top \Sigma^{-1}U_1^\top (C_2B_2+B_2^\top C_2^\top)U_1\Sigma^{-1}\beta\\
				&=
				\begin{bmatrix}
					x_1\\
					U_1\Sigma^{-1} \beta
				\end{bmatrix}
				^\top
				\begin{bmatrix}
					\Xmax A_{11}+A_{11}^\top \Xmax 	& \Xmax B_1-C_1^\top \\
					B_1^\top \Xmax-C_1              & C_2B_2+B_2^\top C_2^\top
				\end{bmatrix}
				\begin{bmatrix}
					x_1\\
					U_1\Sigma^{-1} \beta
				\end{bmatrix}
				\\
				&\leq 0.
			\end{align*}
			Consequently, $\Rhat$ is positive semidefinite which concludes the proof.
	\end{enumerate}
\end{proof}

Since $Q_{11}$ is a solution of \eqref{eq:reducedKYPLMI_index1}, any maximal solution $\Xmax$ of \eqref{eq:reducedKYPLMI_index1} satisfies $\Xmax\succeq Q_{11}$ or, equivalently, 
\begin{equation*}
    \Etilde^\top \Xhat = 
    \begin{bmatrix}
        I_r & 0\\
        0   & 0
    \end{bmatrix}
    \begin{bmatrix}
        \Xmax    & 0\\
        X_{21}              & X_{22}
    \end{bmatrix}
    =
    \begin{bmatrix}
        \Xmax    & 0\\
        0                   & 0
    \end{bmatrix}
    \succeq
    \begin{bmatrix}
        Q_{11}  & 0\\
        0       & 0
    \end{bmatrix}
    =
    \Etilde^\top \Qtilde.
\end{equation*}
Furthermore, the replacement of $\Qtilde$ by $\Xhat$ as presented in Theorem~\ref{thm:optimalpHRealizationOfWCF_indexOne} does not change the realization $(\Etilde,\Atilde,\Btilde,\Ctilde)$ and, thus, we can again use $S$ and $T$ to transform the system back and obtain
\begin{align*}
    A = S^{-1}\Atilde T = S^{-1}\left(\Jhat-\Rhat\right)S^{-\top}S^\top\Xhat T = (\overline{J}-\overline{R})\overline{Q}
\end{align*}
with $\overline{J} \vcentcolon= S^{-1}\Jhat S^{-\top}$, $\overline{R} \vcentcolon= S^{-1}\Rhat S^{-\top}$, and $\overline{Q} \vcentcolon= S^\top\Xhat T$, as well as 
\begin{equation*}
    C = \Ctilde T = \Btilde^\top S^{-\top}S^\top\Xhat T = B^\top \overline{Q}.
\end{equation*}
Especially, we note that 
\begin{equation}
    \label{eq:ETQIneq}
    E^\top \overline{Q} = T^\top \Etilde^\top S^{-\top} S^\top \Xhat T = T^\top \Etilde^\top\Xhat T \succeq T^\top \Etilde^\top\Qtilde T = E^\top Q.
\end{equation}

\subsubsection{Procedure for Systems with Arbitrary Index}

As in section~\ref{subsubsec:indexOne}, we consider \eqref{eq:pHEQ}, but in this subsection we do not assume any restrictions for the index.
In a first step, we apply a feedback $\uF = -y = -Cx = -B^\top Qx$ to the system to obtain
\begin{equation}
    \label{eq:pHWithFeedback}
    \begin{aligned}
        E\dot{x} &= (J-R-BB^\top)Qx+ Bu,\quad x(0)=0,\\ 
        y			&= B^\top Qx.
    \end{aligned}
\end{equation}
The resulting system \eqref{eq:pHWithFeedback} is obviously still port-Hamiltonian with new dissipation matrix $R+BB^\top$.
Moreover, if $(E,A,B,C)$ meets the requirements of Theorem~\ref{thm:solutionsOfNonsymmetricGAREs}, then $(E,(J-R-BB^\top)Q)$ is regular and impulse-free by Theorem~\ref{thm:solutionsOfNonsymmetricGAREs}(iii).
Thus, there exists a state space transformation which transforms \eqref{eq:pHWithFeedback} to Weierstraß canonical form, i.e., there exist $S,T\in\GL{n}$ with
\begin{equation}
    \label{eq:WCFpH}
    \Etilde = SET^{-1} = 
    \begin{bmatrix}
        I_r & 0\\
        0   & 0
    \end{bmatrix}
    ,\quad \Atilde = S(A-BC)T^{-1} = 
    \begin{bmatrix}
        A_{11}  & 0\\
        0       & I_{n-r}
    \end{bmatrix}
\end{equation}
and similarly $\Btilde = SB$ and $\Ctilde = CT^{-1}$ with block partitioning $\Btilde^\top = [B_1^\top\;\;B_2^\top]$ and $\Ctilde = [C_1\;\;C_2]$ in accordance with the partioning in \eqref{eq:WCFpH}.
Again, note that for the following considerations we do not need $A_{11}$ to be in JCF, which allows us some more freedom in choosing $S$ and $T$.

Similarly as in section~\ref{subsubsec:indexOne}, we conclude that $Q_{22}$ is invertible and that there exists a skew-symmetric matrix $J_Q$ and a symmetric positive semidefinite matrix $R_Q$ satisfying $Q_{22} = J_Q-R_Q$.

Due to the special structure of the transformed system, $\Qtilde$ is a solution to the KYP LMI
\begin{equation}
	\label{eq:transformedKYPLMI}
	\begin{bmatrix}
		-\Atilde^\top \Xtilde-\Xtilde^\top \Atilde-2\Ctilde^\top\Ctilde 	& \Ctilde^\top-\Xtilde^\top \Btilde\\
		\Ctilde-\Btilde^\top \Xtilde 															& 0
	\end{bmatrix}
	\succeq 0,\quad \Etilde^\top \Xtilde = \Xtilde^\top \Etilde.
\end{equation}
When substituting $\Qtilde$ for $\Xtilde$, the fact that the equations $\Etilde^\top \Qtilde = \Qtilde^\top\Etilde$ and $\Ctilde^\top-\Qtilde^\top \Btilde=0$ are satisfied follows directly from the fact that the transformed system is port-Hamiltonian.
Furthermore, the upper left block of the LMI is satisfied due to
\begin{equation*}
  \begin{aligned}
	&-\Atilde^\top \Qtilde-\Qtilde^\top \Atilde-2\Ctilde^\top\Ctilde\\
	&\quad = -\Qtilde^\top(\widetilde{J}^\top-\Rtilde^\top-\Btilde\Btilde^\top)\Qtilde-\Qtilde^\top(\widetilde{J}-\Rtilde-\Btilde\Btilde^\top)\Qtilde-2\Qtilde^\top \Btilde\Btilde^\top \Qtilde \\
	&\quad = 2\Qtilde^\top\Rtilde\Qtilde \succeq 0.
  \end{aligned}
\end{equation*}

Again, the matrix equation in \eqref{eq:transformedKYPLMI} is satisfied if and only if $X_{11}$ is symmetric and $X_{12}=0$.
Furthermore, the matrix inequality implies
\begin{equation}
	\label{eq:reducedKYPLMI}
	\begin{aligned}
		&
		\begin{bmatrix}
			I_r & 0\\
			0 	& -B_2\\
			0	& I_m
		\end{bmatrix}
		^\top
		\begin{bmatrix}
			-\Atilde^\top \Xtilde-\Xtilde^\top \Atilde-2\Ctilde^\top\Ctilde 	& \Ctilde^\top-\Xtilde^\top \Btilde\\
			\Ctilde-\Btilde^\top \Xtilde 															& 0
		\end{bmatrix}
		\begin{bmatrix}
			I_r & 0\\
			0 	& -B_2\\
			0	& I_m
		\end{bmatrix}
		\\
		&=
		\underbrace{
		\begin{bmatrix}
			-A_{11}^\top X_{11}-X_{11}A_{11}-2C_1^\top C_1 	& C_1^\top-X_{11}B_1+2C_1^\top C_2B_2\\
			C_1-B_1^\top X_{11}+2B_2^\top C_2^\top C_1 		& -C_2B_2-B_2^\top C_2^\top-2B_2^\top C_2^\top C_2B_2
		\end{bmatrix}
		}_{=\vcentcolon \calW(X_{11})}
		\succeq 0.
	\end{aligned}
\end{equation}
In particular, this matrix inequality is satisfied by $Q_{11}$. 
Furthermore, the following lemma provides conditions for the existence of a maximal solution of this KYP LMI.

\begin{lemma}
    \label{lem:existenceMaximalSolutionOfReducedKYPLMI}
    Consider the system \eqref{eq:pHEQ} with regular pair $(E,A)$, strongly controllable triple $(E,A,B)$, and strongly detectable triple $(E,A,C)$.
    Furthermore, consider the corresponding transformed system $(\Etilde,\Atilde,\Btilde,\Ctilde)$ as in \eqref{eq:WCFpH} as well as the associated reduced KYP LMI \eqref{eq:reducedKYPLMI}.
    Then, there exists a maximizing solution $\Xmax=\Xmax^\top$ of \eqref{eq:reducedKYPLMI} satisfying $\Xmax\succeq X_{11}$ for all $X_{11}=X_{11}^\top$ solving \eqref{eq:reducedKYPLMI}.
\end{lemma}

\begin{proof}
    First, since the requirements of Theorem~\ref{thm:solutionsOfNonsymmetricGAREs} are satisfied, the transformation to $(\Etilde,\Atilde,\Btilde,\Ctilde)$ as in \eqref{eq:WCFpH} is well-defined.
    Furthermore, since $Q_{11}=Q_{11}^\top$ solves the reduced KYP LMI \eqref{eq:reducedKYPLMI}, by Theorems~\ref{thm:existenceOfStabilizingSolutionsOfTheLureEquation} and \ref{thm:extremalSolutionOfKYPLMI}, it is sufficient to show that the triple $(I_r,A_{11},B_1)$ is strongly anti-stabilizable.
    Note that $(I_r,A_{11})$ is trivially regular and $(I_r,A_{11},B_1)$ is trivially impulse controllable, since $I_r$ is invertible.
    
    Since $(E,A,B)$ is strongly controllable and, thus, in particular strongly anti-stabilizable, we infer that also the triple $(E,A-BC,B)$ obtained via the output feedback as in \eqref{eq:pHWithFeedback} is strongly anti-stabilizable.
    This follows from the definition of anti-stabilizability in Definition~\ref{def:controllabilityNotionsForDescriptorSystems} and from the relation
    \begin{equation*}
        \begin{bmatrix}
            \lambda E-A+BC & B
        \end{bmatrix}
        =
        \begin{bmatrix}
            \lambda E-A & B
        \end{bmatrix}
        \begin{bmatrix}
            I_n & 0\\
            C   & I_m
        \end{bmatrix}
        \quad \text{for all }\lambda\in\CC.
    \end{equation*}
    As a consequence, also the triple $(\Etilde,\Atilde,\Btilde)$ is anti-stabilizable which follows from
    \begin{equation*}
        \begin{bmatrix}
            \lambda \Etilde-\Atilde & \Btilde
        \end{bmatrix}
        =
        S
        \begin{bmatrix}
            \lambda E-A+BC & B
        \end{bmatrix}
        \begin{bmatrix}
            T^{-1}  & 0\\
            0       & I_m
        \end{bmatrix}
        \quad \text{for all }\lambda\in\CC.
    \end{equation*}
    Finally, the anti-stabilizability of $(I_r,A_{11},B_1)$ follows from the calculation
    \begin{align*}
        n &= \rank\left(
        \begin{bmatrix}
            \lambda \Etilde-\Atilde & \Btilde
        \end{bmatrix}
        \right) = \rank\left(
        \begin{bmatrix}
            \lambda I_r-A_{11}  & 0         & B_1\\
            0                   & -I_{n-r}  & B_2
        \end{bmatrix}
        \right) \\
        &= \rank\left(
        \begin{bmatrix}
            \lambda I_r-A_{11}  & 0         & B_1\\
            0                   & -I_{n-r}  & 0
        \end{bmatrix}
        \right) \\
        &= n-r+\rank\left(
        \begin{bmatrix}
            \lambda I_r-A_{11}  & B_1
        \end{bmatrix}
        \right)\quad \text{for all }\lambda\in\overline{\CC_-}.
    \end{align*}
\end{proof}

The following theorem shows how to obtain a port-Hamiltonian realization based on a maximal solution of \eqref{eq:reducedKYPLMI}.
The corresponding proof is outlined in appendix~\ref{sec:proof}.

\begin{theorem}
    \label{thm:optimalpHRealizationOfWCF}
    Consider the system \eqref{eq:pHEQ} with regular pair $(E,A)$, strongly stabilizable triple $(E,A,B)$, and strongly detectable triple $(E,A,C)$.
    Furthermore, consider the corresponding transformed system $(\Etilde,\Atilde,\Btilde,\Ctilde)$ as in \eqref{eq:WCFpH} and let there exist a maximal solution $\Xmax$ to the associated reduced KYP LMI \eqref{eq:reducedKYPLMI} in the sense of Lemma~\ref{lem:existenceMaximalSolutionOfReducedKYPLMI}.
	Besides, let 
	\begin{equation*}
		B_2^\top = 
		\begin{bmatrix}
			U_1 & U_2
		\end{bmatrix}
		\begin{bmatrix}
			\Sigma 	& 0\\
			0       & 0
		\end{bmatrix}
		\begin{bmatrix}
			V_1^\top\\
			V_2^\top
		\end{bmatrix}
		= U_1\Sigma V_1^\top
	\end{equation*}
	with $\rank(\Sigma) = \rank(B_2) =\vcentcolon \widetilde{r}$ be a singular value decomposition of $B_2^\top$.
	Moreover, we define
	\begin{align*}
		X_{21} &\vcentcolon= V_1V_1^\top Q_{21}+(B_2^\top)^+B_1^\top (Q_{11}-\Xmax)-2V_2V_2^\top C_2^\top C_1,\\
		X_{22} &\vcentcolon= V_1V_1^\top Q_{22}-V_2V_2^\top Q_{22}^\top V_1V_1^{\top}+V_2KV_2^\top-2V_2V_2^\top C_2^\top C_2V_1V_1^{\top},\\
		 \Jhat_{11} &\vcentcolon= \frac12\left(A_{11}\Xmax^{-1}-\Xmax^{-1}A_{11}^\top\right),\quad \Rhat_{11} \vcentcolon= -\frac12\left(A_{11}\Xmax^{-1}+\Xmax^{-1}A_{11}^\top\right)-B_1B_1^\top
	\end{align*}
	with some arbitrary matrix $K\in\RR^{(n-r-\widetilde{r})\times(n-r-\widetilde{r})}$ satisfying 
	\begin{equation}
		\label{eq:KProperty}
		K+K^\top \prec -2V_2^\top C_2^\top C_2V_2.
	\end{equation}

	Then, the following assertions hold.
	\begin{enumerate}[(i)]
		\item The matrix $X_{22}$ is invertible.
		\item The matrices $\Atilde$ and $\Ctilde$ can be written as
			\begin{equation*}
				\Atilde = 
				\begin{bmatrix}
        				A_{11}  & 0\\
        				0       & I_{n-r}
    				\end{bmatrix}
    				= \Bigg(\underbrace{
    				\begin{bmatrix}
    					\Jhat_{11} & \Jhat_{12}\\
    					\Jhat_{21} & \Jhat_{22}
    				\end{bmatrix}
    				}_{=\vcentcolon \Jhat} -\underbrace{
    				\begin{bmatrix}
    					\Rhat_{11} & \Rhat_{12}\\
    					\Rhat_{21} & \Rhat_{22}
    				\end{bmatrix}
    				}_{=\vcentcolon \Rhat}-\Btilde\Btilde^\top\Bigg) \underbrace{
    				\begin{bmatrix}
    					\Xmax & 0\\
    					X_{21} 					& X_{22}
    				\end{bmatrix}
    				}_{=\vcentcolon \Xhat}
			\end{equation*}
			and $\Ctilde = \Btilde^\top \Xhat$, where $\Jhat_{12},\Jhat_{21},\Jhat_{22},\Rhat_{12},\Rhat_{21},\Rhat_{22}$ are defined via
			\begin{align*}
				\Jhat_{12} = -\Jhat_{21}^\top = \frac12 \Xmax^{-1}X_{21}^\top X_{22}^{-\top},\quad 
				\Rhat_{12} = \Rhat_{21}^\top \vcentcolon= \Jhat_{12}-B_1B_2^\top,\\  
				\Jhat_{22} \vcentcolon= \frac12\left(X_{22}^{-1}-X_{22}^{-\top}\right),\quad \Rhat_{22} \vcentcolon= -\frac12\left(X_{22}^{-1}+X_{22}^{-\top}\right)-B_2B_2^\top.
			\end{align*}
			Furthermore, $\Jhat$ is skew-symmetric and $\Rhat$ is symmetric and positive semidefinite.
	\end{enumerate}
\end{theorem}

Similarly as outlined in section~\ref{subsubsec:indexOne}, we can use the factorization presented in Theorem~\ref{thm:optimalpHRealizationOfWCF} to replace the $\Qtilde$ matrix by $\Xhat$.
After transforming back and undoing the original feedback via the reverse feedback $\uF=y=Cx$, we obtain an alternative port-Hamiltonian representation of $(E,A,B,C)$ via $(E,(\overline{J}-\overline{R})\overline{Q},B,B^\top\overline{Q})$ with $E^\top \overline{Q} \succeq E^\top Q$.

The relation between the KYP LMI \eqref{eq:generalizedKYP_II} associated to the original pH system and the reduced KYP LMI in \eqref{eq:reducedKYPLMI} has been investigated in a more general setting in \cite{Voi15}.
The following statement is a special case of Proposition~3.2.2~b) in \cite{Voi15} for the port-Hamiltonian setting considered here.
Especially, it implies that a maximal solution of the reduced KYP LMI in \eqref{eq:reducedKYPLMI} corresponds to a maximal solution of the original KYP LMI \eqref{eq:generalizedKYP_II}.

\begin{theorem}
    Consider the system \eqref{eq:pHEQ} with regular pair $(E,A)$, strongly stabilizable triple $(E,A,B)$, and strongly detectable triple $(E,A,C)$, as well as the associated KYP LMI \eqref{eq:generalizedKYP_II}.
    Furthermore, consider the corresponding transformed system $(\Etilde,\Atilde,\Btilde,\Ctilde)$ as in \eqref{eq:WCFpH} as well as the associated reduced KYP LMI \eqref{eq:reducedKYPLMI}.
    Then, for any solution $X\in\RR^{n\times n}$ of \eqref{eq:generalizedKYP_II} the matrix
    \begin{equation}
        \label{eq:XOneOne}
        X_{11} = 
        \begin{bmatrix}
            I_r & 0
        \end{bmatrix}
        S^{-\top} XT^{-1}
        \begin{bmatrix}
            I_r\\
            0
        \end{bmatrix}
    \end{equation}
    is symmetric and a solution of \eqref{eq:reducedKYPLMI}.
    Conversely, let $X_{11}\in\RR^{r\times r}$ be a symmetric solution of \eqref{eq:reducedKYPLMI}.
    Then, for any $X_{21}\in\RR^{(n-r)\times r}$ and $X_{22}\in\RR^{(n-r)\times(n-r)}$, the matrix
    \begin{equation}
        \label{eq:X}
        X = 
        S^\top
        \begin{bmatrix}
            X_{11} & 0\\
            X_{21} & X_{22}
        \end{bmatrix}
        T
    \end{equation}
    is a solution of \eqref{eq:generalizedKYP_II}.
\end{theorem}

\begin{proof}
    First, since the requirements of Theorem~\ref{thm:solutionsOfNonsymmetricGAREs} are satisfied, the transformation to $(\Etilde,\Atilde,\Btilde,\Ctilde)$ as in \eqref{eq:WCFpH} is well-defined.
    The fact that $X$ solves \eqref{eq:generalizedKYP_II} if and only if $\Xtilde = S^{-\top} XT^{-1}$ solves 
    \begin{equation}
	    \label{eq:transformedKYPLMI_Vsys}
	    \begin{bmatrix}
		    -\Atilde^\top \Xtilde-\Xtilde^\top \Atilde-2\Ctilde^\top\Ctilde 	& \Ctilde^\top-\Xtilde^\top \Btilde\\
		    \Ctilde-\Btilde^\top \Xtilde 															& 0
	    \end{bmatrix}
	    \succeq_{\Vsystilde} 0,\quad \Etilde^\top \Xtilde = \Xtilde^\top \Etilde
    \end{equation}
    is a direct consequence of Proposition~3.2.2~b) from \cite{Voi15}.
    Furthermore, since the system space of $(\Etilde,\Atilde,\Btilde,\Ctilde)$ is given by
    \begin{equation*}
        \Vsystilde = 
        \colspan\left(
        \begin{bmatrix}
			I_r & 0\\
			0 	& -B_2\\
			0	& I_m
		\end{bmatrix}
		\right),
    \end{equation*}
    cf.~\cite[eq.~(3.2)]{ReiRV15}, $S^{-\top} XT^{-1}$ solves \eqref{eq:transformedKYPLMI_Vsys} if and only if $X_{11}$ as defined in \eqref{eq:XOneOne} is symmetric and solves \eqref{eq:reducedKYPLMI}, which concludes the proof.
\end{proof}

\section{Numerical Examples}\label{sec:numericalExamples}
 
 In this section, we validate our theoretical findings by means of numerical results for two port-Hamiltonian descriptor systems. Since the main purpose is to demonstrate the general potential of our approach, we do not push to the limits in terms of the system dimensions of the original models. As we detail below, to a certain extent this is due to numerical ill-conditioning that requires a more in depth analysis that is out of the scope of this article.
 
All simulations were generated on an AMD Ryzen 7 PRO 4750U @ 1.7 GHz x 16, 32 GB RAM with \matlab \;version R2020b. 

In the following we list some details regarding the implementation used for the numerical experiments presented in subsections~\ref{subsec:transportNetwork} and \ref{subsec:msd}.

\begin{itemize}
    \item The errors between the FOM and the ROMs are measured by means of the $\mathcal{H}_\infty$ norm of the coprime factors, cf.~\eqref{eq:errorBoundForCoprimeFactors}.
        For the numerical computation of the $\mathcal{H}_\infty$ norm, we use the \texttt{Control System Toolbox} of \matlab.
    \item Determining an optimal Hamiltonian as outlined in section~\ref{subsec:optimalQ} involves finding a maximal solution of a KYP LMI of the form \eqref{eq:reducedKYPLMI_index1} or \eqref{eq:reducedKYPLMI}.     This is done by considering the KYP LMIs dual to \eqref{eq:reducedKYPLMI_index1} or \eqref{eq:reducedKYPLMI} which are solved for the minimal solutions. 
        We then obtain the maximal solutions of \eqref{eq:reducedKYPLMI_index1} or \eqref{eq:reducedKYPLMI} by inversion.
        Both considered examples lead to a singular lower right block of the block matrix on the left-hand side of \eqref{eq:reducedKYPLMI_index1} or \eqref{eq:reducedKYPLMI}. Following ideas for the ODE case, see \cite[Theorem 2]{Wil72a}, we introduce an artificial feedthrough term $D+D^\top=10^{-12}I_m$ in order to replace the occurring KYP LMIs by algebraic Riccati equations associated with the Schur complements of \eqref{eq:reducedKYPLMI_index1} or \eqref{eq:reducedKYPLMI}. 
    \item All standard Riccati equations are solved with the \texttt{icare} routine from the \texttt{Control System Toolbox} of \matlab.
    \item For solving generalized Riccati equations as in \eqref{eq:originalGAREs} and \eqref{eq:modifiedGAREs}, we use the code from \cite{MoeRS11} which was kindly made available to us by the authors.
\end{itemize}

Let us emphasize that the numerical procedure described above is challenging already for systems of medium size since the maximal solutions to \eqref{eq:reducedKYPLMI_index1} or \eqref{eq:reducedKYPLMI} are typically ill-conditioned, cf.~\cite[Cor.~13]{Opm13}.

\subsection{Transport Network}\label{subsec:transportNetwork}

\begin{figure}[tb]
    \centering 
%
%
\definecolor{mycolor1}{rgb}{0.00000,0.44700,0.74100}%
\definecolor{mycolor2}{rgb}{0.85000,0.32500,0.09800}%

\begin{tikzpicture}

\begin{axis}[%
width=4in,
height=2.8in,
scale only axis,
separate axis lines,
every outer x axis line/.append style={white!15!black},
every x tick label/.append style={font=\color{white!15!black}},
xmin=2,
xmax=40,
xlabel style={font=\color{white!15!black}},
xlabel={Reduced system dimension $r$},
every outer y axis line/.append style={white!15!black},
every y tick label/.append style={font=\color{white!15!black}},
ymode=log,
ymin=1e-07,
ymax=25,
yminorticks=true,
ylabel style={font=\color{white!15!black}},
ylabel={$\left\| \begin{bsmallmatrix} M \\ N \end{bsmallmatrix}  - \begin{bsmallmatrix} M_r \\ N_r \end{bsmallmatrix}\right\|_{\mathcal{H}_\infty}$},
legend style={draw=white!15!black,fill=white,legend cell align=left},
legend pos = south west
]
\addplot [color=mycolor1, dotted, line width=1.5pt, mark size=5pt, mark=o, mark options={solid, mycolor1}]
  table[row sep=crcr]{2	1.40164588423571\\
4	0.472542631695268\\
6	1.09686481121316\\
8	0.407010074875338\\
10	0.788605292374918\\
12	0.335637189123697\\
14	0.511492351528952\\
16	0.274548398522545\\
18	0.308557670669933\\
20	0.223157870930515\\
22	0.198801253823111\\
24	0.182698974605792\\
26	0.16175261420651\\
28	0.143755957070821\\
30	0.13564101006684\\
32	0.107935428062489\\
34	0.112639757044664\\
36	0.0876469380804663\\
38	0.0929105312783282\\
40	0.0714446122706555\\
42	0.0749048435726467\\
44	0.0585869451093049\\
46	0.0590175350905708\\
48	0.048390630779972\\
50	0.0469463138011533\\
};
\addlegendentry{Error using $\mathcal{P}_{\mathrm{f}}=Q^{-\top}$}
\addplot [color=mycolor1, line width=1.5pt]
  table[row sep=crcr]{2	25.5482433133267\\
4	23.1064023738023\\
6	21.1182367432972\\
8	19.2948286927036\\
10	17.6228595951377\\
12	16.0950079285191\\
14	14.6804344372887\\
16	13.3877074582387\\
18	12.191074641207\\
20	11.0978792813753\\
22	10.0935405800901\\
24	9.17591911148147\\
26	8.33813733962292\\
28	7.57241116061939\\
30	6.87614006389564\\
32	6.23993307216105\\
34	5.66274878678209\\
36	5.13595227779548\\
38	4.65849809870655\\
40	4.2234786737858\\
42	3.82927177875931\\
44	3.47078891283302\\
46	3.14585515832109\\
48	2.85088392443669\\
50	2.58340366507758\\
};
\addlegendentry{Error bound using $\mathcal{P}_{\mathrm{f}}=Q^{-\top}$}
\addplot [color=mycolor2, dotted, line width=1.5pt, mark size=5pt, mark=o, mark options={solid, mycolor2}]
  table[row sep=crcr]{2	1.39633575361393\\
4	0.687468811405881\\
6	0.178703409127719\\
8	0.1233292065835\\
10	0.1312805710984\\
12	0.0566548339080046\\
14	0.0139787221679339\\
16	0.022902345910044\\
18	0.00758052872474863\\
20	0.00167720835663337\\
22	0.00320173256197667\\
24	0.000959590286582213\\
26	0.000202593236830341\\
28	0.00041404790811015\\
30	0.000134522526744864\\
32	3.31567707853529e-05\\
34	2.92024880349469e-05\\
36	1.4115692694241e-05\\
38	1.18347978413182e-05\\
40	1.22100601779493e-05\\
42	1.07110536216994e-05\\
44	1.08260154072669e-05\\
46	1.24555659959627e-05\\
48	1.24412794728414e-05\\
50	1.2390069195975e-05\\
};
\addlegendentry{Error using $\mathcal{P}_{\mathrm{f}}=X_{\mathrm{max}}^{-\top}$}
\addplot [color=mycolor2, line width=1.5pt]
  table[row sep=crcr]{2	4.98076315793742\\
4	3.03804384690141\\
6	1.7594518015065\\
8	0.97666162170226\\
10	0.51090074416899\\
12	0.279186794387371\\
14	0.131309069595807\\
16	0.0748242722160547\\
18	0.0369778126387607\\
20	0.0179810920199969\\
22	0.0111800556872519\\
24	0.00563906970237761\\
26	0.00330251058796208\\
28	0.00248364667649319\\
30	0.00169131577939393\\
32	0.00139939856351056\\
34	0.00125254109911104\\
36	0.00115233480055021\\
38	0.00111183647729845\\
40	0.001072613680099\\
42	0.0010339840201113\\
44	0.000996951354177152\\
46	0.000960122617378956\\
48	0.000923483163696074\\
50	0.000886950137009619\\
};
\addlegendentry{Error bound using $\mathcal{P}_{\mathrm{f}}=X_{\mathrm{max}}^{-\top}$}
\addplot [color=black, dashdotted, line width=1.5pt, mark size=5pt, mark=asterisk, mark options={solid, black}]
  table[row sep=crcr]{2	1.37190625461159\\
4	0.570795061332507\\
6	0.283233967589945\\
8	0.0741733362766564\\
10	0.0740995553817004\\
12	0.038421216508487\\
14	0.0123561364511339\\
16	0.00952660109693222\\
18	0.00508766086122097\\
20	0.00182622803557019\\
22	0.00116045716971943\\
24	0.000641390013319729\\
26	0.000245257535282233\\
28	0.00013965841632317\\
30	9.62705069035599e-05\\
32	4.26486614898326e-05\\
34	1.79005174880735e-05\\
36	1.07380559704633e-05\\
38	1.07375859353021e-05\\
40	1.07392502658123e-05\\
42	1.14748797382556e-05\\
44	1.05652526718398e-05\\
46	1.05302778942549e-05\\
48	1.07312173245891e-05\\
50	1.06848025363358e-05\\
};
\addlegendentry{LQG-BT}
\addplot [color=black, line width=1.5pt]
  table[row sep=crcr]{2	3.16889105629602\\
4	1.37828518922793\\
6	0.819422953531396\\
8	0.405289997551507\\
10	0.240459346575972\\
12	0.109670910856572\\
14	0.0586682628825346\\
16	0.0319012159604616\\
18	0.01478696947651\\
20	0.00846064603474153\\
22	0.00455317320919861\\
24	0.00249260878900159\\
26	0.00170482847317604\\
28	0.00115494684861701\\
30	0.000906790902320656\\
32	0.000779311158458995\\
34	0.000700450256728866\\
36	0.000668467526160875\\
38	0.000645801232499801\\
40	0.000623136199963633\\
42	0.000600740722104112\\
44	0.000578385858173215\\
46	0.000556524333955843\\
48	0.000534691399366108\\
50	0.000513384406586934\\
};
\addlegendentry{Error bound for LQG-BT}
\end{axis}
\end{tikzpicture}%
        \caption{Comparison of the pH-structure-preserving LQG-BT method with the classical unstructured variant for a transport network. The $\mathcal{H}_{\infty}$-error is shown for different reduced system dimensions and different choices of the system Hamiltonian.}
    \label{fig:network}
\end{figure}
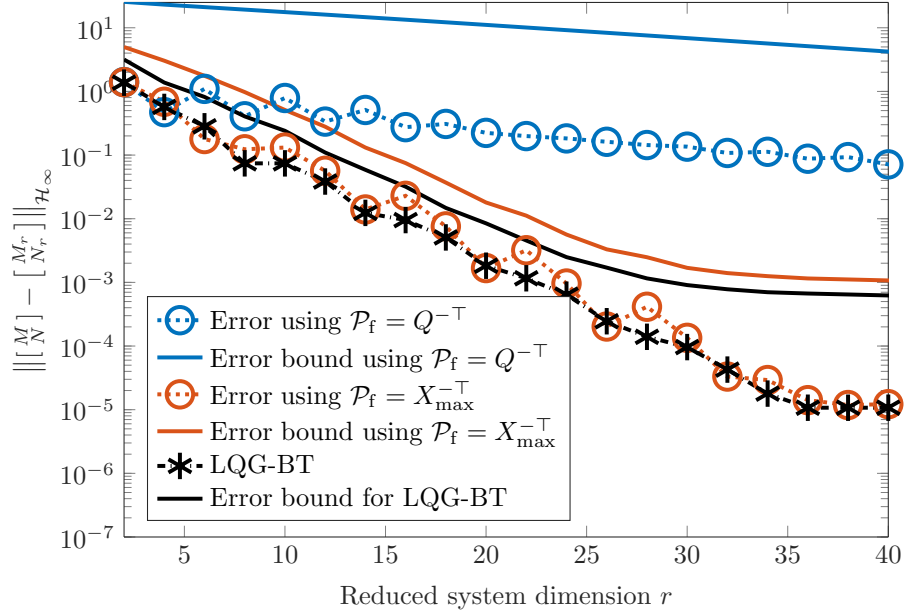

We consider the transport network model from \cite[sec.~7.4]{EggKLMM18} given by
\begin{align*}
	\begin{bmatrix}
		M_1 	& 0 		& 0\\
		0 		& M_2 	& 0\\
		0		& 0		& 0
	\end{bmatrix}
	\begin{bmatrix}
		\dot{x}_1\\
		\dot{x}_2\\
		\dot{x}_3
	\end{bmatrix}
	&=
	\begin{bmatrix}
		0 			& -G & 0\\
		G^\top 	& -D & N^\top\\
		0			& -N & 0
	\end{bmatrix}
	\begin{bmatrix}
		x_1\\
		x_2\\
		x_3
	\end{bmatrix}
	+
	\begin{bmatrix}
		0\\
		B_2\\
		0
	\end{bmatrix}
	u\\
	y &= 
	\begin{bmatrix}
		0 & B_2^\top & 0
	\end{bmatrix}
	\begin{bmatrix}
		x_1\\
		x_2\\
		x_3
	\end{bmatrix}
	,
\end{align*}
where $M_1,M_2,D$ are symmetric and positive definite, and $N$ has full row rank.
This is an index-2 DAE, see \cite[Rem.~2.2]{EggKLMM18}, and in particular it is not impulse-controllable.
In order to make the system impulse-controllable we perform an index reduction as outlined in \cite[sec.~3.3]{KunM06} and arrive at the equivalent index-1 pHDAE
\begin{align*}
	\begin{bmatrix}
		M_1 	& 0 					& 0\\
		0 		& V^\top M_2 V & 0\\
		0		& 0					& 0
	\end{bmatrix}
	\begin{bmatrix}
		\dot{x}_1\\
		\dot{\widetilde{x}}_2\\
		\dot{\widetilde{x}}_3
	\end{bmatrix}
	&=
	\begin{bmatrix}
		0 			& -GV & 0\\
		V^\top G^\top 	& -V^\top DV & 0\\
		0			& 0 & -I
	\end{bmatrix}
	\begin{bmatrix}
		x_1\\
		\widetilde{x}_2\\
		\widetilde{x}_3
	\end{bmatrix}
	+
	\begin{bmatrix}
		0\\
		V^\top B_2\\
		0
		\end{bmatrix}
	u\\
	y &= 
	\begin{bmatrix}
		0 & B_2^\top V & 0
	\end{bmatrix}
	\begin{bmatrix}
		x_1\\
		\widetilde{x}_2\\
		\widetilde{x}_3
	\end{bmatrix}
	.
\end{align*}
Here, the columns of the full column rank matrix $V$ span the kernel of $N$.
Note that the resulting system is impulse-controllable, but it is also essentially an ODE since the algebraic constraint does not contribute to the input-output map and could, thus, be removed.

The resulting index-1 system can be brought into semi-explicit form via a state space transformation leading to
\begin{align*}
	\begin{bmatrix}
		I 	& 0 					& 0\\
		0 		& I & 0\\
		0		& 0					& 0
	\end{bmatrix}
	\begin{bmatrix}
		\dot{\widetilde{x}}_1\\
		\dot{\widehat{x}}_2\\
		\dot{\widetilde{x}}_3
	\end{bmatrix}
	&=
	\begin{bmatrix}
		0 			                        & -L_1^{-\top}GVL_2^{-1}        & 0\\
		L_2^{-\top}V^\top G^\top L_1^{-1} 	& -L_2^{-\top}V^\top DVL_2^{-1} & 0\\
		0			                        & 0                             & -I
	\end{bmatrix}
	\begin{bmatrix}
		\widetilde{x}_1\\
		\widehat{x}_2\\
		\widetilde{x}_3
	\end{bmatrix}
	\\
	&\quad+
	\begin{bmatrix}
		0\\
		L_2^{-\top}V^\top B_2\\
		0
		\end{bmatrix}
	u\\
	y &= 
	\begin{bmatrix}
		0 & B_2^\top VL_2^{-1} & 0
	\end{bmatrix}
	\begin{bmatrix}
		\widetilde{x}_1\\
		\widehat{x}_2\\
		\widetilde{x}_3
	\end{bmatrix}
	,
\end{align*}
where $M_1=L_1^\top L_1$ and $M_2 = L_2^\top L_2$ are Cholesky factorizations of $M_1$ and $M_2$, respectively.
For the numerical simulations, we use the parameters from \cite{EggKLMM18} except for the damping parameter $d_0$ (globally modelling different friction coefficients of the pipes) which we set to $d_0=25.$ For the damped waved equations on the individual pipes, we perform finite element discretizations with 50 inner nodes, leading to a descriptor system of dimension $n=721$, including $4$ algebraic constraints. Since the computation of a maximal solution to \eqref{eq:reducedKYPLMI_index1} via inversion of a minimal solution to the dual KYP LMI is numerically ill-conditioned, we first construct a reduced system of dimension $n=252$ by our approach using the canonical Hamiltonian. 
Let us however emphasize that for the computation of the errors, we use the original realization.

In Figure \ref{fig:network}, we show the results for reduced models obtained by Algorithm~\ref{alg:MOR} and the classical LQG-BT method from \cite{MoeRS11}. For our approach, we distinguish between the canonical port-Hamiltonian representation associated with a finite element discretization of the model equations and an improved representation constructed along the findings from subsection \ref{subsec:optimalQ}. We remark the following two observations that confirm our theoretical discussion: on the one hand, the error as well as the error bound for an optimal choice of the Hamiltonian decays significantly faster than in the canonical case. Moreover, in terms of the $\mathcal{H}_\infty$-error of the coprime factors, our approach yields reduced systems that are comparable to those obtained by classical (unstructured) LQG balanced truncation.

\subsection{Mass-Spring-Damper System}\label{subsec:msd}

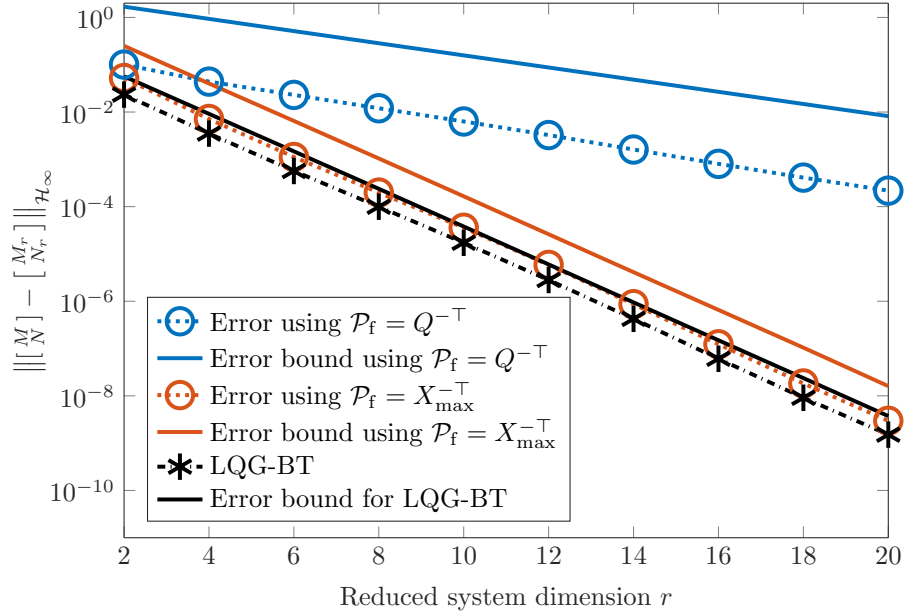
\begin{figure}[tb]
    \centering 
%
%
\definecolor{mycolor1}{rgb}{0.00000,0.44700,0.74100}%
\definecolor{mycolor2}{rgb}{0.85000,0.32500,0.09800}%

\begin{tikzpicture}

\begin{axis}[%
width=4in,
height=2.8in,
scale only axis,
separate axis lines,
every outer x axis line/.append style={white!15!black},
every x tick label/.append style={font=\color{white!15!black}},
xmin=2,
xmax=20,
xlabel style={font=\color{white!15!black}},
xlabel={Reduced system dimension $r$},
every outer y axis line/.append style={white!15!black},
every y tick label/.append style={font=\color{white!15!black}},
ymode=log,
ymin=1e-11,
ymax=2,
yminorticks=true,
ylabel style={font=\color{white!15!black}},
ylabel={$\left\| \begin{bsmallmatrix} M \\ N \end{bsmallmatrix}  - \begin{bsmallmatrix} M_r \\ N_r \end{bsmallmatrix}\right\|_{\mathcal{H}_\infty}$},
legend style={draw=white!15!black,fill=white,legend cell align=left},
legend pos = south west
]
\addplot [color=mycolor1, dotted, line width=1.5pt, mark size=5pt, mark=o, mark options={solid, mycolor1}]
  table[row sep=crcr]{2	0.100707220128138\\
4	0.0439576609642183\\
6	0.0230151086131421\\
8	0.012013881487135\\
10	0.00631268735907714\\
12	0.00325940715561739\\
14	0.00161307705219993\\
16	0.000802481271269503\\
18	0.00041045510219389\\
20	0.000217760828404706\\
};
\addlegendentry{Error using $\mathcal{P}_{\mathrm{f}}=Q^{-\top}$}
\addplot [color=mycolor1, line width=1.5pt]
  table[row sep=crcr]{2	1.68556674220908\\
4	0.931731932772731\\
6	0.514210935884942\\
8	0.284364115031456\\
10	0.157636773949248\\
12	0.0873836642572985\\
14	0.0483707986115207\\
16	0.0267458942818505\\
18	0.0147922239461641\\
20	0.00818448258807375\\
};
\addlegendentry{Error bound using $\mathcal{P}_{\mathrm{f}}=Q^{-\top}$}
\addplot [color=mycolor2, dotted, line width=1.5pt, mark size=5pt, mark=o, mark options={solid, mycolor2}]
  table[row sep=crcr]{2	0.0522763687255148\\
4	0.00713075522638676\\
6	0.00112039005961285\\
8	0.000198190416326255\\
10	3.57303090241412e-05\\
12	5.86460190802462e-06\\
14	8.66629746795991e-07\\
16	1.2304058341673e-07\\
18	1.8254605432046e-08\\
20	2.970897182126e-09\\
};
\addlegendentry{Error using $\mathcal{P}_{\mathrm{f}}=X_{\mathrm{max}}^{-\top}$}
\addplot [color=mycolor2, line width=1.5pt]
  table[row sep=crcr]{2	0.250040971861711\\
4	0.0401324918376263\\
6	0.00648278811757063\\
8	0.00104197422012573\\
10	0.000165478747151126\\
12	2.60671456993652e-05\\
14	4.11657015492169e-06\\
16	6.55673548394756e-07\\
18	1.04377891967755e-07\\
20	1.60471361117867e-08\\
};
\addlegendentry{Error bound using $\mathcal{P}_{\mathrm{f}}=X_{\mathrm{max}}^{-\top}$}
\addplot [color=black, dashdotted, line width=1.5pt, mark size=5pt, mark=asterisk, mark options={solid, black}]
  table[row sep=crcr]{2	0.0233559428156736\\
4	0.00350418646559725\\
6	0.000567074983719261\\
8	9.95040295988022e-05\\
10	1.72240945813195e-05\\
12	2.81069394130988e-06\\
14	4.24623619341843e-07\\
16	6.07699757554104e-08\\
18	9.10707671886401e-09\\
20	1.51210158742562e-09\\
};
\addlegendentry{LQG-BT}
\addplot [color=black, line width=1.5pt]
  table[row sep=crcr]{2	0.0567990400194596\\
4	0.00915663920619488\\
6	0.0014739084803186\\
8	0.000237854250559682\\
10	3.79991236756862e-05\\
12	6.00181785245798e-06\\
14	9.45086037351444e-07\\
16	1.49618368609563e-07\\
18	2.37274845490688e-08\\
20	3.73550181286961e-09\\
};
\addlegendentry{Error bound for LQG-BT}
\end{axis}
\end{tikzpicture}%
        \caption{Comparison of the pH-structure-preserving LQG-BT method with the classical unstructured variant for a constrained mass-spring-damper system. The $\mathcal{H}_{\infty}$-error is shown for different reduced system dimensions and different choices of the system Hamiltonian.}
    \label{fig:msd}
\end{figure}

We consider a constrained mass-spring-damper system, similar as in \cite[sec.~3.4]{StyM05}, given by
\begin{align*}
    M\ddot{x}+D\dot{x}+Kx-N^\top\lambda &= 0,\\
    N\dot{x}                            &= u,\\
    y                                   &= \lambda.
\end{align*}
Here, $M=cI$ with $c\in\RR_{>0}$, $K$, and, $D$ are symmetric and positive definite.
In contrast to \cite{StyM05}, we modified the algebraic constraint, since not the relative position between the first and the last mass in the system is controlled, but instead their relative velocity.
Otherwise, the resulting first-order system would have index 3 and, thus, could not be port-Hamiltonian, cf.~\cite[Thm.~4.3]{MehMW18}.
Furthermore, we modified the output equation such that the Lagrange multiplier $\lambda$ is measured which corresponds to the force which needs to be applied to the system in order to enforce the algebraic constraint.
A first-order formulation reads
\begin{align*}
    \begin{bmatrix}
		K 	& 0 		& 0\\
		0 		& M 	& 0\\
		0		& 0		& 0
	\end{bmatrix}
	\begin{bmatrix}
		\dot{x}_1\\
		\dot{x}_2\\
		\dot{x}_3
	\end{bmatrix}
	&=
	\begin{bmatrix}
		0 	& K & 0\\
		-K 	& -D & N^\top\\
		0	& -N & 0
	\end{bmatrix}
	\begin{bmatrix}
		x_1\\
		x_2\\
		x_3
	\end{bmatrix}
	+
	\begin{bmatrix}
		0\\
		0\\
		1
	\end{bmatrix}
	u,\\
	y &= 
	\begin{bmatrix}
		0 & 0 & 1
	\end{bmatrix}
	\begin{bmatrix}
		x_1\\
		x_2\\
		x_3
	\end{bmatrix}
	,
\end{align*}
which is pH with $Q=I$.
A corresponding semi-explicit form is given by
\begin{align*}
    \begin{bmatrix}
		I 	& 0 		& 0\\
		0 		& I 	& 0\\
		0		& 0		& 0
	\end{bmatrix}
	\begin{bmatrix}
		\dot{\widetilde{x}}_1\\
		\dot{\widetilde{x}}_2\\
		\dot{x}_3
	\end{bmatrix}
	&=
	\begin{bmatrix}
		0 	& \frac1{\sqrt{c}}L^\top & 0\\
		-\frac1{\sqrt{c}}L 	& -\frac1c D & -\frac1{\sqrt{c}}N^\top\\
		0	& \frac1{\sqrt{c}}N & 0
	\end{bmatrix}
	\begin{bmatrix}
		\widetilde{x}_1\\
		\widetilde{x}_2\\
		x_3
	\end{bmatrix}
	+
	\begin{bmatrix}
		0\\
		0\\
		1
	\end{bmatrix}
	u,\\
	y &= 
	\begin{bmatrix}
		0 & 0 & 1
	\end{bmatrix}
	\begin{bmatrix}
		\widetilde{x}_1\\
		\widetilde{x}_2\\
		x_3
	\end{bmatrix}
	,
\end{align*}
where $K=LL^\top$ is a Cholesky decomposition of $K$.
This system has index 2 and, thus, in order to improve the error bound by replacing $Q$, we need to apply a regularizing output feedback leading to
\begin{align*}
    \begin{bmatrix}
		I 	& 0 		& 0\\
		0 		& I 	& 0\\
		0		& 0		& 0
	\end{bmatrix}
	\begin{bmatrix}
		\dot{\widetilde{x}}_1\\
		\dot{\widetilde{x}}_2\\
		\dot{x}_3
	\end{bmatrix}
	&=
	\begin{bmatrix}
		0 	& \frac1{\sqrt{c}}L^\top & 0\\
		-\frac1{\sqrt{c}}L 	& -\frac1c D & -\frac1{\sqrt{c}}N^\top\\
		0	& \frac1{\sqrt{c}}N & -1
	\end{bmatrix}
	\begin{bmatrix}
		\widetilde{x}_1\\
		\widetilde{x}_2\\
		x_3
	\end{bmatrix}
	+
	\begin{bmatrix}
		0\\
		0\\
		1
	\end{bmatrix}
	u\\
	y &= 
	\begin{bmatrix}
		0 & 0 & 1
	\end{bmatrix}
	\begin{bmatrix}
		\widetilde{x}_1\\
		\widetilde{x}_2\\
		x_3
	\end{bmatrix}
	.
\end{align*}
By performing a state space transformation, we obtain the equivalent realization
\begin{equation}
    \label{eq:msd_pHWCF}
    \begin{aligned}
        \begin{bmatrix}
		    I 	& 0 		& 0\\
		    0 		& I 	& 0\\
		    0		& 0		& 0
	    \end{bmatrix}
	    \begin{bmatrix}
		    \dot{\widetilde{x}}_1\\
		    \dot{\widetilde{x}}_2\\
		    \dot{\widetilde{x}}_3
	    \end{bmatrix}
	    &=
	    \begin{bmatrix}
		    0 	& \frac1{\sqrt{c}}L^\top & 0\\
		    -\frac1{\sqrt{c}}L 	& -\frac1c (D+N^\top N) & 0\\
		    0	& 0 & 1
	    \end{bmatrix}
	    \begin{bmatrix}
		    \widetilde{x}_1\\
		    \widetilde{x}_2\\
		    \widetilde{x}_3
	    \end{bmatrix}
	    +
	    \begin{bmatrix}
	    	0\\
    		-\frac1{\sqrt{c}}N^\top\\
    		-1
    	\end{bmatrix}
    	u,\\
    	y &= 
    	\begin{bmatrix}
    		0 & \frac1{\sqrt{c}}N & 1
    	\end{bmatrix}
    	\begin{bmatrix}
    		\widetilde{x}_1\\
    		\widetilde{x}_2\\
    		\widetilde{x}_3
    	\end{bmatrix}
    	,
	\end{aligned}
\end{equation}
where differential and algebraic equations are decoupled.
Note that one could also perform a state space transformation where additionally the $A_{11}$ block is in Jordan canonical form and, thus, the whole system would be in Weierstraß canonical form, but this is not necessary for our approach, cf.~section~\ref{subsec:optimalQ}.
The new $Q$ matrix resulting from the state space transformation leading to \eqref{eq:msd_pHWCF} reads
\begin{equation*}
    Q = 
    \begin{bmatrix}
        I & 0 & 0\\
        0 & I & 0\\
        0 & -\frac2{\sqrt{c}}N & -1
    \end{bmatrix}
    .
\end{equation*}
For our experiments, we choose the same parameter values as in \cite{StyM05}, but we set the bulk damping parameters to $2$ and the boundary damping to $4$.
Furthermore, we consider a system of $500$ masses, leading to a system of dimension $n=1001$ with a single algebraic constraint. 
Since the resulting system is uncontrollable, we further use the \matlab\;routine \texttt{ctrbf} to obtain a minimal realization of dimension $n=501$, including $1$ algebraic constraint.

As shown in Figure \ref{fig:msd}, we obtain similar results as in the case of the transport network. In particular, changing the Hamiltonian in the system representation drastically reduces the error as well as the error bound of our approach by several orders of magnitude. Moreover, the error corresponding to an optimal choice is on a level that is almost identical to the one of the unstructured approach from \cite{MoeRS11}.

\section{Conclusion}\label{sec:conclusion}

In this paper we propose a new structure-preserving model reduction method for linear port-Hamiltonian systems of differential-algebraic equations.
The approach is based on balancing the solutions of two generalized algebraic Riccati equations and subsequent truncation.
The structure preservation is ensured by a proper choice of the weighting matrices of the Riccati equations which yields as a side product a passive LQG-like controller which ensures that the closed-loop system is regular, impulse-free, and asymptotically stable.
Similarly as in classical LQG balanced truncation, the approximation error of the reduced-order model obtained by the new method can be estimated a priori by an error bound in the gap metric.
Moreover, we show how exploiting the freedom in choosing the Hamiltonian may lead to a faster decay of this error bound.
This step involves determining an extremal solution of a Kalman--Yakubovich--Popov linear matrix inequality.
The theoretical findings are illustrated by means of two numerical examples: a transport network and a constrained mass-spring-damper system.

Interesting future research directions include the extension to nonlinear and to infinite-dimensional port-Hamiltonian systems.
Furthermore, the focus of this paper is on the theory, whereas developing an efficient and robust numerical implementation is subject to future investigations.

\section*{Acknowledgements}

We would like to thank Tatjana Stykel for making the code from \cite{MoeRS11} available to us. 
We thank the Deutsche Forschungsgemeinschaft (DFG, German
 Research Foundation) for their support
within the project B03 in the TRR 154 ``Mathematical Modelling, Simulation and Optimization using the Example of Gas Networks'' (Project-ID 239904186). 

\appendix

\section{Theoretical Rationale behind Algorithm~\ref{alg:MOR}}\label{sec:algorithm}

In this section, we demonstrate that Algorithm~\ref{alg:MOR} indeed computes a reduced-order model as in Theorem~\ref{thm:rom_is_ph}.
To this end, we first observe that $S_\ell^\top$ and $T_\ell$ may be obtained by truncating the transformation matrices
\begin{equation}
    \label{eq:STinv}
    S = 
    \begin{bmatrix}
        \widetilde{\Sigma}_1^{-\frac12}\widetilde{U}_1^\top \LPc^\top E^+\\
        \widetilde{V}_2^\top \LPf^\top (E^+)^\top E^\top Q E^+\\
        U_2^\top
    \end{bmatrix}
    ,\quad T^{-1} = 
    \begin{bmatrix}
        E^+\LPf \widetilde{V}_1\widetilde{\Sigma}_1^{-\frac12} & E^+\LPf \widetilde{V}_2 & V_2
    \end{bmatrix}
    .
\end{equation}
In the following, we show that these transformation matrices applied to the port-Hamiltonian full-order model \eqref{eq:pHEQ} lead indeed to a balanced realization as described in Theorem~\ref{thm:bal_real}.
First, we show that the transformed system is semi-explicit by considering
\begin{align*}
    SET^{-1} &= 
    \begin{bmatrix}
        \Etilde_{11} & \Etilde_{12} & 0\\
        \Etilde_{21} & \Etilde_{22} & 0\\
        0 & 0 & 0
    \end{bmatrix}
    ,
\end{align*}
where the single blocks are given by
\begin{align*}
    \Etilde_{11} &= \widetilde{\Sigma}_1^{-\frac12}\widetilde{U}_1^\top \LPc^\top E^+EE^+\LPf \widetilde{V}_1\widetilde{\Sigma}_1^{-\frac12} = \widetilde{\Sigma}_1^{-\frac12}\widetilde{U}_1^\top \LPc^\top E^+\LPf \widetilde{V}_1\widetilde{\Sigma}_1^{-\frac12}\\
    &= \widetilde{\Sigma}_1^{-\frac12}\widetilde{U}_1^\top \widetilde{U}_1\widetilde{\Sigma}_1\widetilde{V}_1^\top \widetilde{V}_1\widetilde{\Sigma}_1^{-\frac12} = I_r,\\
    \Etilde_{12} &= \widetilde{\Sigma}_1^{-\frac12}\widetilde{U}_1^\top \LPc^\top E^+\LPf \widetilde{V}_2 = 0,\\
    \Etilde_{21} &= \widetilde{V}_2^\top \LPf^\top (E^+)^\top E^\top Q E^+\LPf \widetilde{V}_1\widetilde{\Sigma}_1^{-\frac12}\\
    &= \widetilde{V}_2^\top \LPf^\top (E^+)^\top E^\top Q E^+\LPf \widetilde{V}_1\widetilde{\Sigma}_1\widetilde{U}_1^\top\widetilde{U}_1\widetilde{\Sigma}_1^{-\frac32}\\
    &= \widetilde{V}_2^\top \LPf^\top (E^+)^\top E^\top Q E^+\LPf \LPf^\top (E^+)^\top \LPc\widetilde{U}_1\widetilde{\Sigma}_1^{-\frac32}\\
    &= \widetilde{V}_2^\top \LPf^\top (E^+)^\top Q^\top E E^+E\Pf^\top (E^+)^\top \LPc\widetilde{U}_1\widetilde{\Sigma}_1^{-\frac32}\\
    &= \widetilde{V}_2^\top \LPf^\top (E^+)^\top Q^\top \Pf E^\top (E^+)^\top \LPc\widetilde{U}_1\widetilde{\Sigma}_1^{-\frac32} = \widetilde{V}_2^\top \LPf^\top (E^+)^\top \LPc\widetilde{U}_1\widetilde{\Sigma}_1^{-\frac32}\\
    &= 0,\\
    \Etilde_{22} &= \widetilde{V}_2^\top \LPf^\top (E^+)^\top E^\top Q E^+EE^+\LPf \widetilde{V}_2\\
    &= \widetilde{V}_2^\top \LPf^\top (E^+)^\top E^\top Q V_1\Sigma_1^{-\frac12}\Sigma_1^{-\frac12}U_1^\top EE^+\LPf \widetilde{V}_2.
\end{align*}
For $E_{22}$ we first observe that $\LPf$ is an $n\times r$ matrix, since $\Pf = Q^{-\top}$ is invertible and since $\rank(E) = r$.
Thus, the matrix $\LPf^\top (E^+)^\top E^\top Q V_1\Sigma_1^{-\frac12}$ is square and its inverse is given by $\Sigma_1^{-\frac12}U_1^\top EE^+\LPf$, which follows from the calculation
\begin{equation}
    \label{eq:complicatedInverse}
    \begin{aligned}
        \Sigma_1^{-\frac12}U_1^\top EE^+\LPf\LPf^\top (E^+)^\top E^\top Q V_1\Sigma_1^{-\frac12} &= \Sigma_1^{-\frac12}U_1^\top E\Pf^\top (E^+)^\top E^\top Q V_1\Sigma_1^{-\frac12}\\
        &= \Sigma_1^{-\frac12}U_1^\top \Pf E^\top Q V_1\Sigma_1^{-\frac12}\\
        &= \Sigma_1^{-\frac12}U_1^\top E V_1\Sigma_1^{-\frac12} = I_r.
    \end{aligned}
\end{equation}
Consequently, we also obtain $\Etilde_{22} = I$ and, thus, $SET^{-1}$ is in semi-explicit form.
For verifying that the transformed system is also balanced, we first observe that $S^{-\top}$ and $T$ are given by
\begin{equation}
    \label{eq:SinvtopT}
    S^{-\top} = 
    \begin{bmatrix}
        \widetilde{\Sigma}_1^{-\frac12} \widetilde{V}_1^\top \LPf^\top (E^+)^\top E^\top\\
        \widetilde{V}_2^\top \LPf^\top (E^+)^\top E^\top\\
        U_2^\top
    \end{bmatrix}
    ,\quad
    T = 
    \begin{bmatrix}
        \widetilde{\Sigma}_1^{\frac12}\widetilde{V}_1^\top\LPf^\top(E^+)^\top E^\top Q\\
        \widetilde{V}_2^\top\LPf^\top(E^+)^\top E^\top Q\\
        V_2^\top
    \end{bmatrix}
\end{equation}
which follows from the computations
\begin{align*}
    &S^{-\top}S^\top \\ 
    &= 
    \begin{bmatrix}
        \widetilde{\Sigma}_1^{-\frac12} \widetilde{V}_1^\top \LPf^\top (E^+)^\top E^\top\\
        \widetilde{V}_2^\top \LPf^\top (E^+)^\top E^\top\\
        U_2^\top
    \end{bmatrix}
    \begin{bmatrix}
        (E^+)^\top \LPc\widetilde{U}_1\widetilde{\Sigma}_1^{-\frac12} & (E^+)^\top Q^\top E E^+\LPf\widetilde{V}_2 & U_2
    \end{bmatrix}
    \\
    &= 
    \begin{bmatrix}
        I & \widetilde{\Sigma}_1^{-\frac12} \widetilde{V}_1^\top \LPf^\top (E^+)^\top Q^\top E E^+\LPf\widetilde{V}_2 & 0\\
        \widetilde{V}_2^\top \LPf^\top (E^+)^\top \LPc\widetilde{U}_1\widetilde{\Sigma}_1^{-\frac12} & \widetilde{V}_2^\top \LPf^\top (E^+)^\top Q^\top E E^+\LPf\widetilde{V}_2 & 0\\
        0 & 0 & I_{n-r}
    \end{bmatrix}
    \\
    &\stackrel{\eqref{eq:complicatedInverse}}{=}
    \begin{bmatrix}
        I & \widetilde{\Sigma}_1^{-\frac12} \widetilde{V}_1^\top \widetilde{V}_2 & 0\\
        0 & \widetilde{V}_2^\top \widetilde{V}_2 & 0\\
        0 & 0 & I_{n-r}
    \end{bmatrix}
    = I_n
    ,\\[0.2cm]
    &T^{-1}T = 
    \begin{bmatrix}
        E^+\LPf \widetilde{V}_1\widetilde{\Sigma}_1^{-\frac12} & E^+\LPf \widetilde{V}_2 & V_2
    \end{bmatrix}
    \begin{bmatrix}
        \widetilde{\Sigma}_1^{\frac12}\widetilde{V}_1^\top\LPf^\top(E^+)^\top E^\top Q\\
        \widetilde{V}_2^\top\LPf^\top(E^+)^\top E^\top Q\\
        V_2^\top
    \end{bmatrix}
    \\
    &=
    E^+\LPf \widetilde{V}_1\widetilde{V}_1^\top\LPf^\top(E^+)^\top E^\top Q+E^+\LPf \widetilde{V}_2\widetilde{V}_2^\top\LPf^\top(E^+)^\top E^\top Q+V_2V_2^\top\\ 
    &= E^+\LPf\LPf^\top(E^+)^\top E^\top Q+V_2V_2^\top = E^+E\Pf^\top (E^+)^\top E^\top Q+V_2V_2^\top\\
    &= E^+\Pf E^\top Q+V_2V_2^\top = E^+E+V_2V_2^\top = V_1V_1^\top+V_2V_2^\top = I_n.
\end{align*}
Using $S$, $T^{-1}$, $S^{-\top}$, and $T$ as specified in \eqref{eq:STinv} and \eqref{eq:SinvtopT}, the fact that the transformed system is balanced follows from
\begin{align*}
    &S^{-\top} \Pc T^{-1} = 
    \begin{bmatrix}
        \widetilde{\Sigma}_1^{-\frac12} \widetilde{V}_1^\top \LPf^\top (E^+)^\top E^\top\\
        \widetilde{V}_2^\top \LPf^\top (E^+)^\top E^\top\\
        U_2^\top
    \end{bmatrix}
    \Pc
    \begin{bmatrix}
        E^+\LPf \widetilde{V}_1\widetilde{\Sigma}_1^{-\frac12} & E^+\LPf \widetilde{V}_2 & V_2
    \end{bmatrix}
    \\
    &= 
    \begin{bmatrix}
        \widetilde{\Sigma}_1^{-\frac12} \widetilde{\Sigma}_1^2\widetilde{\Sigma}_1^{-\frac12} & \widetilde{\Sigma}_1^{-\frac12} \widetilde{\Sigma}_1^2\widetilde{V}_1^\top\widetilde{V}_2 & 0\\
        \widetilde{V}_2^\top \widetilde{V}_1\widetilde{\Sigma}_1^2\widetilde{\Sigma}_1^{-\frac12} & \widetilde{V}_2^\top \widetilde{V}_1\widetilde{\Sigma}_1^2\widetilde{V}_1^\top \widetilde{V}_2 & 0\\
        * & * & *
    \end{bmatrix}
    = 
    \begin{bmatrix}
        \widetilde{\Sigma}_1 & 0 & 0\\
        0 & 0 & 0\\
        * & * & *
    \end{bmatrix}
    ,\\[0.2cm]
    &S\Pf T^\top\\
    &= 
    \begin{bmatrix}
        \widetilde{\Sigma}_1^{-\frac12}\widetilde{U}_1^\top \LPc^\top E^+\\
        \widetilde{V}_2^\top \LPf^\top (E^+)^\top E^\top Q E^+\\
        U_2^\top
    \end{bmatrix}
    \Pf
    \begin{bmatrix}
        Q^\top EE^+\LPf\widetilde{V}_1\widetilde{\Sigma}_1^{\frac12} & Q^\top EE^+\LPf\widetilde{V}_2 & V_2
    \end{bmatrix}
    \\
    &=
    \begin{bmatrix}
        \widetilde{\Sigma}_1^{-\frac12}\widetilde{U}_1^\top \LPc^\top E^+\LPf\widetilde{V}_1\widetilde{\Sigma}_1^{\frac12} & \widetilde{\Sigma}_1^{-\frac12}\widetilde{U}_1^\top \LPc^\top E^+ \LPf\widetilde{V}_2 & *\\
        \widetilde{V}_2^\top \LPf^\top (E^+)^\top E^\top Q E^+\LPf\widetilde{V}_1\widetilde{\Sigma}_1^{\frac12} & \widetilde{V}_2^\top \LPf^\top (E^+)^\top E^\top Q E^+\LPf\widetilde{V}_2 & *\\
        0 & 0 & *
    \end{bmatrix}
    \\
    &\stackrel{\eqref{eq:complicatedInverse}}{=}
    \begin{bmatrix}
        \widetilde{\Sigma}_1 & 0 & *\\
        \widetilde{V}_2^\top\widetilde{V}_1\widetilde{\Sigma}_1^{\frac12} & \widetilde{V}_2^\top \widetilde{V}_2 & *\\
        0 & 0 & *
    \end{bmatrix}
    =
    \begin{bmatrix}
        \widetilde{\Sigma}_1 & 0 & *\\
        0 & I & *\\
        0 & 0 & *
    \end{bmatrix}
    ,
\end{align*}
where the $*$ entries are not relevant for the balancing.
Finally, the $W_\ell^\top$ matrix specified in Algorithm~\ref{alg:MOR} is obtained by truncation of $S^{-\top}$.

\section{Proof of Theorem~\ref{thm:optimalpHRealizationOfWCF}}\label{sec:proof}

\begin{proof}
    First, since the requirements of Theorem~\ref{thm:solutionsOfNonsymmetricGAREs} are fulfilled, the transformation to $(\Etilde,\Atilde,\Btilde,\Ctilde)$ as in \eqref{eq:WCFpH} is well-defined.
    Furthermore, the invertibility of $\Xmax$ is guaranteed since $Q_{11}$ is symmetric positive definite and a solution of \eqref{eq:reducedKYPLMI}.
	\begin{enumerate}[(i)]
		\item Let $\alpha\in\RR^{n-r}$ be in the kernel of $X_{22}$.
			Furthermore, define $\beta \vcentcolon= V_1^\top \alpha\in\RR^{\widetilde{r}}$ and $\gamma \vcentcolon= V_2^\top \alpha\in\RR^{n-r-\widetilde{r}}$ yielding $\alpha = V_1\beta+V_2\gamma$.
			Since $\alpha$ is in the kernel of $X_{22}$, we have
			\begin{equation*}
				V_1V_1^\top Q_{22}\alpha-V_2V_2^\top Q_{22}^\top V_1V_1^{\top}\alpha+V_2KV_2^\top\alpha-2V_2V_2^\top C_2^\top C_2V_1V_1^{\top}\alpha = 0
			\end{equation*}
			which is equivalent to
			\begin{align}
				\label{eq:V1Qalpha}
				0 &= V_1^\top Q_{22}\alpha = V_1^\top Q_{22}\left(V_1\beta+V_2\gamma\right),\\
				\label{eq:KV2alpha}
				0 &= KV_2^\top\alpha-V_2^\top Q_{22}^\top V_1V_1^{\top}\alpha-2V_2^\top C_2^\top C_2V_1V_1^{\top}\alpha \notag \\& = K\gamma-V_2^\top Q_{22}^\top V_1\beta-2V_2^\top C_2^\top C_2V_1\beta 
			\end{align}
			or in matrix notation
			\begin{equation*}
			    \begin{bmatrix}
			        V_1^\top Q_{22}V_1                                  & V_1^\top Q_{22}V_2\\
			        -V_2^\top Q_{22}^\top V_1-2V_2^\top C_2^\top C_2V_1 & K
			    \end{bmatrix}
			    \begin{bmatrix}
			        \beta\\
			        \gamma
			    \end{bmatrix}
			    =0.
			\end{equation*}
			Multiplying this equation from the left by $[\beta^\top\;\; \gamma^\top]$ yields
			\begin{equation}
				\label{eq:quadraticEquationForbetaAndgamma}
				\begin{aligned}
					0 	&= \beta^\top V_1^\top Q_{22}V_1\beta-2\gamma^\top V_2^\top C_2^\top C_2V_1\beta+\gamma^\top K\gamma\\
					&\le \frac12\beta^\top V_1^\top (Q_{22}+Q_{22}^\top)V_1\beta-2\gamma^\top V_2^\top C_2^\top C_2V_1\beta\\ &\quad+\frac12\gamma^\top (K+K^\top)\gamma+(V_1\beta+V_2\gamma)^\top C_2^\top C_2(V_1\beta+V_2\gamma) \\
					&= \beta^\top V_1^\top \left(\frac12(Q_{22}+Q_{22}^\top)+C_2^\top C_2\right)V_1\beta\\ 
					&\quad+\gamma^\top \left(\frac12(K+K^\top)+V_2^\top C_2^\top C_2V_2\right)\gamma	\\
					&= \underbrace{\beta^\top \Sigma^{-1}U_1^\top}_{=\eta^\top} U_1\Sigma V_1^\top \left(\frac12(Q_{22}+Q_{22}^\top)+C_2^\top C_2\right)V_1\Sigma U_1^\top \underbrace{U_1\Sigma^{-1}\beta}_{=\vcentcolon \eta}\\ &\quad+\gamma^\top \left(\frac12(K+K^\top)+V_2^\top C_2^\top C_2V_2\right)\gamma\\
					&= \eta^\top B_2^\top \left(\frac12(Q_{22}+Q_{22}^\top)+C_2^\top C_2\right)B_2\eta\\
					&\quad+\gamma^\top \left(\frac12(K+K^\top)+V_2^\top C_2^\top C_2V_2\right)\gamma\\
					&= \frac12\eta^\top \left(C_2B_2+B_2^\top C_2^\top+2B_2^\top C_2^\top C_2B_2\right)\eta\\
					&\quad+\frac12\gamma^\top \left(K+K^\top+2V_2^\top C_2^\top C_2V_2\right)\gamma
					\end{aligned}
			\end{equation}
			By \eqref{eq:reducedKYPLMI} and \eqref{eq:KProperty}, both terms on the right-hand side are non-positive and, thus, both have to be zero.
			Since $K+K^\top+2V_2^\top C_2^\top C_2V_2$ is negative definite by \eqref{eq:KProperty}, this implies $\gamma = 0$.
			By the first equality in \eqref{eq:quadraticEquationForbetaAndgamma}, this leads to $\beta^\top V_1^\top Q_{22}V_1\beta=0$ and, thus, to $R_QV_1\beta=0$.
			Hence, by \eqref{eq:V1Qalpha} we obtain $V_1^\top J_QV_1\beta = 0$.
			Moreover, using \eqref{eq:KV2alpha} yields
			\begin{align*}
				0 	&= -V_2^\top Q_{22}^\top V_1\beta-2V_2^\top C_2^\top C_2V_1\beta = V_2^\top J_Q V_1\beta-2V_2^\top C_2^\top B_2^\top Q_{22}V_1\beta\\
					&= V_2^\top J_Q V_1\beta-2V_2^\top C_2^\top U_1\Sigma V_1^\top Q_{22}V_1\beta = V_2^\top J_Q V_1\beta
			\end{align*}
			and, thus, altogether $V^\top J_QV_1\beta=0$.
			This is equivalent to $J_QV_1\beta=0$ which yields $Q_{22}V_1\beta=0$.
			Since $Q_{22}$ is invertible and since $V_1$ has full column rank, we have $\beta=0$ and, hence, altogether $\alpha = 0$.
			Consequently, since the kernel consists just of the zero vector, $X_{22}$ is invertible.
		\item Simple calculations yield
			\begin{align*}
				&	
				\left(
    				\begin{bmatrix}
    					\Jhat_{11} & \Jhat_{12}\\
    					\Jhat_{21} & \Jhat_{22}
    				\end{bmatrix}
    				-
    				\begin{bmatrix}
    					\Rhat_{11} & \Rhat_{12}\\
    					\Rhat_{21} & \Rhat_{22}
    				\end{bmatrix}
    				-\Btilde\Btilde^\top
    				\right)
    				\begin{bmatrix}
    					\Xmax & 0\\
    					X_{21} 					& X_{22}
    				\end{bmatrix}
    				\\
    				&=
    				\begin{bmatrix}
    					A_{11} & 0\\
    					-X_{22}^{-1}X_{21}\Xmax^{-1} \Xmax+X_{22}^{-1}X_{21} & I_{n-r}
    				\end{bmatrix}
    				= 
    				\begin{bmatrix}
    					A_{11} & 0\\
    					0 & I_{n-r}
    				\end{bmatrix}
    				= \Atilde.
			\end{align*}
			Before we show the relation $\Ctilde = \Btilde^\top \Xhat$, we need to have another look at the KYP LMI \eqref{eq:reducedKYPLMI}.
			It implies in particular that for all $x_1\in\RR^r$ and $x_2\in\Ker(B_2)=\colspan(U_2)\subseteq \RR^m$ we have
			\begin{align*}
				0 &\leq 
				\begin{bmatrix}
					x_1^\top & x_2^\top
				\end{bmatrix}
				\calW(\Xmax)
				\begin{bmatrix}
					x_1\\
					x_2
				\end{bmatrix}
				\\
				&= -x_1^\top(A_{11}^\top \Xmax+\Xmax A_{11}+2C_1^\top C_1)x_1+2x_1^\top(C_1^\top-\Xmax B_1)x_2\\
				&= -x_1^\top(A_{11}^\top \Xmax+\Xmax A_{11}+2C_1^\top C_1)x_1\\
				&\quad+2x_1^\top(Q_{11}B_1+Q_{21}^\top B_2-\Xmax B_1)x_2\\
				&= -x_1^\top(A_{11}^\top \Xmax+\Xmax A_{11}+2C_1^\top C_1)x_1+2x_1^\top(Q_{11}-\Xmax )B_1x_2.
			\end{align*}
			Especially, this can only hold for all $x_1\in\RR^r$ if $(Q_{11}-\Xmax )B_1x_2=0$.
			Thus, we showed $\Ker(B_2)=\colspan(U_2)\subseteq \Ker((Q_{11}-\Xmax )B_1)$.
			Using this, we obtain
			\begin{align*}
				&\Btilde^\top \Xhat =
				\begin{bmatrix}
					B_1^\top & B_2^\top
				\end{bmatrix}
				\begin{bmatrix}
					\Xmax & 0\\
    					X_{21} 					& X_{22}
				\end{bmatrix}
				= 
				\begin{bmatrix}
					B_1^\top \Xmax+B_2^\top X_{21} & B_2^\top X_{22}
				\end{bmatrix}
				\\
				&=
				\begin{bmatrix}
					B_1^\top \Xmax+B_2^\top V_1V_1^\top Q_{21}+B_2^\top(B_2^\top)^+B_1^\top (Q_{11}-\Xmax) & B_2^\top X_{22}
				\end{bmatrix}
				\\
				&= 
				\begin{bmatrix}
					B_1^\top \Xmax+B_2^\top Q_{21}+U_1U_1^\top B_1^\top (Q_{11}-\Xmax) & B_2^\top X_{22}
				\end{bmatrix}
				\\
				&=
				\begin{bmatrix}
					B_1^\top \Xmax+B_2^\top Q_{21}+(I_m-U_2U_2^\top)B_1^\top (Q_{11}-\Xmax) & B_2^\top X_{22}
				\end{bmatrix}
				\\
				&=
				\begin{bmatrix}
					B_1^\top \Xmax+B_2^\top Q_{21}+B_1^\top (Q_{11}-\Xmax) & B_2^\top X_{22}
				\end{bmatrix}
				\\
				&=
				\begin{bmatrix}
					B_2^\top Q_{21}+B_1^\top Q_{11} & B_2^\top V_1V_1^\top Q_{22}
				\end{bmatrix}
				\\
				&=
				\begin{bmatrix}
					C_1 & B_2^\top Q_{22}
				\end{bmatrix}
				=
				\begin{bmatrix}
					C_1 & C_2
				\end{bmatrix}
				= \Ctilde.
			\end{align*}
			It remains to show the properties of $\Jhat$ and $\Rhat$.
			The skew-symmetry of $\Jhat$ and the symmetry of $\Rhat$ follow directly from the definition.
			Finally, we show that $\Rhat\succeq 0$ which is equivalent to $2\Xhat^\top \Rhat \Xhat\succeq0$ since $\Xhat$ is invertible by (i) and by the positive definiteness of $\Xmax$.
			The matrix $2\Xhat^\top \Rhat \Xhat$ can be written as
			\begin{align*}
				2\Xhat^\top \Rhat \Xhat &= 2
				\begin{bmatrix}
    					\Xmax 	& X_{21}^\top\\
    					0							& X_{22}^\top
    				\end{bmatrix}
				\begin{bmatrix}
					\Rhat_{11} & \Rhat_{12}\\
    					\Rhat_{21} & \Rhat_{22}
				\end{bmatrix}
				\begin{bmatrix}
    				\Xmax & 0\\
    				X_{21} 					& X_{22}
    			\end{bmatrix}
    			\\
    			&= 2
    			\begin{bmatrix}
    				\Xmax 	& X_{21}^\top\\
    				0							& X_{22}^\top
    			\end{bmatrix}
    			\begin{bmatrix}
    				\Rhat_{11}\Xmax+\Rhat_{12}X_{21} & \Rhat_{12}X_{22}\\
    				\Rhat_{21}\Xmax+\Rhat_{22}X_{21} & \Rhat_{22}X_{22}
    			\end{bmatrix}
			\end{align*}	
			For ease of exposition, we continue the calculations block by block.
			The upper left block of $2\Xhat^\top \Rhat \Xhat$ reads
			\begin{align*}
				&2\Xmax(\Rhat_{11}\Xmax+\Rhat_{12}X_{21})+2X_{21}^\top(\Rhat_{21}\Xmax+\Rhat_{22}X_{21})\\
				&= 2\Xmax\Rhat_{11}\Xmax+2\Xmax\Rhat_{12}X_{21}+2X_{21}^\top\Rhat_{21}\Xmax+2X_{21}^\top\Rhat_{22}X_{21}\\
				&= -\Xmax\left(A_{11}\Xmax^{-1}+\Xmax^{-1}A_{11}^\top+2B_1B_1^\top\right)\Xmax\\
				&\quad+\Xmax\left(\Xmax^{-1}X_{21}^\top X_{22}^{-\top}-2B_1B_2^\top\right)X_{21}\\
				&\quad +X_{21}^\top \left(X_{22}^{-1}X_{21}\Xmax^{-1} -2B_2B_1^\top\right)\Xmax\\
				&\quad-X_{21}^\top\left(X_{22}^{-1}+X_{22}^{-\top}+2B_2B_2^\top\right)X_{21}\\
				&= -\Xmax A_{11}-A_{11}^\top \Xmax-2\Xmax B_1B_1^\top \Xmax+X_{21}^\top X_{22}^{-\top}X_{21}\\
				&\quad-2\Xmax B_1B_2^\top X_{21}+X_{21}^\top X_{22}^{-1}X_{21}-2X_{21}^\top B_2B_1^\top \Xmax-X_{21}^\top X_{22}^{-1}X_{21}\\
				&\quad-X_{21}^\top X_{22}^{-\top}X_{21}-2X_{21}^\top B_2B_2^\top X_{21}\\
				&= -\Xmax A_{11}-A_{11}^\top \Xmax-2\left(B_1^\top \Xmax+B_2^\top X_{21}\right)^\top\left(B_1^\top \Xmax+B_2^\top X_{21}\right)\\
				&= -\Xmax A_{11}-A_{11}^\top \Xmax-2C_1^\top C_1
			\end{align*}
			For the upper right block we obtain
			\begin{align*}
				&2\Xmax\Rhat_{12}X_{22}+ 2X_{21}^\top\Rhat_{22}X_{22}\\
				&= \Xmax\left(\Xmax^{-1}X_{21}^\top X_{22}^{-\top}-2B_1B_2^\top\right)X_{22}\\
				&\quad-X_{21}^\top\left(X_{22}^{-1}+X_{22}^{-\top}+2B_2B_2^\top\right)X_{22}\\
				&= X_{21}^\top X_{22}^{-\top}X_{22}-2\Xmax B_1B_2^\top X_{22}-X_{21}^\top-X_{21}^\top X_{22}^{-\top}X_{22}\\
				&\quad-2X_{21}^\top B_2B_2^\top X_{22}\\
				&= -2\left(\Xmax B_1+X_{21}^\top B_2\right)B_2^\top X_{22}-X_{21}^\top = -2C_1^\top C_2-X_{21}^\top.
			\end{align*}
			Furthermore, the bottom left block is the transpose of the upper right block and the bottom right block is given by
			\begin{align*}
				2X_{22}^\top\Rhat_{22}X_{22} &= -X_{22}^\top\left(X_{22}^{-1}+X_{22}^{-\top}+2B_2B_2^\top\right)X_{22}\\
				&= -X_{22}^\top-X_{22}-2X_{22}^\top B_2B_2^\top X_{22} = -X_{22}^\top-X_{22}-2C_2^\top C_2.
			\end{align*}
			Thus, in total we obtain
			\begin{equation*}
				2\Xhat^\top \Rhat \Xhat = -
				\begin{bmatrix}
					\Xmax A_{11}+A_{11}^\top \Xmax+2C_1^\top C_1 	& 2C_1^\top C_2+X_{21}^\top\\[0.1cm]
					2C_2^\top C_1+X_{21}				            & X_{22}^\top+X_{22}+2C_2^\top C_2
				\end{bmatrix}
				.
			\end{equation*}
			In the following we conclude the proof by showing that $-2\Xhat^\top \Rhat \Xhat$ is negative semidefinite.
			To this end, let $x_1\in\RR^r$ and $x_2\in\RR^{n-r}$ be arbitrary.
			Furthermore, let $\beta \in\RR^{\widetilde{r}}$ and $\gamma\in\RR^{n-r-\widetilde{r}}$ be such that $x_2 = V_1\beta+V_2\gamma$.
			Then, there holds 
			\begin{align*}
				&
				\begin{bmatrix}
					x_1^\top & x_2^\top
				\end{bmatrix}
				\begin{bmatrix}
					\Xmax A_{11}+A_{11}^\top \Xmax+2C_1^\top C_1 	& 2C_1^\top C_2+X_{21}^\top\\
					2C_2^\top C_1+X_{21} 															& X_{22}^\top+X_{22}+2C_2^\top C_2
				\end{bmatrix}
				\begin{bmatrix}
					x_1\\
					x_2
				\end{bmatrix}
				\\
				&= x_1^\top\left(\Xmax A_{11}+A_{11}^\top \Xmax+2C_1^\top C_1\right)x_1+2x_2^\top\left(2C_2^\top C_1+X_{21}\right)x_1\\
				&\quad+x_2^\top\left(X_{22}^\top+X_{22}+2C_2^\top C_2\right)x_2\\
				&= x_1^\top\left(\Xmax A_{11}+A_{11}^\top \Xmax+2C_1^\top C_1\right)x_1\\
				&\quad +2(V_1\beta+V_2\gamma)^\top\left(2C_2^\top C_1+X_{21}\right)x_1\\
				&\quad +(V_1\beta+V_2\gamma)^\top\left(2X_{22}+2X_{22}^\top B_2B_2^\top X_{22}\right)(V_1\beta+V_2\gamma)\\
				&= x_1^\top\left(\Xmax A_{11}+A_{11}^\top \Xmax+2C_1^\top C_1\right)x_1\\
				&\quad +2\beta^\top V_1^\top\left(2C_2^\top C_1+Q_{21}+(B_2^\top)^+B_1^\top (Q_{11}-\Xmax)\right)x_1\\
				&\quad +2(V_1\beta+V_2\gamma)^\top X_{22}(V_1\beta+V_2\gamma)+2\lVert B_2^\top X_{22}(V_1\beta+V_2\gamma)\rVert^2\\
				&= x_1^\top\left(\Xmax A_{11}+A_{11}^\top \Xmax+2C_1^\top C_1\right)x_1\\
				&\quad +2\beta^\top V_1^\top\left(2C_2^\top C_1+Q_{21}+(B_2^\top)^+B_1^\top (Q_{11}-\Xmax)\right)x_1\\
				&\quad +2\left(\beta^\top V_1^\top Q_{22}V_1\beta-2\gamma^\top V_2^\top C_2^\top C_2V_1\beta+\gamma^\top K\gamma\right)\\
				&\quad+2\lVert B_2^\top V_1V_1^\top Q_{22}(V_1\beta+V_2\gamma)\rVert^2\\
				&= x_1^\top\left(\Xmax A_{11}+A_{11}^\top \Xmax+2C_1^\top C_1\right)x_1\\
				&\quad +2\beta^\top V_1^\top\left(2C_2^\top C_1+Q_{21}+(B_2^\top)^+B_1^\top (Q_{11}-\Xmax)\right)x_1\\
				&\quad +2\left(\beta^\top V_1^\top Q_{22}V_1\beta-2\gamma^\top V_2^\top C_2^\top C_2V_1\beta+\gamma^\top K\gamma\right)\\
				&\quad+2\lVert B_2^\top Q_{22}V_1\beta+B_2^\top Q_{22}V_2\gamma\rVert^2\\
				&= x_1^\top\left(\Xmax A_{11}+A_{11}^\top \Xmax+2C_1^\top C_1\right)x_1\\
				&\quad +2\beta^\top V_1^\top\left(2C_2^\top C_1+Q_{21}+(B_2^\top)^+B_1^\top (Q_{11}-\Xmax)\right)x_1\\
				&\quad +2\left(\beta^\top V_1^\top Q_{22}V_1\beta-2\gamma^\top V_2^\top C_2^\top C_2V_1\beta+\gamma^\top K\gamma\right)\\
				&\quad+2\lVert C_2V_1\beta+C_2V_2\gamma\rVert^2\\
				&= x_1^\top\left(\Xmax A_{11}+A_{11}^\top \Xmax+2C_1^\top C_1\right)x_1\\
				&\quad +2\beta^\top V_1^\top\left(2C_2^\top C_1+Q_{21}+(B_2^\top)^+B_1^\top (Q_{11}-\Xmax)\right)x_1\\
				&\quad +2\left(\beta^\top V_1^\top Q_{22}V_1\beta+\gamma^\top K\gamma+\beta^\top V_1^\top C_2^\top C_2V_1\beta+\gamma^\top V_2^\top C_2^\top C_2V_2\gamma\right)\\
				&= x_1^\top\left(\Xmax A_{11}+A_{11}^\top \Xmax+2C_1^\top C_1\right)x_1\\
				&\quad +2\beta^\top V_1^\top\left(2C_2^\top C_1+Q_{21}+(B_2^\top)^+B_1^\top (Q_{11}-\Xmax)\right)x_1\\
				&\quad +\beta^\top V_1^\top\left(Q_{22}+Q_{22}^\top+2C_2^\top C_2\right)V_1\beta+\gamma^\top (K+K^\top+2V_2^\top C_2^\top C_2V_2)\gamma\\
				&\le x_1^\top\left(\Xmax A_{11}+A_{11}^\top \Xmax+2C_1^\top C_1\right)x_1\\
				&\quad +2\beta^\top V_1^\top\left(2C_2^\top C_1+Q_{21}+(B_2^\top)^+B_1^\top (Q_{11}-\Xmax)\right)x_1\\
				&\quad+\beta^\top V_1^\top\left(Q_{22}+Q_{22}^\top+2C_2^\top C_2\right)V_1\beta\\
				&= x_1^\top\left(\Xmax A_{11}+A_{11}^\top \Xmax+2C_1^\top C_1\right)x_1\\
				&\quad +2\beta^\top \Sigma^{-1}U_1^\top U_1\Sigma V_1^\top\left(2C_2^\top C_1+Q_{21}+(B_2^\top)^+B_1^\top (Q_{11}-\Xmax)\right)x_1\\
				&\quad+\underbrace{\beta^\top\Sigma^{-1}U_1^\top}_{=\eta^\top} U_1\Sigma V_1^\top \left(Q_{22}+Q_{22}^\top+2C_2^\top C_2\right)V_1\Sigma U_1^\top \underbrace{U_1\Sigma^{-1}\beta}_{=\vcentcolon \eta}\\
				&= x_1^\top\left(\Xmax A_{11}+A_{11}^\top \Xmax+2C_1^\top C_1\right)x_1\\
				&\quad+2\eta^\top B_2^\top\left(2C_2^\top C_1+Q_{21}+(B_2^\top)^+B_1^\top (Q_{11}-\Xmax)\right)x_1\\
				&\quad+\eta^\top B_2^\top\left(Q_{22}+Q_{22}^\top+2C_2^\top C_2\right)B_2\eta\\
				&= x_1^\top\left(\Xmax A_{11}+A_{11}^\top \Xmax+2C_1^\top C_1\right)x_1\\
				&\quad+2\eta^\top\left(2B_2^\top C_2^\top C_1+B_2^\top Q_{21}+B_1^\top (Q_{11}-\Xmax)\right)x_1\\
				&\quad+\eta^\top \left(C_2B_2+B_2^\top C_2^\top+2B_2^\top C_2^\top C_2B_2\right)\eta\\
				&= x_1^\top\left(\Xmax A_{11}+A_{11}^\top \Xmax+2C_1^\top C_1\right)x_1\\
				&\quad+2\eta^\top\left(2B_2^\top C_2^\top C_1+C_1-B_1^\top \Xmax\right)x_1\\
				&\quad+\eta^\top \left(C_2B_2+B_2^\top C_2^\top+2B_2^\top C_2^\top C_2B_2\right)\eta\\
				&=
				\begin{bmatrix}
					x_1^\top & -\eta^\top
				\end{bmatrix}	
				\calW(\Xmax)
				\begin{bmatrix}
					x_1\\
					-\eta
				\end{bmatrix}
				\le 0.		
			\end{align*}
			Consequently, $\Rhat$ is positive semidefinite which concludes the proof.
	\end{enumerate}
\end{proof}

\bibliographystyle{siam}
\bibliography{references}

\end{document}